\documentclass[11pt,oneside]{article}  
\usepackage{subcaption}
\usepackage{amsmath,amsthm,graphicx,amssymb,enumerate,color,bbm,rotating,tikz,authblk,romannum,array,enumitem,dsfont,booktabs,multirow,mathtools,float}
\usepackage{mathtools}
\usepackage{algorithm}
\usepackage{algpseudocode}
\usepackage{caption} % Add the caption package
\usepackage{pgfplots}
\usepackage{diagbox}
\usepackage[pagewise]{lineno}
\usepackage[mathscr]{euscript}
\usepackage[top=1in, bottom=1in, left=1in, right=1in]{geometry}
\usetikzlibrary{chains,shapes.multipart,shapes,calc,fit}
\usepackage[percent]{overpic}
\usetikzlibrary{automata,positioning}
\usepackage[bottom]{footmisc}
\usepackage[hidelinks=true]{hyperref}
\usepackage{setspace}
\usepackage[T1]{fontenc}
\usepackage[utf8]{inputenc}
\usepackage[para,online,flushleft]{threeparttable}
\usepackage[all]{xy}
\usepackage{natbib}
 \bibpunct[, ]{(}{)}{,}{a}{}{,}%

\DeclareMathAlphabet{\mathpzc}{OT1}{pzc}{m}{it}
\DeclarePairedDelimiter{\norm}{\lVert}{\rVert}

\newtheoremstyle{ModifiedStyle}
	{\topsep} % Space above
	{3pt} % Space below
	{} % Body font
	{} % Indent amount
	{\bfseries} % Theorem head font
	{.} % Punctuation after theorem head
	{.5em} % Space after theorem head
	{} % Theorem head spec (can be left empty, meaning `normal')
\theoremstyle{ModifiedStyle}

\newtheorem{proposition}{Proposition}

\newtheorem{definition}{Definition}
\newtheorem{remark}{Remark}

\newtheorem{to-do}{To-Do}

\newcommand{\E}{\mathbb{E}}
\newcommand{\Var}{{\rm Var}}

\newcommand{\BSmall}[1]{\mkern-1.7mu\raisebox{-1.2pt}{\scalebox{0.9}{$\scriptscriptstyle B$}}}
\newcommand{\NSmall}[1]{\mkern-1.7mu\raisebox{-1.2pt}{\scalebox{0.9}{$\scriptscriptstyle N$}}}

\newcommand{\nquad}{\kern-1em}

\makeatletter
\newcommand{\Rom}[1]{\expandafter\@slowromancap\romannumeral #1@}
\newcommand{\Biggg}{\bBigg@{2.5}}
\newcommand{\vast}{\bBigg@{3}}
\newcommand{\Vast}{\bBigg@{3.5}}
\newcommand{\massive}{\bBigg@{4.5}}
\newcommand{\Massive}{\bBigg@{6}}
\newcommand{\thickhline}{%
	\noalign {\ifnum 0=`}\fi \hrule height 1pt
	\futurelet \reserved@a \@xhline}
\newcolumntype{"}{@{\hskip\tabcolsep\vrule width 1pt\hskip\tabcolsep}}

\newcommand{\PreserveBackslash}[1]{\let\temp=\\#1\let\\=\temp}
\newcolumntype{C}[1]{>{\PreserveBackslash\centering}p{#1}}
\newcolumntype{R}[1]{>{\PreserveBackslash\raggedleft}p{#1}}
\newcolumntype{L}[1]{>{\PreserveBackslash\raggedright}p{#1}}
\makeatother

\newcommand{\Subroutine}[1]{\textproc{#1}}

\title{{\bf Analysis and Improvement of Eviction Enforcement}}
\author[1]{Barış Ata}
\author[2]{Yuwei Zhou}
\affil[1]{Booth School of Business, The University of Chicago}
\affil[2]{Kelley School of Business, Indiana University}
\date{\vspace{-7mm} \today}

\makeatletter
\let\LN@align\align
\let\LN@endalign\endalign
\renewcommand{\align}{\linenomath\LN@align}
\renewcommand{\endalign}{\LN@endalign\endlinenomath}
\let\LN@gather\gather
\let\LN@endgather\endgather
\renewcommand{\gather}{\linenomath\LN@gather}
\renewcommand{\endgather}{\LN@endgather\endlinenomath}
\makeatother

\begin{document}

\maketitle

\pagenumbering{arabic}

\allowdisplaybreaks

\singlespacing
%%%%%% Abstract %%%%%%
\begin{abstract}
 Each year, nearly 13,000 eviction orders are issued in Cook County, Illinois. While most of these orders have an enforcement deadline, a portion does not. The Cook County Sheriff’s Office (CCSO) is responsible for enforcing these orders, which involves selecting the orders to prioritize and planning daily enforcement routes. This task presents a challenge: balancing ``equity'' (i.e., prioritizing orders that have been waiting longer) with ``efficiency'' (i.e., maximizing the number of orders served). Although the current CCSO policy is highly efficient, a significant fraction of eviction orders miss their deadline.

Motivated by the CCSO’s operations, we study a model of eviction enforcement planning and propose a policy that dynamically prioritizes orders based on their type (deadline or no deadline), location, and waiting time. Our approach employs a budgeted prize-collecting vehicle routing problem (VRP) for daily planning, where the ``prizes'' are determined by solving a stochastic control problem. This stochastic control problem, which relies on the VRP for determining feasible actions at each decision point, is high-dimensional due to its spatial nature, leading to the curse of dimensionality. We overcome this challenge by building on recent advances in high-dimensional stochastic control using deep neural networks.

We compare the performance of our proposed policy with two practical benchmark policies, including one that mimics the current CCSO policy, using data from CCSO. Similar to the CCSO policy, our proposed policy leads to efficient resource utilization, but it also reduces the percentage of orders that miss their deadline by 72.38\% without degrading the overall service effort for either type of orders. In a counterfactual study, we show that increasing the service capacity or extending the enforcement deadline further reduces the fraction of orders missing their deadline. 

  \vspace{12pt}

\end{abstract}

\doublespacing
%%%%%% Main Text %%%%%%

\section{Introduction}
\label{sec:Introduction}

In the United states, there are about $3.6$ million eviction filings every year (\citet{Gromis2022}). Eviction orders are issued after a legal process by which a landlord removes his tenant from the rental property usually due to a breach of rental contract such as not paying rent or violating other terms of the lease. Once the presiding judge orders an eviction, the eviction order is forwarded for enforcement, typically to the sheriff's department of the county. In Cook County, Illinois, the second largest county in the United States, Cook County Sheriff's Office (CCSO) is responsible for enforcing eviction orders. CCSO receives about $13,000$ eviction orders per year, resulting in a large and complex operation. CCSO's eviction enforcement operations motivates our work, and it is the focus of our computational study in Section \ref{sec:Simulation}.

The eviction enforcement operations primarily involve two key operational questions: Given the outstanding eviction orders, on each day \textit{i}) which orders should each eviction team of sheriff's deputies enforce? and \textit{ii}) what is the best route for each eviction team? Eviction orders arrive to the system randomly over time. In Cook County, about 83\% of them have a deadline by which they must be enforced, whereas the rest have no deadline. Roughly speaking, eviction orders against individuals usually have a deadline whereas eviction orders against property do not have a deadline \citep{PeterPon2023}. The deadline for serving the orders is typically 120 calendar days. Although there are provisions to extend this deadline, such exceptions require additional processing time and resources. Without such exceptions, it is illegal to serve an eviction order once its deadline passes. In our analysis, we focus attention on the initial 120-day period and refer to eviction orders that are not executed or canceled during that period as missing their deadline. In what follows, we calibrate our model using the CCSO data to reflect this. As such, the performance metrics we report below should be interpreted with our focus on the 120-day period in mind. On each day, a system manager reviews the outstanding eviction orders and their deadlines on each day to decide which eviction orders each team should enforce as well as which routes they should follow. The eviction orders that are not enforced remain in the system, moving closer to their deadline. In addition, some eviction orders may be canceled as they wait in the system to be enforced.

The spatial and uncertain nature of eviction orders' arrivals, and the dynamic nature of the system makes the scheduling problem challenging.  On the one hand, it may seem reasonable to follow the first-come-first-served (FCFS) rule to schedule the enforcement of eviction orders for the purpose of equity, that is, prioritizing the eviction orders that have been waiting the longest. However, this may not be efficient if those orders are far away from each other, because that would increase the travel time for eviction teams. On the other hand, the system manager can focus on efficiency by prioritizing nearby orders or serves a region once the number of orders exceeds a certain threshold. 
%While such scheduling rules can lower the total travel time, they can cause some eviction orders to miss their deadline. 
%The system manager seeks to balance the trade-off between travel times and the delays in enforcing eviction orders.
Currently, the CCSO is highly efficient in resource utilization and generally treats orders with and without deadlines similarly. However, 18.83\% of the eviction orders miss their deadlines. 

Designing the routes of the eviction teams on a given day can be viewed as a vehicle routing problem (VRP), but our problem has additional features. First, each team can only work for a certain number of hours, implying an upper bound on the total travel time of each eviction team per day. More importantly, our formulation involves a sequence of VRPs, one for each day; and these VRPs interact with each other through the evolution of the system state, that is, the outstanding eviction orders. To be specific, a certain ``shadow price'' of the system state should be incorporated into each VRP to account for the impact of daily decisions on the future evolution of the system state in order to avoid those decisions from being myopic. To address these concerns formally, we assign a prize for each eviction order that are derived from a stochastic control problem. Given these prizes, the daily decisions can be viewed as solving a Team Orienteering Problem (TOP) or a Budgeted Prize-Collecting VRP whose objective is to maximize the total prize collected. A discussion of the literature related to VRP and its variants is provided below. As mentioned earlier, these prizes (or the shadow prices) are derived from a stochastic control problem, i.e., the Brownian control problem formulated in Section \ref{sec:BackgroundAndModel}, using the gradient of its value function. In turn, the Brownian control problem incorporates the budgeted prize-collecting VRPs for determining the set of feasible actions. 

To be more specific, the Brownian control problem we formulate is one of drift rate control for Reflected Brownian Motion (RBM). However, due to the spatial nature of our application, its state descriptor is high dimensional. For example, the computational study based on CCSO's eviction enforcement operations leads to a 24 dimensional problem; see Section \ref{sec:Simulation}. Due to the curse of dimensionality, such problems are considered intractable computationally at least by historical standards. Nonetheless, we develop a simulation-based computational method that relies on the deep neural network technology to solve our problem in high dimensions. To do so, we first express the problem analytically by considering the associated Hamilton–Jacobi–Bellman (HJB) equation; see \citet{FlemingSoner2006}. The HJB equation is a semilinear elliptic partial differential equation (PDE); see \citet{GilbargTrudinger2001}. To solve the HJB equation, we build on the seminal work of \citet{Han2018}, who develop a method for solving semilinear parabolic PDEs.

Using our computational method, we propose a policy in Section \ref{sec:ComputationalMethod} and test its performance against two natural benchmarks; see Section \ref{sec:Simulation}. One of these policies is designed to replicate the CCSO policy in effect during the data collection period from 2014 to 2019. Our proposed policy outperforms the benchmarks by reducing the percentage of the eviction orders that miss their deadline by 72.38\%. It also achieves a significant reduction in terms of the number of pending eviction orders compared with the benchmarks we consider. The average number of eviction orders served by the proposed policy is close to that observed in the CCSO dataset and to that served by a benchmark policy mimics the current CCSO practice, whereas the number of eviction orders served daily by the FCFS benchmark is significantly lower.

\paragraph{The remainder of this paper.} Section \ref{sec:Literature} includes the relevant literature. Section \ref{sec:BackgroundAndModel} presents our model. In Section \ref{sec:HJB}, we present the associated HJB equation and identify an equivalent stochastic differential equation (SDE) that is used to identify the proposed policy. Section \ref{sec:ComputationalMethod} presents the computational method that we use the solve the control problem. Section \ref{sec:BenchmarkPolicies} presents our proposed policy and the benchmark policies. Section \ref{sec:Simulation} demonstrates the effectiveness of our proposed policy through a simulation study that is calibrated using the CCSO data as well as the counterfactual analysis with increased service capacity and the extended enforcement deadline.
%Finally, Section \ref{sec:Conclusion} provides the concluding remarks. 
Appendices present the auxiliary algorithms we use, a detailed discussion of our proposed policy and benchmark policies, the further details of the dataset, a discussion on the calibration for one of the benchmark policies, details on the computation of performance measures through simulation, and implementation details of our computational study. 

\section{Literature Review} 
\label{sec:Literature}

Our work is related to four different streams of literature: \textit{i}) study of evictions in the sociology the literature, \textit{ii}) vehicle routing problems in combinatorial optimization, \textit{iii}) drift rate control problems in applied probability; and \textit{iv}) computational methods for solving high-dimensional PDEs in applied mathematics and physics. 

\paragraph{Evictions in the sociology literature.}
Researchers have explored different aspects of evictions in the sociology literature. For example, \citet{Desmond2022} studies consequences of eviction, especially for the low-income communities and propose policy interventions to remedy its negative consequences. \citet{VasquezVera2017} documents effect of threat of eviction on health through a comprehensive survey of the literature and further studies them. The authors conclude that the affected population have worse physical and mental health and the threat of eviction need to be addressed urgently. Several authors explore issues related to evictions through quantitative analysis. For example, \citet{Desmond2012} studies prevalence of and ramifications of evictions of low-income urban neighborhoods. Similarly, \citet{Gromis2022} estimate eviction prevalence in the United States. Our work focuses on making eviction enforcement more efficient. As such, its implementation can make the negative societal consequences of evictions studied in the literature worse. This points to the need for further policy interventions and research studying them.

\paragraph{Vehicle Routing Problem (VRP).}

As a generalization of the TSP \citep{Flood1956}, the VRP was first proposed by \citet{Dantzig1959}. The classical form of the VRP is both static and deterministic, meaning that all the relevant information is known before planning begins; see \citet{TothVigo2002} for a review of VRP and its many variants, and see  \citet{Pessoa2020} and \citet{Yang2024} for some recent state-of-the-art methods to solve such problems. In practice, the static and deterministic assumptions rarely hold. With uncertain parameters, such as stochastic demands (\citet{Bertsimas1992}, \citet{Secomandi2009}, and \citet{Goodson2013}) and stochastic travel times (\citet{Laporte1992}, \citet{Toriello2014}, \citet{Tas2013}, and \citet{Yang2023}), the VRP can be stochastic; see \citet{Gendreau1996} for a review of stochastic VRP. Additionally, when information is gradually revealed through the execution of the routing plan, the VRP becomes dynamic; see \citet{Pillac2013} for a review of dynamic VRP and \citet{Flatberg2007} for examples of dynamic and stochastic VRPs in practice. 
Our focus is on daily decision making with solutions to a sequence of VRPs, determining vehicle routes to serve a subset of current pending eviction orders while also planning for ``future information'', such as new arrivals and cancellations. Each VRP is static in the sense that the system manager has complete information for daily planning, while the sequence of VRPs is interconnects through the dynamic evolution of the system state with daily updates. 

Within the static VRP literature, the most relevant problem is the Team Orienteering Problem, where the objective is to maximize the total prize collected subject to a maximum duration constraint for each route. This problem is also referred to as the budgeted prize-collecting VRP, which we use because it emphasizes the connection to the classical VRP and related logistics applications. Solution methods for such problems are widely studied in the literature, see for example \citet{ElHajj2016}, \citet{Feillet2005b}, and \citet{Baldacci2023}. When only a single vehicle is available, our problem reduces to a variant of the TSP, known as the orienteering problem or the budgeted prize-collecting TSP. Solving the TSP, VRP, or their variants exactly is computationally challenging, especially when the number of customers is large. To the best of our knowledge, among the heuristics and approximation algorithms, \citet{Hegde2015} is the first to present a nearly-linear time algorithm for solving the prize-collecting Steiner tree (PCST) problem. We build on their algorithm, using the approach of \citet{GoemansWilliamson1995}, to propose an efficient solution to the budgeted prize-collecting VRP we consider; see Section \ref{subsec:ApproximateAuxiliaryFunction} and Appendix \ref{app:TrainingFFunction_Appendix} for details.

Researchers have worked on different aspects of dispatch and routing decisions in stochastic and dynamic VRP; see for example, \citet{Azi2012}, \citet{Liu2021}, \citet{Carlsson2024}. Within that literature, a stream of work that is related to ours is the dynamic multiperiod problem. In that problem formulation, dispatch decisions are made for fixed intervals and the information needed for route planning, such as customer arrivals and service time windows, is revealed sequentially over time during the planning horizon (\citet{Angelelli2007a}, \citet{Angelelli2007b}, \citet{Wen2010}, \citet{Keskin2023}). Due to the complexity of the problem, many authors resort to approximate methods, such as approximate dynamic programming, in order to propose effective dynamic dispatch and routing policies. For example, \citet{Klapp2018} focuses on a single-vehicle problem and proposes two approaches for obtaining dynamic policies: a rollout of the apriori policy and an approximate linear program to estimate the cost-to-go function. \citet{Ulmer2019} also considers a single-vehicle routing problem, but combines offline value approximation with online rollout algorithms, resulting in an effective policy and computationally tractable solution approach.
\citet{Voccia2019} and \citet{LiuLuo2023} study dynamic dispatch and routing problems for on-demand delivery with random demand that arrives over a fixed number of periods.
\citet{Voccia2019} presents a Markov decision process model and applies neighborhood search heuristics as a solution subroutine. 
\citet{LiuLuo2023} proposes a structured approximation framework to estimate the value function with a decomposed dispatch and routing policy. The estimation is then embedded into the dynamic program to derive solutions in a rollout fashion.
\citet{AlbaredaSambola2014} considers multiple vehicles, where customers arrive over a discrete time horizon and must be fulfilled within a time window that comprises several time periods. They assume that the probabilities of customer requests in the future and their associated time windows are known. In that setting, they propose a solution to the problem using a prize-collecting VRP model, where prize of each pending request is based on its urgency (measured by the remaining time within its time window) and the probability that neighboring customers will request service within the same time window. 
Both our modeling assumptions and the solution method differ from the aforementioned studies. More specifically, prizes to pending eviction orders based on the cost-to-go function derived from a Brownian control problem formulation (see detailed review in next paragraph) and determine the subset of orders to serve, as well as the route, by solving the budgeted prize-collecting VRP.

\paragraph{Drift rate control problems.}
Brownian control problems are first introduced in \citet{Harrison1988} and later extended in \citet{Harrison2000, Harrison2003}. Starting with \citet{HarrisonWein1989, HarrisonWein1990} and \citet{Wein1990a, Wein1990b}, many authors used them to model manufacturing and service systems. For example, \citet{Reiman1999} considers the inventory-routing problem and model it as a queueing control problem. For shipping to the retailers, the authors consider either a pre-determined TSP tour or direct shipping. By analyzing the problem under heavy traffic assumption, they propose dynamic delivery allocation policy as well as the policy on whether the vehicle should be busy or idle. \citet{Markowitz2000} considers a queueing system scheduling problem that is a stochastic version of the economic lot scheduling problem and proposes dynamic cyclic policy that minimizes the long-run average cost. One class of Brownian models focus on drift rate control problems for RBM. \citet{Ata2005b} considers a drift rate control problem on a bounded interval under a general cost of control but no state costs. The authors solve the problem in closed form and discuss its application to power control problem in wireless communication. \citet{GhoshWeerasinghe2007, GhoshWeerasinghe2010} extend \citet{Ata2005b} in several directions. The authors incorporate state costs, abandonment and optimally choose the interval where the process lives.

Drift rate control problems have been used to study a variety of applications. \citet{Ata2019} studies a dynamic staffing problem for choosing the number of volunteer gleaners, who sign up but may not show up, for harvesting leftover crops donated by farmers for the purpose of feeding food-insecure individuals. Formulating the problem as one of drift rate control, the authors derive an effective policy for dynamic decision making in closed form. \citet{Ata2024a} use a drift rate control model to study volunteer engagement; see also \citet{Ata2024b} for a similar model.
\citet{CelikMaglaras2008} studies a problem of dynamic pricing, lead time quotation, outsourcing and scheduling for a make-to-order production system. They approximate the problem by a drift-rate control problem and solve for the optimal pricing policy in closed form, building on \citet{Ata2005b}; also see \citet{AtaBarjesteh2023} for a related study. Similarly, \citet{Alwan2024} consider dynamic pricing and dispatch control decisions for a ride hailing system using a queueing network model. Under the complete resource pooling assumption, the authors approximate their pricing problem by a one-dimensional drift rate control problem. A common feature of these papers is that the underlying state process is one dimensional. To the best of our knowledge, there has not been any papers studying drift control of RBM in high dimensions except for \citet{Ata2023}. Our paper contributes to this literature by solving the drift rate control of RBM that we use to model the eviction enforcement operations.

\paragraph{Solving high-dimensional PDEs.}
As mentioned above, our method builds on the seminal work of \citet{Han2018} who developed a method for solving semilinear parabolic PDEs. Recent work by \citet{Ata2023} builds on \citet{Han2018} and develops a computational method for solving problems of drift rate control for RBM in the orthant, also see \citet{Ata2023b} for solving singular control problems building on \citet{Ata2023}. 
Those authors considered several test problems of dimension at least up to $30$ and reported near optimal performance. From a methodological perspective, we build on \citet{Ata2023} but modify their method to accommodate the following differences of our setting: \textit{i}) the set of feasible drift rates are defined implicitly through the aforementioned budgeted prize-collecting VRP; see Section \ref{sec:BackgroundAndModel} and \ref{subsec:ApproximateAuxiliaryFunction} and \textit{ii}) our state space is different. Another related paper is the recent working paper \citet{AtaKasikaralar2023} that studies dynamic scheduling of a multiclass queue motivated by call center applications. Focusing on the Halfin-Whitt asymptotic regime, the authors approximate their scheduling problem by a drift rate control problem whose state space is $\mathbb{R}^d$, where $d$ is the number of buffers in their queueing model. Similar to us, those authors build on \citet{Han2018} to solve their (high-dimensional) drift-rate control problem. However, our control problem differs from theirs significantly, because theirs has no state space constraints. There have been many papers on solving PDEs using deep neural networks in the last five years; see \citet{E2021}, \citet{Beck2023} and \citet{Chessari2023} for surveys of that literature.

\section{Model}
\label{sec:BackgroundAndModel}

Two key features of eviction orders affect scheduling and routing decisions of the system manager: their location and whether they have a deadline. According to CCSO data, 83\% of eviction orders have a deadline, typically set at 120 calendar days after the presiding judge approves an enforcement for the eviction order. The remaining 17\% of orders have no deadlines; see Section \ref{sec:Simulation} for further details. 
The location of an eviction order affects the travel time required to enforce (or serve) it, whereas its deadline affects the urgency of the order. We divide the service area into $K$ regions and each eviction order falls into one of those regions. We choose $K$ sufficiently large to ensure the region provides a good proxy for location (see Sections \ref{subsec:ApproximateAuxiliaryFunction} and \ref{subsec:Data} for further details). Within each region, eviction orders are further classified into two types: those with a deadline, or a delay constraint, and those without a deadline. This leads to a $2K$-dimensional state descriptor, denoted by $Z(t)=(Z_1(t), \ldots, Z_K(t), Z_{K+1}(t), \ldots, Z_{2K}(t))'$, where $Z_{k}(t)$ and $Z_{K+k}(t)$ represent the number of outstanding eviction orders in region $k$ with and without a deadline, respectively, for $k=1,\ldots, K$.

We let $A_k(t)$ denote the cumulative number of class $k$ eviction orders that have arrived in the system by time $t$, and assume $A_1, \ldots, A_{2K}$ are mutually independent Gaussian processes. Specially, we assume $A_k(t)$ is a Gaussian random variable with mean $\lambda_k t$ and variance $\sigma_k^2 t$ for $t\geq 0$. In what follows, we consider the evolution of the system in continuous time and allow fractional values for the state descriptor for tractability. Note that the process $\{A_k(t)-\lambda_k t: t\geq 0\}$ is equal in distribution to $\{B_k(t): t\geq 0\}$, where $B_k$ is a driftless Brownian motion with variance parameter $\sigma_k^2$ for $k=1,\ldots, 2K$. It is important to note that our approach does not rely on a heavy traffic analysis, which often leads to similar Brownian control problems. Instead, in our setting, the assumption of Gaussian arrival processes directly leads to the Brownian control problem. Crucially, we relax this assumption in our simulation study. Specifically, while our proposed policy is derived under the Gaussian arrival assumption, we evaluate its performance by bootstrapping daily arrival data from the CCSO dataset. This allows us to test the policy's effectiveness without relying on the Gaussian assumption.

At each time $t\geq 0$, the system manager chooses a vector $\mu(t)\in {\cal A}(Z(t))$, where $\mu_k(t)$ denotes the rate at which class $k$ jobs are served for $k=1,\ldots, 2K$ and ${\cal A}(z)\subset \mathbb{R}_+^{2K}$ denotes the set of feasible service rates as a function of the system state $z$. Vector $\mu$ of service rates can be determined based on the chosen route and the number of eviction orders of different types enforced along that route. The vector $\mu$ is considered feasible if the total travel time combined with the total service time remains within the allotted time budget. The set ${\cal A}(z)$ represents the collection of all such feasible vectors $\mu$. In other words, for each state vector $z$, the set ${\cal A}(z)$ of feasible service rates corresponds to the set of feasible solution for the aforementioned budgeted prize-collecting VRP; see Section \ref{subsec:ApproximateAuxiliaryFunction} for further details and for a discussion of related computations. 

As mentioned earlier, eviction orders in classes 1 through $K$ have a deadline or a delay constraint. Our model assumes a common deadline $d$ for classes $k=1,\ldots, K$. 
Following \citet{Plambeck2001}, we replace the delay constraint with state-space constraints for tractability (see also \citet{Ata2006} and \citet{RubinoAta2009}). To be specific, we let $b_k=\lambda_k d$ and require that $Z_k(t)\leq b_k$ for $k=1,\ldots, K$ and $t\geq 0$, which leads to the state space \vspace{-0.5cm}
\begin{align}
\label{eqn:StateSpace}
    S=\prod_{k=1}^K \left[0, b_k\right] \times \mathbb{R}_+^K,
\end{align} and we rewrite the state-space constraint as 
\vspace{-0.5cm}
\begin{align}
\label{eqn:StateSpaceConstraint}
    Z(t)\in S \quad \text{for } t\geq 0.
\end{align}
\vspace{-1.3cm}

We let $U_k(t)$ denote the cumulative number of class $k$ eviction orders that have missed their deadline during $[0, t]$ for $k=1,\ldots, K$ and $t\geq 0$. According to CCSO data, among those eviction orders with a deadline, 18.83\% of them miss their deadline. We associate a penalty $p$ for each eviction order that misses its deadline. Intuitively, the penalty parameter $p$ reflects the undesirability of not executing the eviction order within the deadline. In addition, while an eviction order is waiting to be served, the plaintiff may cancel it. In the CCSO data, about 25.8\% of eviction orders are canceled. We approximate the instantaneous cancellation rate by $\gamma_k Z_k(t)$ for class $k$ ($k=1,\ldots,2K$) and $t\geq 0$; see Section \ref{sec:Simulation} for estimation of cancellation rate parameters $\gamma_k$ for $k=1,\ldots, 2K$.

Focusing attention on stationary Markov controls, i.e., $\mu(t)=\mu(Z(t))$, the evolution of the system state can be described as follows: For classes with a deadline ($k=1,\ldots,K$), we have 
\begin{align}
\label{eqn:StateDynamics_WithDeadline}
    Z_k(t) &= Z_k(0)+\lambda_k t+B_k(t)-\int_{0}^t \gamma_k Z_k(s)ds-\int_0^t \mu_k(Z_k(s))ds +L_k(t)-U_k(t), \quad t\geq 0. 
\end{align}
Similarly, for classes without a deadline ($k=K+1,\ldots,2K$), we have
\begin{align}
\label{eqn:StateDynamics_WithoutDeadline}
    Z_k(t) &= Z_k(0)+\lambda_k t+B_k(t)-\int_{0}^t \gamma_k Z_k(s)ds-\int_0^t \mu_k(Z_k(s))ds +L_k(t).
\end{align}
In other words, the state process $Z$ is modeled as a Reflected Brownian Motion (RBM) with state-dependent and controllable drift; see \citet{HarrisonReiman1981}.

The processes $L$ and $U$ are pushing processes at the state boundary to ensure the state process $Z$ lives in the state space $S$. The pushing process $L$ is usually interpreted as the cumulative unused service capacity; see \citet{Harrison2013}. Similarly, $U_k(t)$ models the cumulative number of class $k$ eviction orders that have missed their deadline for $k=1,\ldots, K$. They further satisfy the following conditions: 
\vspace{-0.5cm}
\begin{align}
    & L \text{ is non-decreasing, non-anticipating with } L(0)=0, \label{eqn:L}   \\
    & U \text{ is non-decreasing, non-anticipating with } U(0)=0, \label{eqn:U}  \\
    & \int_{0}^t Z_k (s) d L_k(s)= 0, \quad k=1,\ldots, 2K, \quad t\geq 0, \label{eqn:LowerPushingConstraint}  \\
    & \int_0^t \left( b_k - Z_k(s) \right) dU_k(s) = 0, \quad k=1, \ldots, K,  \quad \hspace{0.2cm} t\geq 0, \label{eqn:UpperPushingConstraint}
\end{align} 
which ensure that $L$ and $U$ are minimal processes that keep the state process within the state space. 

The system manager's control $\mu(Z(t))$ corresponds to the vector of service rates at which eviction orders of various classes are enforced at time $t$. The integrals involving $\mu_k(Z_k(t))$ in Equations (\ref{eqn:StateDynamics_WithDeadline})-(\ref{eqn:StateDynamics_WithoutDeadline}) give the cumulative number of class $k$ eviction orders enforced until time $t$. The system manager seeks to choose a policy $\mu(z) \in {\cal A}(z)$ for $z\in S$ so as to 
\begin{align}
    &\text{minimize} \quad \overline{\lim}_{T\rightarrow\infty} \frac{1}{T} \mathbb{E}\left[ \int_0^T \sum_{k=1}^{2K} c_k \left(Z_k(t)\right)dt+p\sum_{k=1}^{K} U_k(T) \right]  \label{eqn:Objective}  \\
    &\text{subject to} \quad (\ref{eqn:StateSpaceConstraint})-(\ref{eqn:UpperPushingConstraint}), \label{eqn:Constraints}
\end{align} 
where the holding cost function $c_k(\cdot)$ expresses the congestion concerns for and the relative urgency of class $k$ ($k=1,\ldots, 2K$). Similarly, the second term in the objective accounts for the penalty incurred for missed deadlines,  as eviction orders cannot be served after their deadline has passed.

In our numerical example, calibrated by the CCSO data, we restrict attention to the following linear holding costs for simplicity: 
\vspace{-0.5cm}
\begin{align*}
    c_k(x) &= \Bigg\{
    \begin{array}{ll}
        c_1 x, & \quad k=1,\ldots, K, \\
        c_2 x, & \quad k=K+1,\ldots, 2K.
    \end{array}
    \Bigg.
\end{align*}
Under appropriate scaling, the linear holding cost with $c_1=c_2$ corresponds to minimizing the total number of pending eviction orders plus the penalty for missing their deadlines. 

Crucially, we do not include a penalty for canceled eviction orders, because cancellations typically result from the parties mediating a solution that averts eviction.

\section{HJB equation and an equivalent SDE} 
\label{sec:HJB}

As a means to characterize the optimal policy, we consider the HJB equation: Find a constant $\beta$ and a sufficiently smooth function $V$ with polynomial growth that jointly satisfy the following:
\begin{align}
\label{eqn:HJB}
	\beta &= {\cal L} V(z)- \max_{\mu\in \mathcal{A}(z)} \left\{\mu \cdot \nabla V(z) \right\}+\sum_{k=1}^{2K} c_k(z_k), \quad z\in S, 
\end{align} subject to the boundary conditions 
\begin{align}
    \frac{\partial V(z)}{\partial z_k} & = 0 \quad \text{ if } z_k=0 \quad (k=1,\ldots, 2K), \label{eqn:HJB_Boundary1} \\
    \frac{\partial V(z)}{\partial z_k} &= p \quad \text{ if } z_k=b_k \quad (k=1,\ldots, K), \label{eqn:HJB_Boundary2} \\
    \lim_{z_k\rightarrow\infty} \frac{\partial V(z)}{\partial z_k} &= \frac{c_2}{\gamma_k} \quad (k=K+1,\ldots, 2K),  \label{eqn:HJB_Boundary3}
\end{align}
where the generator ${\cal L}$ is given as follows: 
\begin{align}
{\cal L} V(z) &= \frac{1}{2} \sum_{k=1}^{2K} \sigma_{k}^2 \frac{\partial^2 V(z)}{\partial z_k^2} + \sum_{k=1}^K \left(\lambda_k-\gamma_k z_k\right) \frac{\partial V(z)}{\partial z_k}. 
\end{align}
Here, one interprets $\beta$ as a guess at the minimum average cost, and the unknown function $V$ is often called the relative value function. 

Next, we derive a key identity, given by Equation (\ref{eqn:Identity}), which provides an equivalent characterization of the value function. In Section \ref{sec:ComputationalMethod}, this identity is used to define the loss function used in our computational method. We begin with a \textit{reference policy} to generate sample paths of the system state. Loosely speaking, we aim to choose a reference policy so that its paths tend to occupy the paths of the state space that we expect the optimal policy to visit frequently. We denote our reference policy by $\tilde{\mu}$. In the computational study, we randomly generate $\tilde{\mu}_k$ for class $k$ to keep the system state (i.e., the number of pending eviction orders) at low, medium, and high levels; see Section \ref{sec:ComputationalMethod} for further details. 

The corresponding reference process, denoted by $(\tilde{Z}, \tilde{L}, \tilde{U})$, satisfy the following: For $t\geq 0$,
\begin{alignat*}{2}
%\label{eqn:StateDynamics_WithDeadline_ReferencePolicy}
    &\tilde{Z}_k(t) =  \tilde{Z}_k(0)+\left(\lambda_k-\tilde{\mu}_k\right) t+B_k(t)-\int_{0}^t \gamma_k \tilde{Z}_k(s)ds +\tilde{L}_k(t)-\tilde{U}_k(t),&& \quad k=1,\ldots, K, \\
    &\tilde{Z}_k(t) = \tilde{Z}_k(0)+\left(\lambda_k-\tilde{\mu}_k\right) t+B_k(t)-\int_{0}^t \gamma_k \tilde{Z}_k(s)ds+\tilde{L}_k(t), && \quad k=K+1,\ldots, 2K,   \\
    &\int_0^t \tilde{Z}_k(s) d\tilde{L}_k(s) = 0, \quad k=1,\ldots, 2K,  \\
    &\int_0^t (b_k - \tilde{Z}_k(s) ) d\tilde{U}_k (s) = 0, \quad k=1,\ldots, K, \\
    & \tilde{Z}(t) \in S,   \\
    & \tilde{L} \text{ is non-decreasing, non-anticipating with } \tilde{L}(0)=0, \\
    & \tilde{U} \text{ is non-decreasing, non-anticipating with } \tilde{U}(0)=0.
\end{alignat*}
To facilitate the definition of the key identity, let 
\begin{align}
    \label{eqn:AuxiliaryFunction}
    F(z, v) &= \max_{\mu\in\mathcal{A}(z)} \left\{ \mu \cdot v \right\}-\sum_{k=1}^{2K} c_k(z_k), \quad z\in S, \; v\in \mathbb{R}_+^{2K}. 
\end{align}

\begin{proposition}
\label{prop:EquivalentSDE}
    If $V(\cdot)$ and $\beta$ solve the HJB equation, then the following holds almost surely for $T>0$,
    \begin{align}
        \begin{split}
            V(\tilde{Z}(T))-V(\tilde{Z}(0)) &= \int_0^T \nabla V (\tilde{Z}(t)) dB(t) + \beta T - p \sum_{k=1}^K U_k(T) \\ 
            &\quad -\int_0^T\tilde{\mu}\cdot\nabla V(\tilde{Z}(t)) dt+\int_0^T F (\tilde{Z}(t), \nabla V(\tilde{Z}(t)) ) dt.  \label{eqn:Identity}
        \end{split}
    \end{align}
\end{proposition}
\begin{proof}
    Applying Ito’s formula to $V(\tilde{Z}(t))$ yields 
    \begin{align}
        \begin{split}
            V(\tilde{Z}(T))-V(\tilde{Z}(0)) &= \int_0^T \sum_{k=1}^{2K}\frac{\partial V}{\partial z_k}(\tilde{Z}(t)) dB_k(t)+\frac{1}{2}\int_0^T \sum_{k=1}^{2K} \sigma_k^2\frac{\partial^2 V}{\partial z_k^2} (\tilde{Z}(t)) dt  \\
            &\quad + \int_0^T \sum_{k=1}^{2K}\frac{\partial V}{\partial z_k}(\tilde{Z}(t)) \left(\lambda_k-\gamma_k \tilde{Z}_k(t)-\tilde{\mu}_k \right) dt \\ 
            &\quad +\int_0^T \sum_{k=1}^{2K}\frac{\partial V}{\partial z_k}(\tilde{Z}(t))   d\tilde{L}_k(t)-\int_0^T \sum_{k=1}^{K}\frac{\partial V}{\partial z_k}(\tilde{Z}(t)) d\tilde{U}_k(t).  
        \end{split}
    \end{align}
    It follows from HJB equation in (\ref{eqn:HJB})-(\ref{eqn:HJB_Boundary2}) and (\ref{eqn:AuxiliaryFunction}) that 
    %\resizebox{\textwidth}{!}{
    \begin{align}
     %   \begin{split}
            V(\tilde{Z}(T))-V(\tilde{Z}(0)) &= \int_0^T \sum_{k=1}^{2K}\frac{\partial V}{\partial z_k}(\tilde{Z}(t)) dB_k(t)
            + \int_0^T \left[{\cal L}V(\tilde{Z}(t))- \tilde{\mu}\cdot\nabla V(\tilde{Z}(t))\right]dt - p\sum_{k=1}^{K}\tilde{U}_k(T) \nonumber \\
            &= \int_0^T \sum_{k=1}^{2K}\frac{\partial V}{\partial z_k}(\tilde{Z}(t)) dB_k(t) + \int_0^T \left[\beta + F(\tilde{Z}(t), \nabla V(\tilde{Z}(t))-\tilde{\mu}\cdot\nabla V(\tilde{Z}(t))  \right] dt \nonumber \\
            &\quad - p\sum_{k=1}^{K}\tilde{U}_k(T),
    %    \end{split}
    \end{align}
    %} 
    from which (\ref{eqn:Identity}) follows by rearranging the terms.
\end{proof}

\begin{proposition}
\label{prop:EquivalentSDE2}
    Suppose that $V:S\rightarrow \mathbb{R}$ is a ${\cal C}^2$ function, $G:S\rightarrow \mathbb{R}$ is continuous and $V, \nabla V, G$ all have polynomial growth. Also assume that the following identity holds almost surely for some fixed $T>0$, a scalar $\beta$ and every $Z(0)=z\in S$: 
    \begin{align}
        \begin{split}
        V(\tilde{Z}(T))-V(\tilde{Z}(0)) &= \int_0^T G( \tilde{Z}(t))\cdot dB(t)+\beta T- p \sum_{k=1}^{K} U_k(T) \\
        &\quad -\int_0^T \tilde{\mu}\cdot G(\tilde{Z}(t))dt+\int_0^T F(\tilde{Z}(t), G(\tilde{Z}(t))) dt. \label{eqn:Proposition2}
        \end{split}
    \end{align}
    Then, $G(\cdot)=\nabla V(\cdot)$ and $(V, \beta)$ solve the HJB equation. 
\end{proposition}
\begin{proof}
    One can show that under the reference policy $\tilde{\mu}$, the reference process $\tilde{Z}(\cdot)$ has a unique stationary distribution. We denote it by $\tilde{\pi}$, and let $\tilde{Z}(\infty)$ be a random variable with distribution $\tilde{\pi}$. Then assuming the initial distribution of $\tilde{Z}(\cdot)$ is $\tilde{\pi}$, i.e., $\tilde{Z}(0)\sim \tilde{\pi}$, its marginal distribution of time $t$ is also $\tilde{\pi}$, i.e., $\tilde{Z}(t)\sim \tilde{\pi}$ for every $t\geq 0$.

    Moreover, because $\tilde{Z}$ is a time-homogeneous Markov process, we can express (\ref{eqn:Proposition2}) equivalently as follows for any $\ell=0,1,\ldots$:
    \begin{align}
        \begin{split}
            V(\tilde{Z} ((\ell+1)T)) - V(\tilde{Z} (\ell T)) &= \int_{\ell T}^{(\ell+1)T} G(\tilde{Z}(t))\cdot dB(t)+\beta T - p\sum_{k=1}^K \left[ \tilde{U}_k ((\ell+1)T)-\tilde{U}_k (\ell T) \right] \\
            &\quad -\int_{\ell T}^{(\ell+1)T} \tilde{\mu}\cdot G(\tilde{Z}(t))dt+\int_{\ell T}^{(\ell+1)T} F (\tilde{Z}(t), G(\tilde{Z}(t)))dt. 
        \end{split}
    \end{align}
    Adding these for $\ell=0,1,\ldots,n-1$ gives 
    \begin{align}
        \begin{split}
            V(\tilde{Z} (nT))  &= V(\tilde{Z} (0))+ \int_0^{nT} G(\tilde{Z}(t))\cdot dB(t)+\beta n T - p\sum_{k=1}^K \tilde{U}_k (nT) \\
            & \quad - \int_0^{nT} \tilde{\mu}\cdot G(\tilde{Z}(t))dt + \int_0^{nT} F (\tilde{Z}(t), G(\tilde{Z}(t)))dt \label{eqn:StationaryIdentity}
        \end{split}
    \end{align}
    for $n\geq 0$. Taking the expectation of both sides, and noting that $\E_{\tilde{\pi}} [V(\tilde{Z}(0)]=\E_{\tilde{\pi}} [V(\tilde{Z}(nT)]$, we conclude that 
    \begin{align}
        \beta nT &= \E_{\tilde{\pi}} \left[p \sum_{k=1}^K \tilde{U}_k (nT)+\int_0^{nT}\tilde{\mu}\cdot G(\tilde{Z}(t))dt-\int_0^{nT} F(\tilde{Z}(t), G(\tilde{Z}(t)))dt \right].
    \end{align}
    Letting $n\rightarrow \infty$, it follows that 
    \begin{align}
        \beta &= \overline{\lim}_{nT\rightarrow\infty} \frac{1}{nT} \E_{\tilde{\pi}} \left[ p \sum_{k=1}^K U_k (nT)+\int_0^{nT} \tilde{\mu}\cdot G(\tilde{Z}(t))dt-\int_0^{nT} F(\tilde{Z}(t), G(\tilde{Z}(t))) dt \right]
    \end{align}
    In words, $\beta$ is the long-run average cost associated with the reference process $\tilde{Z}$, where the running cost is $\tilde{\mu}\cdot G(\cdot)-F(\cdot, G(\cdot))$. Then it is straightforward to argue that there exists a relative value function $\tilde{V}$ that satisfies the following PDE: 
    \begin{align}
        \beta &= {\cal L} \tilde{V}(z)-\tilde{\mu}\cdot \nabla \tilde{V}(z)+\tilde{\mu}\cdot G(z) - F(z, G(z)), \quad z\in S, \label{eqn:HJB_Prop2}  \\
        \frac{\partial \tilde{V}(z)}{\partial z_k} &= 0 \quad \text{ if } z_k = 0,\; k=1,\ldots, 2K, \label{eqn:HJB_Boundary1_Prop2} \\
        \frac{\partial \tilde{V}(z)}{\partial z_k} &= p \quad \text{ if } z_k = b_k,\; k=1,\ldots, K, \label{eqn:HJB_Boundary2_Prop2}   \\
        \lim_{z_k\rightarrow \infty} \frac{\partial \tilde{V}(z)}{\partial z_k} &= \frac{c_2}{\gamma_k}, \; k=K+1,\ldots,2K. \label{eqn:HJB_Boundary3_Prop2}
    \end{align}
    Applying Ito's formula to $\tilde{V}(\tilde{Z}(t))$ yields 
    \begin{align}
        \begin{split}
            \tilde{V}(\tilde{Z}(nT))-\tilde{V}(\tilde{Z}(0)) &= \int_0^{nT} \nabla \tilde{V} (\tilde{Z}(t))\cdot dB(t)+\int_0^{nT} {\cal L} \tilde{V}(\tilde{Z}(t)) dt \\
            &\quad - \int_0^{nT} \tilde{\mu}\cdot \nabla \tilde{V} (\tilde{Z}(t)) dt - p \sum_{k=1}^K \tilde{U}_k (nT).  \label{eqn:ItoLemma_Prop2}
        \end{split}
    \end{align}
    Then substituting (\ref{eqn:HJB_Prop2}) into (\ref{eqn:ItoLemma_Prop2}) gives 
    \begin{align}
        \begin{split}
            \tilde{V}(\tilde{Z}(nT))-\tilde{V}(\tilde{Z}(0)) &= \int_0^{nT} \nabla \tilde{V} (\tilde{Z}(t))\cdot dB(t)+\beta nT + \int_0^{nT} F (\tilde{Z}(t), G(\tilde{Z}(t))\cdot \tilde{\mu} dt \\
            &\quad -\int_0^{nT} \tilde{\mu}\cdot G(\tilde{Z}(t)) dt - p \sum_{k=1}^K \tilde{U}_k (nT).  \label{eqn:ItoLemma_Substitute}
        \end{split}
    \end{align}
    Conditioning on $\tilde{Z}(0)=z$, and taking expectations give 
    \begin{align}
        \tilde{V}(z) &= \E_{z} \left[ \tilde{V}(\tilde{Z}(nT)) \right]-\beta nT+\E_z \left[p \sum_{k=1}^K \tilde{U}_k (T) +\int_0^{nT}  \tilde{\mu} \cdot G(\tilde{Z}(t)) dt -\int_0^{nT} F(\tilde{Z}(t), G(\tilde{Z}(t))dt \right]. \label{eqn:ExpectationVTilde}
    \end{align}
    Similarly, it follows from (\ref{eqn:StationaryIdentity}) that 
    \begin{align}
        V(z) &= \E_z \left[V(\tilde{Z}(nT))\right]-\beta nT+\E_z \left[p \sum_{k=1}^K \tilde{U}_k (T)+\int_0^{nT} \tilde{\mu} \cdot G(\tilde{Z}(t)) dt-\int_0^{nT} F(\tilde{Z}(t), G(\tilde{Z}(t))dt \right], \label{eqn:ExpectationV}
    \end{align}
    for $z\in S$ and $n$ an arbitrary positive integer. Subtracting (\ref{eqn:ExpectationVTilde}) from (\ref{eqn:ExpectationV}) gives
    \begin{align}
        V(z)-\tilde{V}(z) &= \E_z \left[V(\tilde{Z}(nT))\right]-\E_z \left[  \tilde{V}(\tilde{Z}(nT))\right].  \label{eqn:Prop2Convergence}
    \end{align}
    Without loss of generality, we assume $\E[V(Z(\infty))]=E[\tilde{V}(Z(\infty))]$. Since $V(\cdot)$ and $\tilde{V}(\cdot)$ have polynomial growth, we have $\sup_{n>0} \E_z [V(\tilde{Z}(nT))^2]<\infty$ and $\sup_{n>0} \E_z [\tilde{V}(\tilde{Z}(nT))^2]<\infty$. Then, by Vitali's convergence theorem, we conclude that \begin{align}
        \lim_{n\rightarrow\infty} \E_z \left[V(Z(nT)) \right] &= \E\left[ V(\tilde{Z}(\infty))\right],   \\
        \lim_{n\rightarrow\infty} \E_z \left[\tilde{V}(Z(nT)) \right] &= \E\left[ \tilde{V}(\tilde{Z}(\infty))\right].
    \end{align}
    Therefore, passing to the limit in (\ref{eqn:Prop2Convergence}), we conclude 
    \begin{align}
        V(z) &= \tilde{V}(z) \quad \text{ for } z\in S, 
    \end{align}
    which means $V$ satisfies the PDE (\ref{eqn:HJB_Prop2})-(\ref{eqn:HJB_Boundary3_Prop2}). That is, 
    \begin{align}
        \beta &= {\cal L} V(z) -\tilde{\mu}\cdot \nabla V(z)+\tilde{\mu}\cdot G(z) - F(z, G(z)), \quad z\in S, \label{eqn:HJB2_Prop2} \\
        \frac{\partial V(z)}{\partial z_k} &= 0, \quad \text{ if } z_k=0 \; (k=1,\ldots, 2K), \label{eqn:HJB2_Boundary1_Prop2} \\
        \frac{\partial V(z)}{\partial z_k} &= p, \quad \text{ if } z_k=b_k \; (k=1,\ldots, K). \label{eqn:HJB2_Boundary2_Prop2} \\
        \lim_{z_k\rightarrow\infty} \frac{\partial V(z)}{\partial z_k} &= \frac{h_2}{\gamma_k},\quad (k=K+1,\ldots,2K). \label{eqn:HJB2_Boundary3_Prop2}
    \end{align}
    Suppose for now that $G(\cdot)=\nabla V(\cdot)$ (which will be shown later). Substituting this in Equation (\ref{eqn:HJB2_Prop2}) and using the definition of $F$ gives 
    \begin{align}
        \beta &= {\cal L} V(z) - \max_{\mu\in {\cal A}(z)} \{\mu\cdot \nabla V(z)\} + \sum_{k=1}^{2K} c_k(z),  \quad z\in S, 
    \end{align}
    which along with (\ref{eqn:HJB2_Boundary1_Prop2})-(\ref{eqn:HJB2_Boundary3_Prop2}) gives the desired result, i.e., $(\beta, V)$ satisfies HJB equations. 

    To complete the proof, it remains to show that $G(\cdot)=\nabla V(\cdot)$. By applying Ito's lemma to $V(\tilde{Z}(t))$, we conclude 
    \begin{align}
        \begin{split}
            V(\tilde{Z}(T))-V(\tilde{Z}(0)) &= \int_0^{T} ({\cal L} V (\tilde{Z}(t))-\tilde{\mu} \cdot \nabla V(\tilde{Z}(t)) ) dt + \int_0^{T} \nabla V (\tilde{Z}(t))\cdot dB(t) - p \sum_{k=1}^K \tilde{U}_k (T).  \label{eqn:ItoLemma2_Prop2}
        \end{split}
    \end{align}
    Then using (\ref{eqn:HJB2_Prop2}), we write this as follows: 
    \begin{align}
        \begin{split}
            V(\tilde{Z}(T))-V(\tilde{Z}(0)) &= \beta T -\int_0^T \tilde{\mu} \cdot G(\tilde{Z}(t))dt + \int_0^T F(\tilde{Z}(t), G(\tilde{Z}(t))dt \\
            &\quad - p\sum_{k=1}^K \tilde{U}_k(t)+ \int_0^{T} \nabla V (\tilde{Z}(t))\cdot dB(t)
        \end{split}
    \end{align}
    Combining this with Equation (\ref{eqn:Proposition2}) yields 
    \begin{align*}
        \int_0^T (G(\tilde{Z}(t))-\nabla V(\tilde{Z}(t))) \cdot dB(t)= 0.
    \end{align*}  
    which leads to the following 
    \begin{align}
        \E_z \left[\left( \int_0^T (G(\tilde{Z}(t))-\nabla V(\tilde{Z}(t))) \cdot dB(t) \right)^2 \right] &= 0
    \end{align}
    Then, by Ito's isometry, we write 
    \begin{align}
        \E_z \left[ \int_0^T \norm{g(\tilde{Z}(t))-\nabla V(\tilde{Z}(t))}  ^2 dt \right] &= 0.
    \end{align}
    Thus, $\nabla V(\tilde{Z}(t))=G(\tilde{Z}(t))$ almost surely. By continuity of $\nabla V(\cdot)$ and $G(\cdot)$, we conclude that $\nabla V(\cdot)=G(\cdot)$. 
\end{proof}

\subsection{Approximating the Auxiliary Function $F(z, v)$}
\label{subsec:ApproximateAuxiliaryFunction}

Equation (\ref{eqn:AuxiliaryFunction}) defines the auxiliary function $F(z, v)$ implicitly due to the term $\max_{\mu\in \mathcal{A}(z)} \{\mu v\}$, where ${\cal A}(z)$ denotes the set of feasible eviction enforcement rate vectors given the system state $z$. Thus, computing $F(z, v)$ involves solving a budgeted prize-collecting VRP.
It is budgeted, because each eviction team can work only a certain number of hours. The ``prizes'' correspond to the vector $v$, which will ultimately be given by the gradient $\nabla V$ of the value function, but it is computationally taxing to solve an instance of the budgeted prize-collecting VRP online as needed at every iteration of our algorithm. 

To facilitate our analysis, we define the auxiliary function $H(z, v)=\max_{\mu\in {\cal A}(z)} \{\mu v\}$ and approximate it offline using a separate neural network model, which is trained using the solutions to a large number of representative budgeted prize-collecting VRPs for a wide-range of state-prize pairs $(z,v)$. In doing so, we sample the input for the state vector $z$ using the CCSO data, utilizing historical information on pending requests on a daily basis. In addition, we choose $v$ uniformly at random in a set that we expect the gradient $\nabla V$ of the value function to take values in. 

Solving the budgeted prize-collecting VRP repeatedly for the purpose of generating the data to be used to train the neural network, which approximates the auxiliary function $H$, is also computationally demanding. Therefore, we build on \citet{Hegde2015} and \citet{GoemansWilliamson1995} to derive an algorithm for solving the budgeted prize-collecting TSP. \citet{Hegde2015} considers a related problem, the prize-collecting Steiner tree problem (PCST), and develops an efficient algorithm for solving it. We adapt their approach by converting the tree identified by their algorithm into a tour using the method described by \citet{GoemansWilliamson1995}. We then repeat this for each eviction team, one at a time. We refer the reader to Appendix \ref{app:TrainingFFunction_Appendix} for further details of the approximation of the function $H$.

\section{Computational Method}
\label{sec:ComputationalMethod}

Our computational method builds on the seminal work of \citet{Han2018}. Similar to their approach, we use the identity in Equation (\ref{eqn:Identity}) to define our loss function. With the chosen reference policy $\tilde{\mu}$ introduced in Section \ref{sec:HJB}, we first simulate the corresponding discretized sample paths. To do so, we first fix a partition $0=t_0<t_1<\cdots<t_N=T$ of the time horizon $[0, T]$ and then simulate the corresponding discretized sample paths of the reference process at each time $t=t_0,\ldots, t_N$. 
Algorithm \ref{alg:Discretization} describes the discretization procedure. Note that the discretization procedure requires solving a related Skorokhod problem; see Algorithm \ref{alg:Skorokhod}.
\begin{algorithm}[h!]
	\caption{Euler discretization scheme}\label{alg:Discretization}
    {\bf Input:} The drift vector $\tilde{\mu}$, the variance parameter $\sigma_k^2$, the time horizon $T$, the number of intervals $N$, a step-size $\Delta t$ (for simplicity, we assume $\Delta t=T/N$ is an integer), and a starting point $\tilde{Z}(0)=z\in S$.

    {\bf Output:} A discretized reflected Brownian motion and the Brownian increments at times $t_n$, where $n=0,1,\ldots, N-1$, and the values of $\tilde{U}(T)$.
	\begin{algorithmic}[1]
        \Function{Discretize}{$T, \Delta t, z$} 
		\State For interval $[0, T]$ and $N=T/\Delta t$, construct the partition $0=t_0<\cdots<t_N=T$, where $t_n=n\Delta t$ for $n=0,\ldots,N$.
		\State Generate $N$ i.i.d. $2K$-dimensional Gaussian random vectors with mean zero and covariance $\Delta t\sigma^2 I$ for $n=0,\ldots, N-1$, denoted by $\delta_0, \ldots, \delta_{N-1}$.
        \For{$k\leftarrow 1$ to $K$} 
        \State $u_k\leftarrow 0$
        \EndFor
        \For{$n\leftarrow 0$ to $N-1$}
        \For{$k\leftarrow 1$ to $2K$}
        \State $x\leftarrow \tilde{Z}_k(t_n)+\delta_n+(\lambda_k-\tilde{\mu}_k-\gamma_k\tilde{Z}_k(t_n)) \Delta t$
        \EndFor
        \State $(\tilde{Z}(t_{n+1}), u)\leftarrow \Call{Skorokhod}{x, u}$
        \EndFor
        \State \Return Discretized reflected Brownian motion $\tilde{Z}(t_1), \tilde{Z}(t_2), \ldots, \tilde{Z}(t_N)$, Brownian increments $\delta_0, \ldots, \delta_{N-1}$, and the value $u$.
        \EndFunction
	\end{algorithmic}
\end{algorithm}

\begin{algorithm}[h!]
	\caption{Solution to the Skorokhod problem}\label{alg:Skorokhod}
    {\bf Input:} A vector $x\in \mathbb{R}^{2K}$ and $u\in \mathbb{R}^{2K}$. 

    {\bf Output:} A solution to the Skorokhod problem $y\in \prod_{k=1}^K [0, b_k] \times \mathbb{R}_+^{K}$ and current value of $\tilde{U}(\cdot)$.
	\begin{algorithmic}[1]
        \Function{Skorokhod}{$x, u$}
		\For{$k\leftarrow 1$ to $2K$}
        \State $u_k\leftarrow u_k+\max\{0, x_k-b_k\}$. 
        \State If $1\leq k\leq K$, $x_k\leftarrow\max\{0, \min\{x_k, b_k\}\}$. 
        \State Else, $x_k\leftarrow \max\{0, x_k\}$.
        \EndFor
        \State \Return $x$ and $u$.
        \EndFunction
	\end{algorithmic}
\end{algorithm}

As discussed in Section \ref{sec:HJB}, we select $\tilde{\mu}$ so that the resulting sample paths are likely to occupy those generated by the optimal policy. In our numerical study, we choose $\tilde{\mu}$ randomly so that the generated sample paths occupy low, medium, or high levels within the state space. More specifically, we choose $\tilde{\mu}$ differently depending on whether the class has a deadline or not. When a class $k$ has a deadline (i.e., $1\leq k\leq K$), we aim to select $\tilde{\mu}_k$ so that the generated sample paths stay near $b_k$ (high level), $b_k/2$ (medium level), and 0 (low level). Similarly, for a class $k$ has no deadline (i.e., $K+1\leq k\leq 2K$), we also choose $\tilde{\mu}$ to occupy a high, medium, and low levels; see a detailed discussion of the reference policy in Appendix \ref{subapp:ReferencePolicy}. 
In addition, we also consider multiplying the diffusion coefficient $\sigma$ by a constant so that the reference policy can visit a larger set of states. In a certain sense, this increases the variation of the training data we generate using the reference process.

We parameterize $V(\cdot)$ and $\nabla V(\cdot)$ using deep neural networks $V^{w_1}(\cdot)$ and $G^{w_2}(\cdot)$ with parameters vector $w_1$ and $w_2$, respectively. We then use the identity (\ref{eqn:Identity}) to define the loss function:
\begin{align*}
	L (w_1, w_2, \beta) &= \E \Bigg[\Bigg(V^{w_1} (\tilde{Z}(T))-V^{w_1} (\tilde{Z}(0))- \int_{0}^T G^{w_2} (\tilde{Z}(t)) \cdot dB(t) -T\beta +p \sum_{k=1}^K \tilde{U}_k(T) \\
    &\hspace{1cm} \quad +\int_0^T \tilde{\mu}\cdot G^{w_2}(\tilde{Z}(t)) dt   -\int_{0}^T F (\tilde{Z}(t), G^{w_2}(\tilde{Z}(t)))dt \Bigg)^2\Bigg].
\end{align*}

By defining $\ell (w_1, w_2)=\min_\beta L(w_1, w_2,\beta)$ and noting that $\min_\beta \E [(X-\beta)^2]=\Var (X)$, we work with the following loss function: \begin{align*}
	\ell (w_1, w_2) & = \Var \Bigg(V^{w_1} (\tilde{Z} (T))- V^{w_1} (\tilde{Z}(0)) - \int_{0}^T G^{w_2} (\tilde{Z} (t)) dB(t) +p\sum_{k=1}^K \tilde{U}_k(T)  \\
 &\hspace{1cm} \quad +\int_0^T \tilde{\mu}\cdot G^{w_2}(\tilde{Z}(t)) dt -\int_{0}^T  F (\tilde{Z}(t), G^{w_2} (\tilde{Z}(t)) )dt \Bigg). 
\end{align*}
Algorithm \ref{alg:MultiDimControl} presents our method to derive the proposed policy, which approximates the loss function $\ell(w_1,w_2)$ by the empirical loss $\hat{\ell}(w_1,w_2)$ (defined in Algorithm \ref{alg:MultiDimControl}) using discretized sample paths to approximate the integrals.

\begin{algorithm}[h!]
	\caption{}\label{alg:MultiDimControl}
    {\bf Input:} The number of iteration steps $M$, a batch size $B$, a learning rate $\alpha$, a time horizon $T$, a discretization step-size $\Delta t$ (for simplicity, we assume $N=T/\Delta t$ is an integer), a starting point (initial state) $z$, and an optimization solver (SGO, Adam, RMSprop, etc.).

    {\bf Output:} A neural network approximation of the value function $V^{w_1}$ and the gradient function $G^{w_2}$. 
	\begin{algorithmic}[1]
        \State Obtain the approximated function $\hat{H}(z, v)$ as described in Appendix \ref{subapp:ApproximateFunctionH}. Let \begin{align*}
            \hat{F}(z, v)=\hat{H}(z, v)-\sum_{k=1}^{2K} c_k(z_k).
        \end{align*} 
		\State Initialize the neural networks $V^{w_1}$ and $G^{w_2}$. Set $x_0^{(i)}=z_0$ for $i=1,2,\ldots,B$. 
		\For{$j\leftarrow 0$ to $M-1$}
		\State Simulate $B$ discretized RBM paths and the Brownian increments $\{\tilde{Z}^{(i)}, \delta^{(i)}\}$ and $\tilde{U}^{(i)}(T)$ with a time horizon $T$ and a discretization step-size $\Delta t$ starting from $\tilde{Z}^{(i)}(0)=x_k^{(i)}$ by invoking \Call{Discretize}{$T, \Delta t, x_k^{(i)}$}, for $i=1,\ldots, B$.
		\State Compute the empirical loss \begin{align*}
	       \hat{\ell} (w_1, w_2) & = \widehat{\Var} \Bigg(V^{w_1} (\tilde{Z}^{(i)} (T))- V^{w_1} (\tilde{Z}^{(i)}(0)) - \sum_{j=0}^{N-1} G^{w_2} (\tilde{Z}^{(i)} (j \Delta t)) \cdot \delta_j^{(i)}+ p\sum_{k=1}^K \tilde{U}^{(i)}_k(T)  \\ &\hspace{2cm} +\sum_{j=0}^{N-1}\tilde{\mu}\cdot G^{w_2}(\tilde{Z}^{(i)}(j\Delta t)) -\sum_{j=0}^{N-1}  \hat{F}\left(\tilde{Z}^{(i)}(j\Delta t), G^{w_2} (\tilde{Z}^{(i)}(j\Delta t)) \right) \Delta t \Bigg). 
        \end{align*}
		\State Compute the gradient $\partial \hat{\ell}(w_1,w_2)/\partial w_1$ and $\partial \hat{\ell}(w_1,w_2)/\partial w_2$ and update $w_1, w_2$ using the chosen optimizer. 
		\State Update $x^{(i)}_{j+1}$ as the end point of the path $\tilde{Z}^{(i)}: x^{(i)}_{j+1} \leftarrow \tilde{Z}^{(i)}(T)$.
		\EndFor
  
		\State \Return Functions $V^{w_1}(\cdot)$ and $G^{w_2}(\cdot)$.
	\end{algorithmic}
\end{algorithm}

\begin{remark}
    By choosing the time unit suitably so that $t_{n+1}-t_n$ is small for all $n$, one can view $t_n$ as the beginning of day $n$ in our application. Then the various stochastic processes and decisions made during the interval $[t_n, t_{n+1})$ correspond to those on day $n$. However, our proposed policy is simple enough that one does not need to do this translation explicitly, because all it requires is the neural network approximation $G^{w_2}(\cdot)$ of the value function gradient $\nabla V(\cdot)$; see Section \ref{sec:BenchmarkPolicies}.
\end{remark}

%\newpage
\section{Proposed Policy and Benchmarks}
\label{sec:BenchmarkPolicies}

In this section, we describe our proposed policy along with two benchmark policies. First, we describe the general framework common to all three policies. Then we explain how they differ. We compare the performance of the two benchmarks with the proposed policy through a simulation study that is calibrated using CCSO data in Section \ref{sec:Simulation}. 

In order to plan the eviction enforcement activities on a daily basis, the system manager considers the location, the deadline (if there is one), and how long it has been waiting for each eviction order. Recall that the class of an eviction order is determined by its location and whether it has a deadline. Using this information, she assigns prizes to the pending orders. She then determines which orders to enforce and what route to follow using these prizes. In doing so, she solves a budgeted prize-collecting VRP using the exact address of each order. This general framework is common to all three policies described below. The policies we consider differ from each other in two ways: First, they all differ in how they assign prizes to each eviction order. Second, our second benchmark (motivated by the current practice) assigns each team to a geographic region and solves the VRP restricting attention to the eviction orders in that region, whereas our proposed policy and the first benchmark consider all eviction orders when solving the VRP, regardless of their location.

Given the system state $z$, the set ${\cal I}=\{1, \ldots, \sum_{k=1}^{2K} z_k\}$ indexes the pending eviction orders. Similarly, the set ${\cal I}^{(k)}=\{1,\ldots,z_k\}$ indexes the eviction orders in class $k$ ($k=1,\ldots,2K$). Letting ${\cal S}\subset {\cal I}$ a generic index set, we consider the $i$th eviction order in that set, and let $\tau_{\cal S}(i), d_{\cal S}(i)$, and $\pi_{\cal S}(i)$ denote how long it has been waiting, its deadline (if any) and its prize, respectively. Note that the index of an eviction order can be different under different policies.

\paragraph{Proposed policy.}
Fix $t>0$ and assume, without loss of generality, that eviction orders in class $k$, where $1\leq k\leq K$, are labeled in ascending order of the remaining time until their deadline. Specifically, we assume
\begin{align*}
    d_{{\cal I}^{(k)}}(1)-\tau_{{\cal I}^{(k)}}(1) \leq d_{{\cal I}^{(k)}}(2)-\tau_{{\cal I}^{(k)}}(2) \leq \cdots \leq d_{{\cal I}^{(k)}}(z_{k})-\tau_{{\cal I}^{(k)}} (z_{k}), \quad k=1, \ldots, K. 
\end{align*}
For eviction orders in class $k$ where $K+1\leq k\leq 2K$, i.e., those classes with no deadlines, they are labeled based in descending order of how long the order has been waiting, i.e., we assume 
\vspace{-2mm}
\begin{align*}
    \tau_{{\cal I}^{(k)}}(1)\geq \tau_{{\cal I}^{(k)}}(2)\geq \cdots \geq \tau_{{\cal I}^{(k)}}(z_k), \quad k=K+1,\ldots,2K.
\end{align*}
\vspace{-2mm}
For $i=1,\ldots, z_k$ and $k=1,\ldots, 2K$, the proposed policy assigns prizes to each eviction order in class $k$ as follows: 
\vspace{-0.5cm}
\begin{align*}
    \pi_{{\cal I}^{(k)}}(i) &= G_k^{w_2}(z_k-i+1).
\end{align*}
Recall that $G^{w_2}$ is the neural network that approximates the gradient $\nabla V$ of the value function. Naturally, one would expect $G^{w_2}$ to be increasing, leading to higher prizes for eviction orders that have a smaller index within each class. Algorithm \ref{alg:HighLevel_Proposed} presents a high-level description of how we implement the proposed policy.

\begin{algorithm}[h!]
		\caption{High-level description of the implementation for the proposed policy} \label{alg:HighLevel_Proposed}
		%{\fontsize{10}{12}\selectfont	
			\begin{algorithmic}[1]
				\State Obtain $G^{w_2}$ from Algorithm 3.
				\For{each vehicle}
				\For{$k=1,\ldots,2K$}
				%\State $I_k\leftarrow \{i\in {\cal P}\mid i(\text{class})=k\}$
				%\State $z_k\leftarrow |I_k|$
				\State Let $I_k$ be the set of pending orders of class $k$ and $z_k=|I_k|$.
				\If{$1\leq k\leq K$}
				\State $I_k\leftarrow$ Sort $I_k$ by ascending remaining time to deadline for $i\in I_k$.
				\Else
				\State $I_k\leftarrow$ Sort $I_k$ by descending pending time for $i\in I_k$.
				\EndIf
				\EndFor
				\For{$k=1,\ldots,2K$}
				%\State ${\rm index}\leftarrow 0$
				\For{$i\in I_k$}
				%\State $\tilde{z}\leftarrow 2K$-vector of $z_k$ for $k=1,\ldots, 2K$
				%\State $\tilde{z}_k\leftarrow z_k-{\rm index}+1$
				%\State $\pi_i\leftarrow G_k^{w_2}(\tilde{z})$
				%\State ${\rm index}\leftarrow {\rm index}+1$
				\State Assign prize for the $i$th order in class $k$ as $G_k^{w_2}(z_k-i+1)$ as described.
				\EndFor
				\EndFor
				\State Solve the budgeted prize-collecting TSP for the specific vehicle.
				\EndFor
			\end{algorithmic}
			%}
	\end{algorithm}
    
\paragraph{An urgency-based policy.}

This benchmark strives to prioritize eviction orders according to their urgency. First, we assign an artificial deadline of $\tilde{d}$ to all eviction orders without a deadline. Then, at any decision time, we let ${\cal I}_1$ be the set of eviction orders that either have a deadline or they do not have a deadline, but they have been waiting for less than $\tilde{d}$. We let the set ${\cal I}_2$ contain the rest of the eviction orders, i.e., those with no deadline and have been waiting longer than $\tilde{d}$. We also let ${\cal I}={\cal I}_1\cup {\cal I}_2$ denote the set of all pending eviction orders and index them such that the first $|{\cal I}_1|$ of them are in set ${\cal I}_1$. Their indexing further satisfies
\begin{align*}
    d_{{\cal I}_1}(1)-\tau_{{\cal I}_1}(1)\leq d_{{\cal I}_1}(2)-\tau_{{\cal I}_1}(2)\leq \cdots \leq d_{{\cal I}_1}(|{\cal I}_1|)-\tau_{{\cal I}_1}(|{\cal I}_1|).
\end{align*}

That is, the eviction orders with the least slack in set ${\cal I}_1$ is ranked highest among the orders in set ${\cal I}_1$ (recall that if eviction order $i$ does not have a deadline, we set $d_{{\cal I}_1}(i)=\tilde{d}$). Then we set the prizes as follows: 
\begin{align}
    \pi_{\cal I}(i) &= \begin{cases}
    1/2^i, & \text{ if } i\in {\cal I}_1, \\
    1, & \text{ otherwise }.
    \end{cases}  \label{eqn:FCFSPrize}
\end{align}

The motivation for this is as follows: First, the eviction orders in set ${\cal I}_2$, i.e., those that have been waiting longer than $\tilde{d}$, will have the highest priority. Second, note for the eviction orders in set ${\cal I}_1$ that 
\begin{align*}
    1 > \pi_{{\cal I}_1}(i) > \sum_{\ell=i+1}^{|{\cal I}_1|} \pi_{{\cal I}_1}(\ell) \quad \text{ for }  i=1,\ldots, {\cal I}_1.
\end{align*}
That is, the reward from serving eviction order $i$ is higher than all of the lower ranked eviction orders in set ${\cal I}_1$. In addition, the reward for serving an eviction order in set ${\cal I}_2$ is larger than the total reward from serving all orders in set ${\cal I}_1$. Collectively, these observations imply the following priority rule: First, serve the eviction orders in set ${\cal I}_2$ which all have prize one. Then, serve the remaining eviction orders in the order of their urgency, i.e., lowest slack first. The urgency-based policy solves the budgeted prize-collecting VRP formulation using the prizes given in (\ref{eqn:FCFSPrize}). Lastly, the value of the artificial deadline $\tilde{d}$ is calibrated based on how the service effort is divided between the eviction orders with or without deadlines in the CCSO data; see Appendix \ref{app:ModelCalibration} for further details.

\begin{remark}
One can implement other priority rules such as FCFS similarly.

\end{remark}

\paragraph{Threshold-based zone clearing policy (a benchmark policy motivated by the current practice).} 

Roughly speaking, the current CCSO eviction enforcement schedule is based on the following broad principles \citep{PeterPon2023}: Cook county is divided into 12 zones (see Figure \ref{fig:OriginalZoneDesign_WithLabel} in Section \ref{sec:Simulation}), each with either dense demand, e.g., downtown Chicago, or sparse demand, e.g., the suburbs. An eviction team is dispatched to a zone with sparse demand only if the number of pending eviction orders in that zone is sufficiently high. If no such zones exist, all teams are sent to the zones with dense demand. 

Thus, we propose the following benchmark policy referred to as the {\it threshold-based zone clearing policy}. Recall from Section \ref{sec:BackgroundAndModel} that the entire service area are divided into $K$ regions. We use the same division in this policy. We split the $K$ regions of Cook county into two groups: ${\cal G}_s$ and ${\cal G}_d$, where ${\cal G}_s$ is the set of regions with sparse demand (such as the suburbs) and ${\cal G}_d$ consists of the rest of the regions\footnote{We set ${\cal G}_s=\{2, 3, 8, 9, 11, 12\}$ and ${\cal G}_d=\{1, 4, 5, 6, 7, 10\}$ in our computational study \citep{PeterPon2023}.}. We choose a common threshold $\xi>0$ for regions in ${\cal G}_s$.
Each day, the system manager checks whether at least one region in ${\cal G}_s$ has more than $\xi$ pending eviction orders. If so, she dispatches eviction teams according to two rules: (1) she sends one team to a region in ${\cal G}_s$, that is chosen randomly with a probability proportional to the number of pending orders in each region of ${\cal G}_s$ (in which case the team may be sent to a region that has a number of pending orders lower than $\xi$), and (2) she sends the remaining teams to randomly chosen regions in ${\cal G}_d$ with probabilities proportional to the number of pending orders in each region of ${\cal G}_d$. If the number of pending eviction orders is less than $\xi$ in all regions with sparse demand (i.e., in all regions in ${\cal G}_s$), then all teams are sent to randomly chosen regions in ${\cal G}_d$. 

Once a team is assigned to a region, the system manager solves the budgeted prize-collecting VRP to determine which orders will be enforced and what route will be taken. Now, we discuss how to set the prizes. Given that the CCSO does not explicitly prioritize eviction orders with a deadline over those without, it is natural assign the orders within the assigned region an equal prize. However, based on our observation from the CCSO data, the deadline information appears to affect enforcement planning. More specifically, we observe that eviction orders that are closer to their deadline are more likely to be served; see Appendix \ref{subapp:OrderAge}. To incorporate this, we propose the following approach to assign prizes to the pending orders, and calibrate its parameters using the CCSO data.

Consider the route design of an eviction team and let $r^\ast\in \{1, \ldots, K\}$ be the index of the chosen region (i.e., $\pi_{{\cal I}^{(k)}}(i)=0$ for all $k=1,\ldots,2K$ with $k\not =r^\ast$ and $i=1,\ldots,z_k$). Let $i$ be the index of a pending eviction order in that region. If it has a deadline, we set its prize randomly to the realization of a Bernoulli random variable that has mean $\tau_{{\cal I}^{(r^\ast)}}(i)/d_{{\cal I}^{(r^\ast)}}(i)$ for $i=1,\ldots,z_{r^\ast}$, i.e., the ratio of the order's waiting time to its deadline. The higher this ratio (and the lower the slack to serve that order), the more likely it is that the prize is one. On the other hand, if the eviction order does not have a deadline, then its prize is set to the realization of a Bernoulli random variable with mean $w$, which is calibrated using the CCSO dataset; see Appendix \ref{app:ModelCalibration} for further details. Given that prizes for all pending eviction orders in region $r^\ast$, one solves the budgeted prize-collecting VRP in that region to choose the eviction orders to enforce and the route to follow. 

\section{Computational Results}
\label{sec:Simulation}

Our study is calibrated using the CCSO data which we describe in Section \ref{subsec:Data}. The results are then presented in Section \ref{subsec:Numerical}.

\subsection{Data Description and Parameter Estimation}
\label{subsec:Data}

Our dataset consists of 73,683 eviction orders received by the CCSO starting January 1, 2015 until September 30, 2019. The CCSO divides Cook county into 12 zones as shown in Figure \ref{fig:OriginalZoneDesign_WithLabel}, which also highlights the location of the CCSO's main office, serving as the depot for eviction enforcement planning. In practice, eviction orders with associated deadlines may have a random deadline, varying up to 120 calendar days, rather than a common deadline $d$ as assumed in Section \ref{sec:BackgroundAndModel}. This results from potential communication delays occurring after the judge issues the eviction order until the CCSO team receives it. This is incorporated in the simulation study discussed in Section \ref{subsec:Numerical}. 
%The three policies discussed in Section \ref{sec:BenchmarkPolicies} all incorporate this information. 
\begin{figure}[htbp!]
	\centering
	\includegraphics[width=\textwidth]{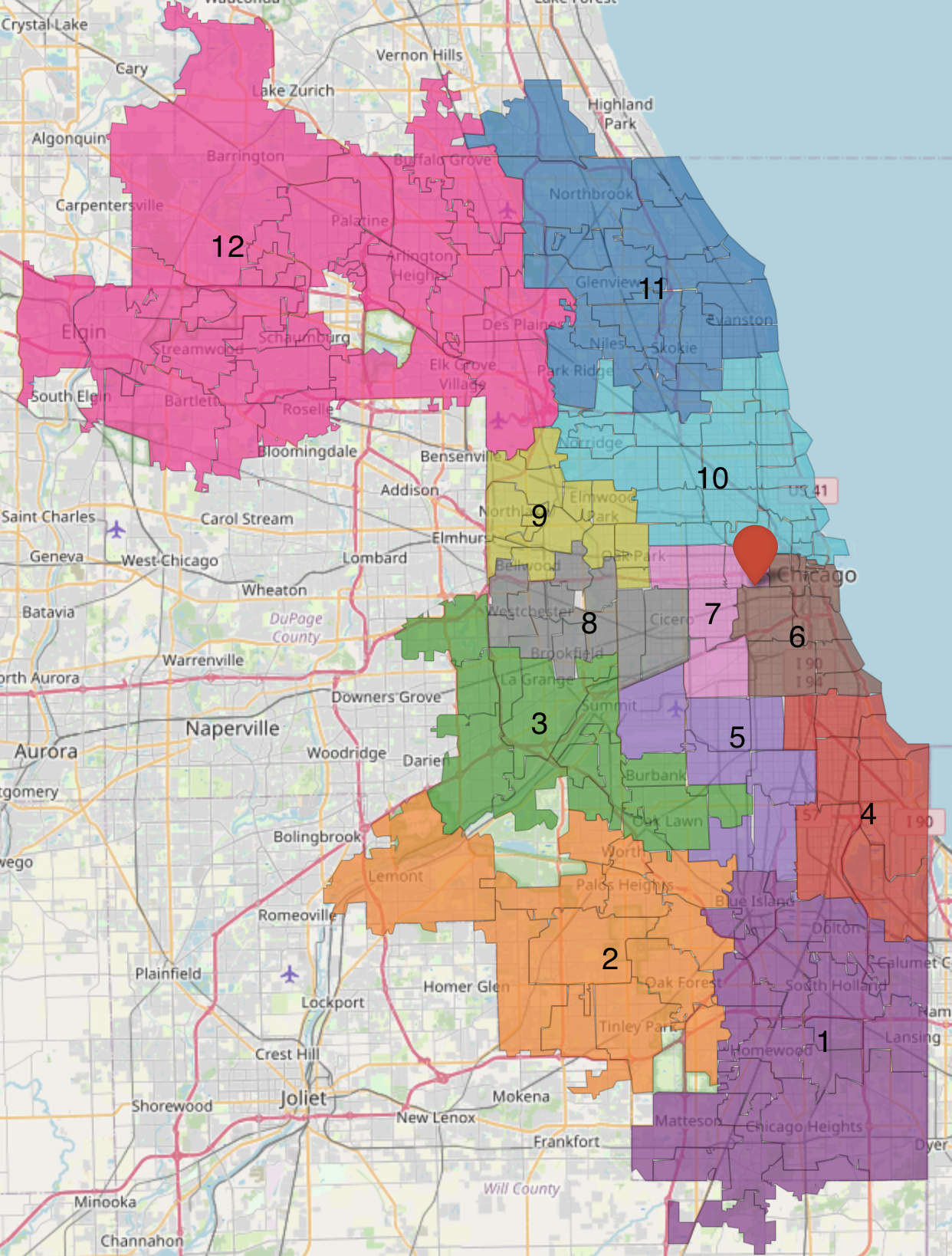}
	\caption{The 12-zone design from the CCSO. The red marker is the CCSO office (which serves as the depot in eviction enforcement planning). }
	\label{fig:OriginalZoneDesign_WithLabel}
\end{figure}

The dataset records the following information for each eviction order: (1) the received date, (2) the address, (3) the enforcement completion date or cancellation date, depending on whether the order is served or canceled, (4) the start and end times of the enforcement (if applicable), (5) the identity of the serving deputy (if applicable), and (6) the deadline associated with the order (if applicable). Next, we describe how we estimate the following quantities using the CCSO dataset: (i) The arrival rates of eviction orders, (ii) The cancellation rates, (iii) The time required to enforce eviction orders once the team is on site, (iv) Travel times, (v) The average number of eviction teams and their availability, and (vi) The deadline for eviction orders.

\paragraph{Estimation of the order arrival process.}

As discussed in Section \ref{sec:BackgroundAndModel}, we assume the arrival process follows a $2K$-dimensional Gaussian process. To estimate the mean and covariance matrix of the arrival process for each class, we perform Maximum Likelihood Estimation (MLE) using the CCSO dataset. Specifically, we use the number of received orders per day for this estimation. 
Figure \ref{fig:SystemDemand} shows the histogram of the aggregate number of eviction orders received on a daily basis across all 12 zones, depending on whether the orders have a deadline or not. 
\begin{figure}[h!]
	\centering
	\begin{subfigure}{.5\textwidth}
		\centering
		\includegraphics[width=\linewidth]{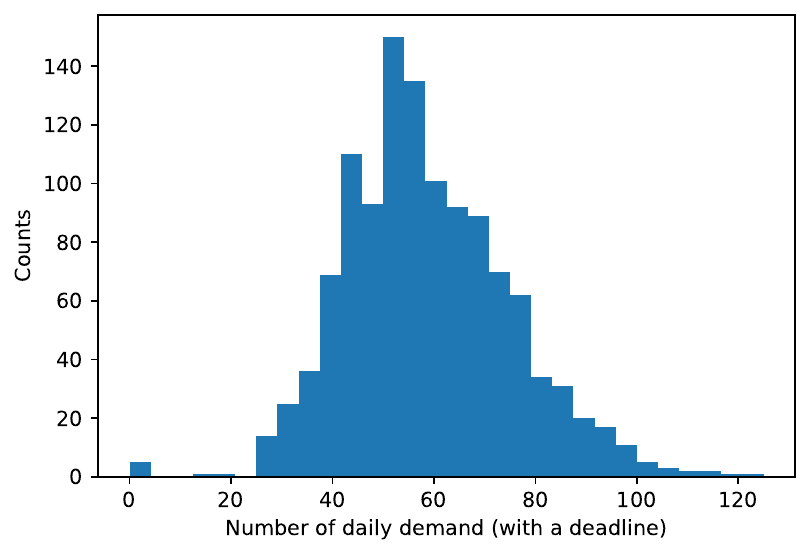}
		\caption{With a deadline}
		\label{fig:sub1}
	\end{subfigure}%
	\begin{subfigure}{.5\textwidth}
		\centering
		\includegraphics[width=\linewidth]{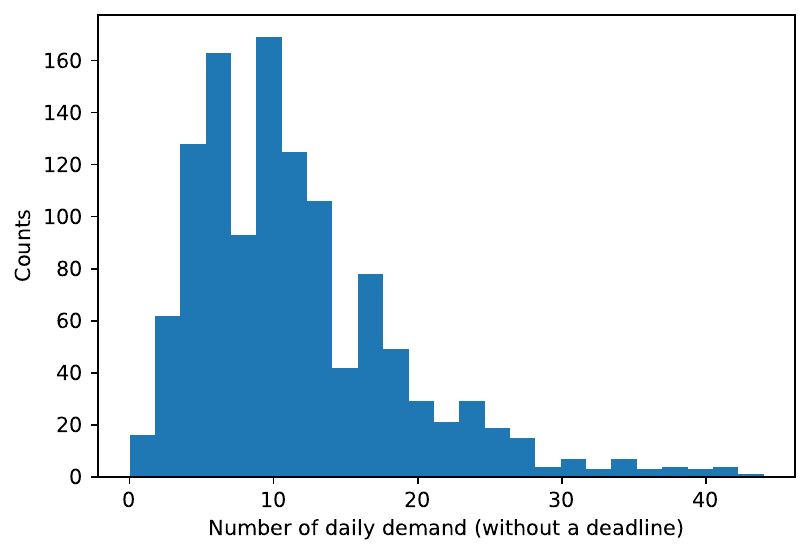}
		\caption{Without a deadline}
		\label{fig:sub2}
	\end{subfigure}
	\caption{Histogram of the number of daily received orders across all zones, depending on whether the eviction orders have a deadline or not. }
	\label{fig:SystemDemand}
\end{figure}

Table \ref{tab:SystemDemand_Stats} presents the MLE results for the estimated mean and standard deviation of the arrival process for the entire system. For a more detailed discussion of the estimation method and the results for each individual zone, see Appendix \ref{subapp:DemandEstimation}.
\begin{table}[h!]
	\centering
	%	\resizebox{\textwidth}{!}{ 
		\begin{tabular}{ c||c|c }
			\toprule
			Order type & Mean & Standard deviation \\  \midrule
			With a deadline & 58.68 & 16.43  \\
			Without a deadline & 11.63 & 7.16 \\
			\bottomrule
		\end{tabular} 
		%}
	\caption{Estimated mean and standard deviation of the total arrivals per day. }
	\label{tab:SystemDemand_Stats}
\end{table}

\vspace{-0.5cm}

\paragraph{Estimation of the order cancellation rate.}

For simplicity, we estimate a common cancellation rate.
Recall that the eviction orders received can be canceled before they are served. 
To estimate the cancellation rate, we apply a Kaplan-Meier estimator with bias correction, as detailed by \citet{StuteWang1994}. In our estimation, orders that are canceled are treated as uncensored observations, while those that are served are considered censored observations.
The estimated cancellation rate is 0.0080 per day. A detailed discussion of the method and the results are included in Appendix \ref{subapp:CancellationRateEstimation}.

\paragraph{Estimation of eviction enforcement duration (not including travel times).}

When a team of deputies arrives at the address of a pending eviction order, a non-negligible time period is required to enforce the order, i.e., the service time. We estimate the service time by calculating the difference between the recorded eviction start time and end time. Figure \ref{fig:ServiceTimeEstimation} shows the histogram of the service time recorded from the data. 
\begin{figure}[h!]
    \centering
    \includegraphics[scale=0.6]{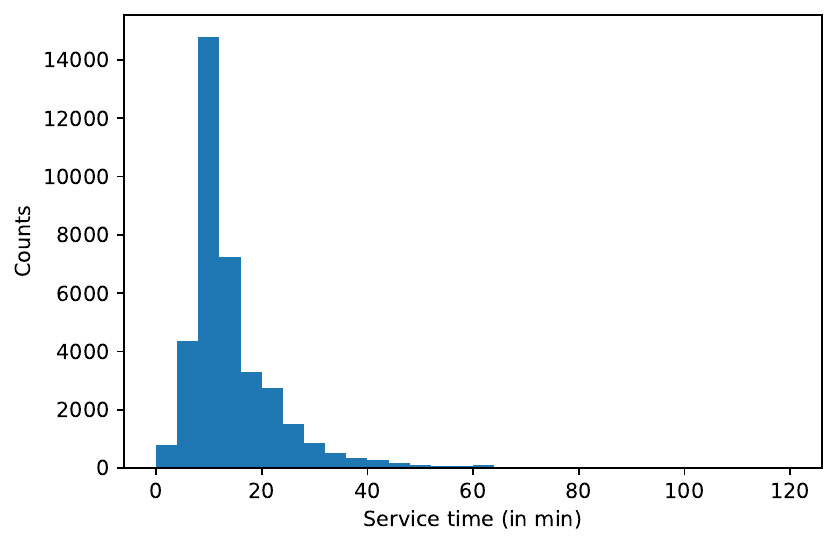}
    \caption{Empirical distribution of the service time of the eviction orders. The mean and standard deviation of the service time is 14.39 and 10.42 minutes, respectively.}
    \label{fig:ServiceTimeEstimation}
\end{figure}

To plan the route, we assume that each eviction order takes 14.39 minutes to enforce, based on the average service time derived from the CCSO dataset. When estimating the total working hours for each team in the simulation, we randomly select a service time for each served order using the empirical distribution of service times from the CCSO data, as shown in Figure \ref{fig:ServiceTimeEstimation}. A detailed discussion of the estimation is provided in Appendix \ref{subapp:ServiceTimeEstimation}.

\paragraph{Estimation of the average number of eviction teams and their availability.}

The number of eviction teams dispatched for daily enforcement assignments varies. Figure \ref{fig:NumDailyTeams} shows the histogram of the number of teams assigned each day in the CCSO dataset. For our simulation study, we assume the number of teams for each daily assignment is 4, which corresponds to the average number of teams observed in the data. Appendix \ref{subapp:NumTeamEstimation} provides further details for this estimation.  
\begin{figure}[h!]
    \centering
    \includegraphics[scale=0.6]{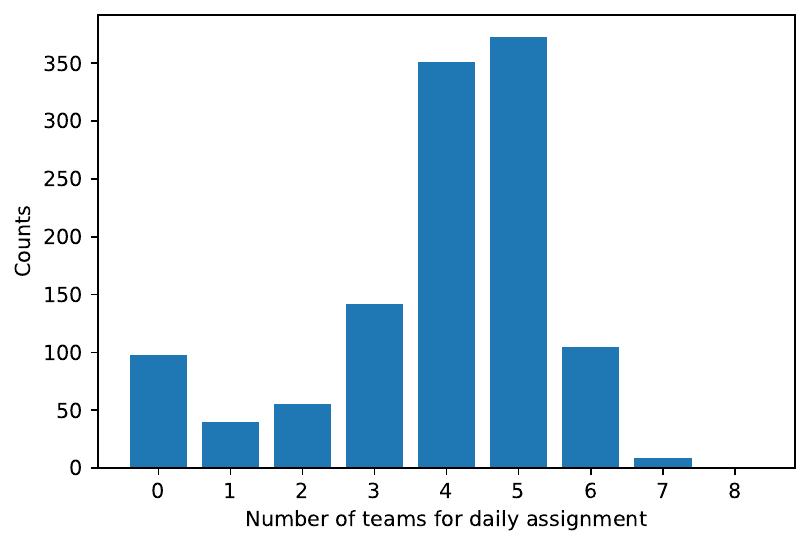}
    \caption{Number of the teams for daily assignments throughout the history. The mean, median, and standard deviation of the number of teams in a daily assignment is 3.8722, 4, and 1.6408, respectively.}
    \label{fig:NumDailyTeams}
\end{figure}

To estimate the daily working hours for each team, we first calculate the total number of working hours throughout the entire dataset. This information is obtained using the recorded eviction start and end time of the served eviction orders of each team. The estimated daily working hours are then determined by dividing the total working hours by the estimated average number of teams. Based on our calculations, each team is estimated to work approximately 5 hours per day\footnote{The deputies' daily working hours is less than the standard 8-hour workday, as they may engage in additional duties throughout the day, such as roll call, responding to unexpected 911 calls, returning their vehicles, and documenting their work summaries at the end of each shift.}. For a detailed discussion of the estimation, please refer to Appendix \ref{subapp:DailyWorkingHoursEstimation}.

\paragraph{Estimation of travel times.} 

The CCSO data does not include explicit information on travel times, but it does record the start and end times of each eviction order (for those that are enforced). While we can obtain the routes of the eviction enforcement from the CCSO data, we cannot directly use the time gap between the end of one order and the start of the next as the travel time for two key reasons: (1) In our proposed policy and the benchmark policies, it is possible that some routes not present in the CCSO data are recommended, leaving us with no information on travel times for those routes,
and (2) travel times can vary significantly throughout the day (e.g., rush hours may lead to longer travel times), a factor that cannot be inferred directly from the data. To address these issues, we use a simple linear regression to estimate the travel time.
Specifically, we first obtain the routes of the eviction enforcement from the CCSO data, treating each pair of consecutively enforced orders as a data point. Each data point has two attributes: (1) the distance difference, which is computed based on the geographical distance between the two orders, and (2) the travel time difference, which is calculated as the time gap between the eviction end time of the former order and the eviction start time of the latter order. Then the regression model is given by: 
\begin{align*}
    {travel\; time} &= \beta \times {distance},
\end{align*}
where $\beta$ is the parameter to be estimated. Our estimation yields $\beta=1.9792$, with travel time and distance measured in minutes and kilometers, respectively. We provide a detailed discussion of the travel time estimation in Appendix \ref{subapp:TravelTimeEstimation}.

\paragraph{Estimation of the deadline.} 
As discussed in Section \ref{sec:BackgroundAndModel}, the eviction orders of class $k$, where $k=1,\ldots, K$, have a deadline, which corresponds to the maximum time length until the order can be served. In the case of the Cook county, a deadline is typically set at 120 calendar days (which is approximately 85 workdays) after the judge issues the eviction order. However, there can be a delay between this approval and the receipt of the order by the CCSO for enforcement, which shorten the actual time duration until an order can be served. We estimate deadlines based on the duration between the received date and the assigned deadline from the data set. Figure \ref{fig:CommonDeadline} shows the distribution of these duration, where we focus on the number of workdays. For the simulation study, we generate each order's deadline randomly based on this empirical distribution.
\begin{figure}[h!]
    \centering
    \includegraphics[scale=0.6]{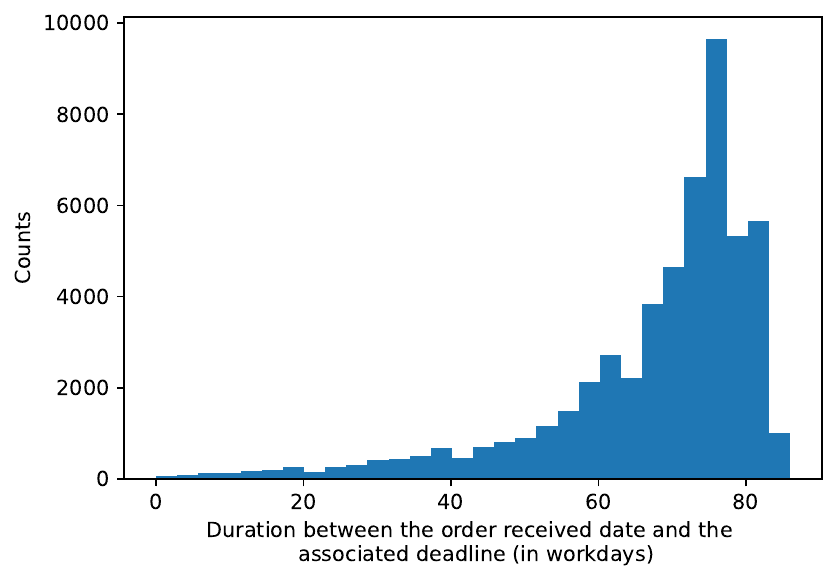}
    \caption{Distribution of the duration between the deadline and received date. The mean, median, and standard deviation is 69.17, 72, and 12.64, respectively.}
    \label{fig:CommonDeadline}
\end{figure}

\subsection{Simulation Study}
\label{subsec:Numerical}

In this section, we compare the performance of our proposed policy against the two benchmark policies introduced in Section \ref{sec:BenchmarkPolicies}. 
Table \ref{tab:PerformanceMeasure} provides a summary of the performance measures used for comparison and their descriptions; see Appendix \ref{app:PerformanceMeasures} further details.

\begin{table}[h!]
	\centering
%	\resizebox{\textwidth}{!}{
		\begin{tabular}{ p{0.17\textwidth} | p{0.7\textwidth} }
			\toprule
			Performance & Description \\  \midrule
            Missing deadline percentage & The percentage of the orders that are not served nor canceled prior to their deadline.  \\  \hline
            Cancellation percentage & The percentage of the orders that are canceled (prior to their deadline). We report this metric for orders with or without a deadline separately. \\  \hline
            Average number of orders waiting to be enforced & Number of orders that are waiting for service. We report this metric for orders with or without a deadline separately. \\  \hline
            Number of orders served daily & Average number of orders that are served per day. We report this metric for orders with or without a deadline separately. \\ 
			\bottomrule
		\end{tabular}
%	}
	\caption{Performance metrics. }
	\label{tab:PerformanceMeasure}
\end{table}

In the simulation study, one can vary several operational levers. First, the cost parameters $c_1, c_2$, and $p$ should be chosen so that the relative importance of reducing the number of pending orders and avoiding missing the deadline are reflected. Then, for the sake of computational efficiency, it is more practical to focus on a subset of pending orders with longer pending duration during daily planning, rather than considering all pending orders. The performance of the policies considered can depend on the specific size of this selected subset. In order to counteract this effect, we opt for a substantial subset, such as 500 orders that have  the highest prizes. We determine this number so that results do not change significantly beyond this value.
We refer to this number as the service threshold. In the remainder of this section, we first compare our proposed policy with the benchmark policies and then investigate on how different operational factors affect the performance of the proposed policy. 
In the simulation, we set the run length of all the experiments as 2000 simulated workdays, i.e., about eight years.

\subsubsection{Comparison Between Proposed Policy and Benchmarks}

Without loss of generality, we set $p=1$. Since $c_1$ represents the holding cost per pending day for each order with a deadline, and the deadline ranges approximately from 0 to 85 calendar days (see Figure \ref{fig:CommonDeadline}), it is reasonable to assume $c_1<p/85$. In our experiments, we set $c_1=p/200=0.005$. The parameter $c_2$ represents the holding cost per day for orders without a deadline. The relative magnitude of $c_1$ and $c_2$ affect the allocation of service effort between orders with and without deadlines. A higher $c_2$ means that the holding cost for orders without a deadline increases, leading to more such orders being served. For our experiments, we let $h=c_2/c_1$, the ratio of the daily holding cost for orders with and without deadlines, and consider $h\in \{1, 2, \ldots, 10\}$. Note that $h=1$ implies equal holding costs for both orders with and without a deadline. Since orders with missed deadlines incur a penalty cost $p$, a great majority of the service effort is allocated to orders with deadlines in this case. Therefore, we do not consider $h<1$. In addition, we select $\xi=190$ and $w=0.3$ for the threshold-based policy and set $\tilde{d}=143$ for the urgency-based policy. These parameter values are calibrated using the CCSO dataset; see Appendix \ref{app:ModelCalibration} for further details.

\paragraph{Effect of the service effort allocation on the percentage of eviction orders that miss their deadline.}

A larger holding cost ratio $h$ increases the service effort allocated to the eviction orders without a deadline. Figure \ref{fig:ProposedPerformance_DiffHolding} shows how the number of eviction orders with (Figure \ref{subfig:ProposedPerformance_DiffHolding_WithDie}) and without (Figure \ref{subfig:ProposedPerformance_DiffHolding_WithoutDie}) a deadline served changes as one varies $h$. It also shows how the percentage of eviction orders that miss their deadline varies (Figure \ref{subfig:ProposedPerformance_DiffHolding_MissingDie}).

\begin{figure}[h!]
\begin{minipage}{.5\linewidth}
\centering
\subfloat[Number of orders served daily (with a deadline)]{\label{subfig:ProposedPerformance_DiffHolding_WithDie}\includegraphics[width=0.9\linewidth]{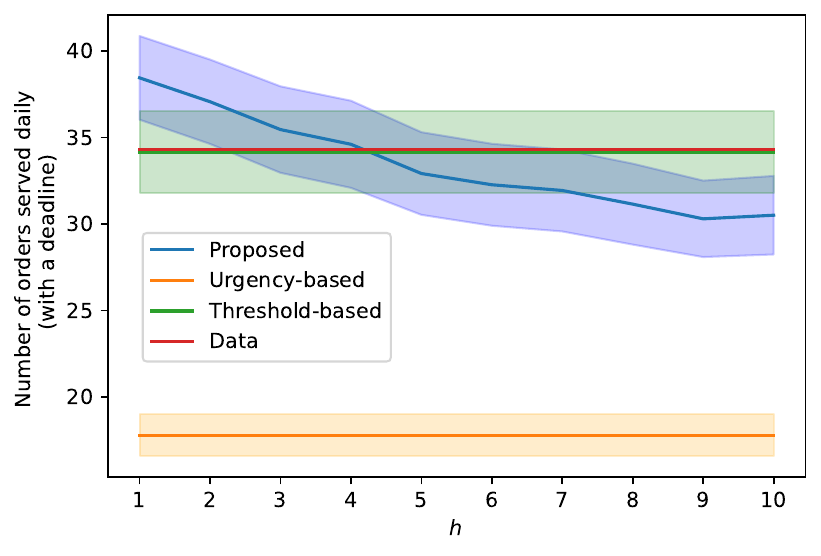}}
\end{minipage}%
\begin{minipage}{.5\linewidth}
\centering
\subfloat[Number of orders served daily (without a deadline)]{\label{subfig:ProposedPerformance_DiffHolding_WithoutDie}\includegraphics[width=0.9\linewidth]{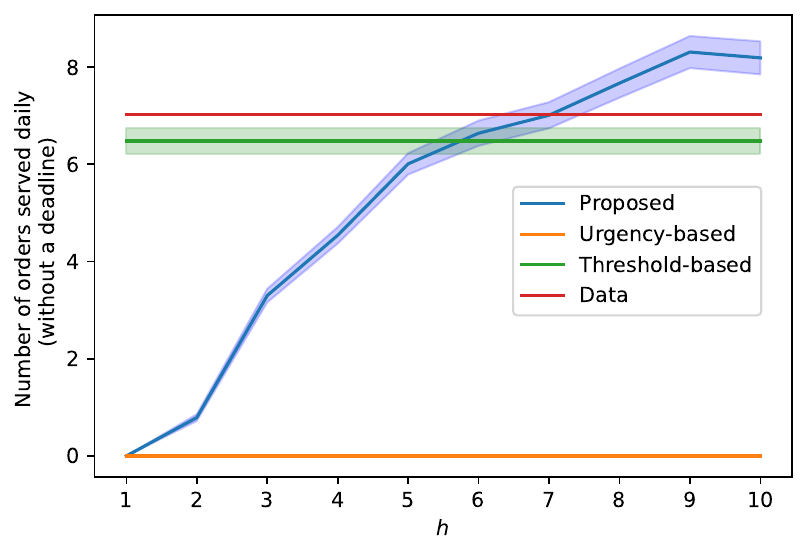}}
\end{minipage}\par\medskip
\centering
\subfloat[Missing deadline
percentage]{\label{subfig:ProposedPerformance_DiffHolding_MissingDie}\includegraphics[width=0.45\linewidth]{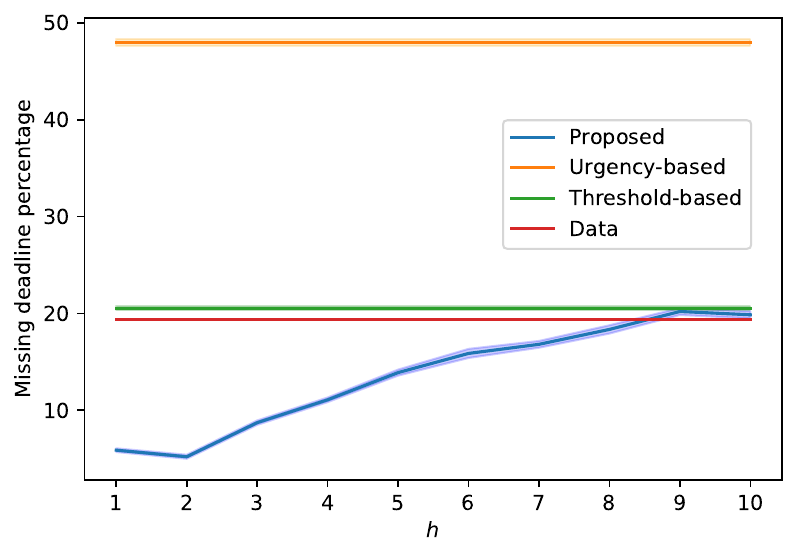}}
\caption{The number of orders served daily and missing deadline
percentage for the two benchmark policies, the data, and the proposed policy in terms of $h\in \{1, 2, \ldots, 10\}$. The shaded area represents the 95\% confidence interval estimated by 10 macro-replications through our simulation. Note that some results have narrow confidence intervals that are not visible in the figure. }
\label{fig:ProposedPerformance_DiffHolding}
\end{figure}

We observe that the number of eviction orders enforced per day is close to that in the CCSO data under the proposed policy when $h=4$, but it significantly lowers the percentage of eviction orders that miss their deadline (by about 73\%). Thus the holding cost parameter $h$ helps one balance the concerns of missing the deadline versus giving sufficient service attention to eviction orders without a deadline.

\paragraph{Trade-off between the percentage of eviction orders that miss their deadline and the cancellation percentage for eviction orders without a deadline.}

As the holding cost ratio parameter $h$ increases, the service effort allocated to the eviction orders with a deadline decreases. As one would expect, this leads to larger percentage of eviction orders missing their deadline and a larger percentage of the orders with a deadline being canceled. More interestingly, Figure \ref{fig:MissDieAndCancellation_WithoutDie} shows the trade off (as $h$ varies) between the missing deadline percentage and the cancellation percentage (for orders without a deadline) as well as the average age of those orders before cancellation.

\begin{figure}[h!]
\begin{subfigure}[h]{0.45\linewidth}
\includegraphics[width=\linewidth]{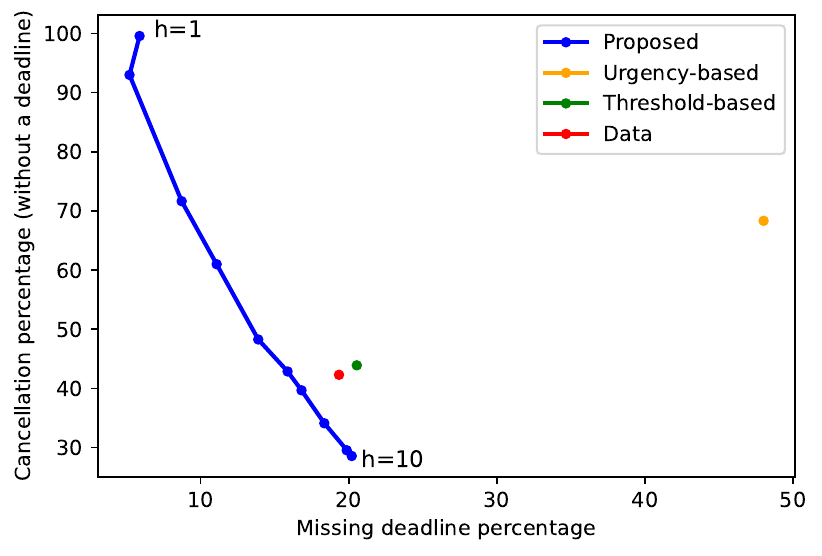}
\caption{Cancellation
percentage (without deadline)}
\end{subfigure}
\hfill
\begin{subfigure}[h]{0.45\linewidth}
\includegraphics[width=\linewidth]{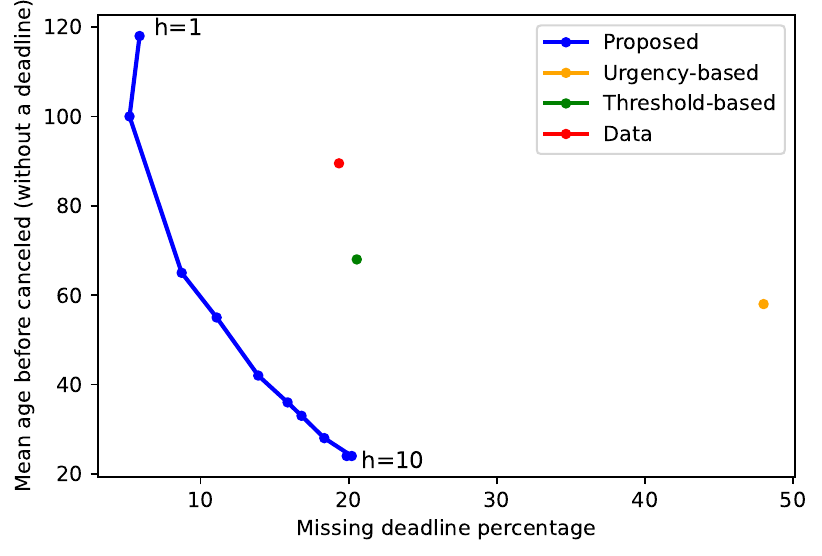}
\caption{Age before canceled (without deadline)}
\end{subfigure}%
\caption{Trade-off curves for the missing deadline percentage against the cancellation percentage (without deadline) and the age of the orders (without deadline) before canceled for the proposed policy with different holding cost, the urgency-based policy, and the threshold-based policy.}
\label{fig:MissDieAndCancellation_WithoutDie}
\end{figure}

As mentioned in Section \ref{sec:BackgroundAndModel}, cancellations are not necessarily an undesirable outcome, as certain orders might be resolved by reaching an agreement between the tenant and the landlord. 
We further include the trade-off curve of the percentage of cancellation between the orders with and without a deadline in Figure \ref{fig:WithAndWithoutDie_Cancellation}.
%That being said, the steep nature of the trade off curves in Figures \ref{fig:MissDieAndCancellation_WithoutDie} and \ref{fig:WithAndWithoutDie_Cancellation} can guide service effort allocation decision in practice. 

\begin{figure}[h!]
    \centering
    \includegraphics[width=0.5\linewidth]{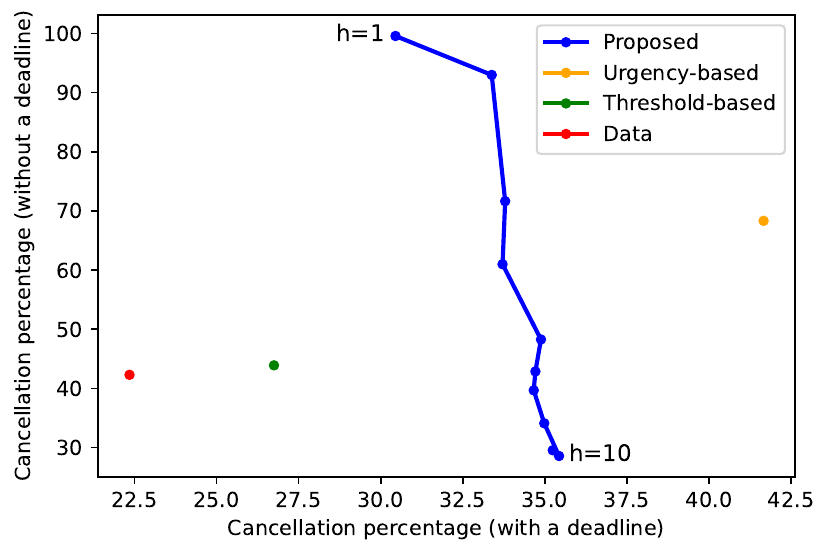}
    \caption{Relationship between the orders with and without deadlines in terms of the cancellation percentage.}
    \label{fig:WithAndWithoutDie_Cancellation}
\end{figure}

\paragraph{Trade-off between missing deadline percentage and the number of orders pending.} 

As the holding cost ratio $h$ increases, more resources are diverted to enforcing the eviction orders without a deadline. Clearly, both the percentage of eviction orders missing their deadline and their number in the system increase in that case. Figure \ref{subfig:Tradeoff_MissDie_NumPending_WithoutDie} shows the trade-off between the percentage missing their deadline and the number of pending orders without a deadline. Similarly, Figure \ref{subfig:Tradeoff_NumPending} shows the trade-off between the number of pending eviction orders of two types (with or without a deadline) as $h$ varies. 
%Once again, the sharp trade-off displayed in Figure \ref{fig:MissDieAndNumPending} are noteworthy. 

\begin{figure}[h!]
\begin{subfigure}[h]{0.45\linewidth}
\includegraphics[width=\linewidth]{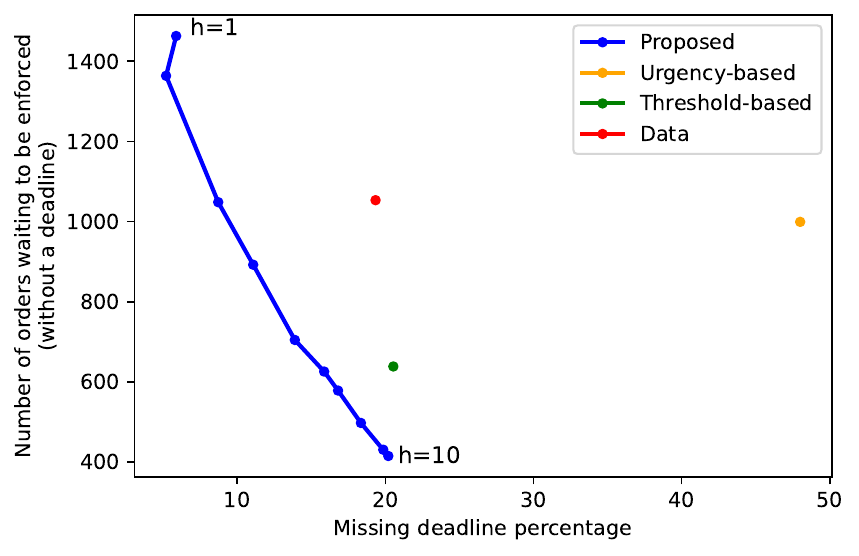}
\caption{Trade-off for missing deadline percentage}
\label{subfig:Tradeoff_MissDie_NumPending_WithoutDie}
\end{subfigure}
\hfill
\begin{subfigure}[h]{0.45\linewidth}
\includegraphics[width=\linewidth]{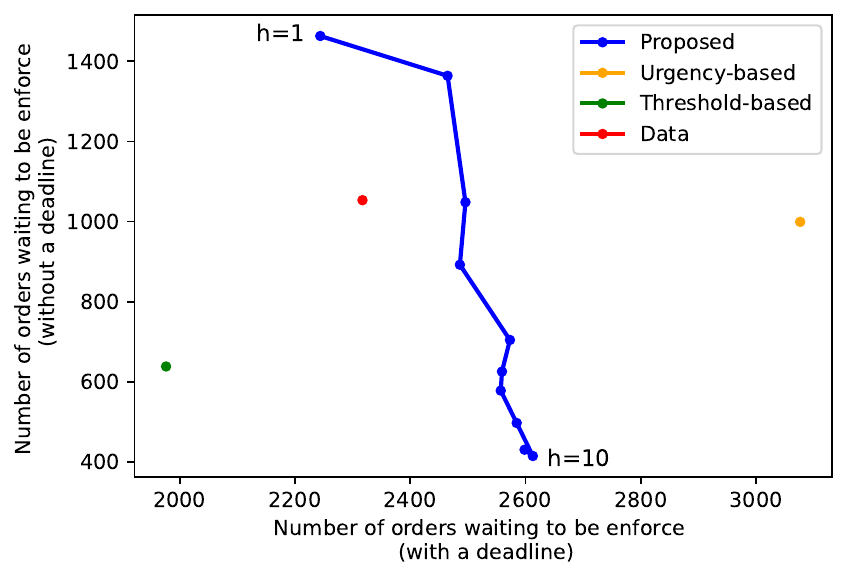}
\caption{Trade-off for number of orders waiting to be enforced (with deadline)}
\label{subfig:Tradeoff_NumPending}
\end{subfigure}%
\caption{Trade-off between the number of orders without a deadline waiting to be enforced with respect to the percentage of orders that miss their deadline and the number of pending orders with a deadline.}
\label{fig:MissDieAndNumPending}
\end{figure}

In summary, as more service effort is diverted from the eviction orders with a deadline to those without a deadline (by increasing $h$), we observe the performance metrics (cancellation percentage, missing deadline percentage, and the number of pending orders) for the orders with a deadline gets slightly worse, whereas the performance metrics for the orders without a deadline improve dramatically. This is reflected by the steep nature of the trade-off curves in Figures \ref{fig:MissDieAndCancellation_WithoutDie}--\ref{fig:MissDieAndNumPending}, and it can guide decision making in practice.

\subsubsection{Counterfactual Analysis}
\label{subsubsec:Counterfactual}

In this section, we conduct the counterfactual analysis of three counterfactual scenarios that vary the number of eviction teams, the number of daily working hours, or the deadline duration. In our experiments, we consider $h=4$ for the proposed policy, because it yields a similar number of daily served orders compared with that of the CCSO data.

As one would expect, increasing service capacity by increasing the number of eviction teams leads to improvements across the board. Figures \ref{fig:Counterfactual_NumVehicles_MissingDeadline}--\ref{fig:Counterfactual/Counterfactual_NumVehicles_NumPending} quantify the magnitudes of these improvements, which, of course, should be weighted against the capacity investment costs. 
In general, for all three policies, we observe that the percentage of missing deadline orders decreases, whereas the percentage of cancellation decreases, the number of daily served orders increases, and the number of daily pending orders decreases, as the number of vehicles increases. 

\begin{figure}[h!]
    \centering
    \includegraphics[width=0.45\linewidth]{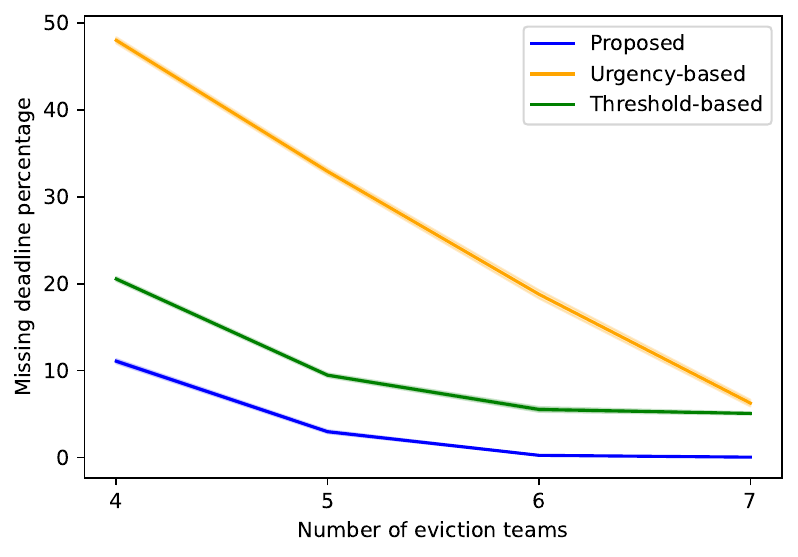}
    \caption{Missing deadline percentage as a function of eviction teams.}
    \label{fig:Counterfactual_NumVehicles_MissingDeadline}
\end{figure}

\begin{figure}[h!]
\centering
\begin{subfigure}{.5\textwidth}
  \centering
  \includegraphics[width=.9\linewidth]{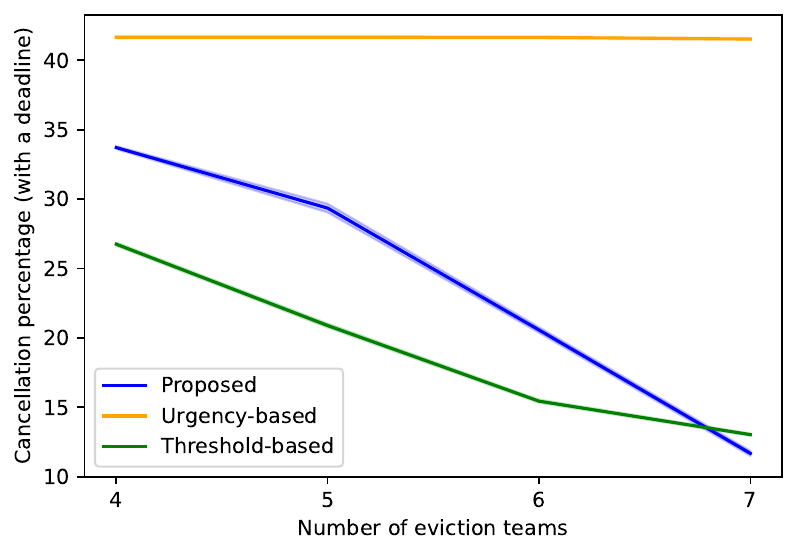}
  \caption{With a deadline}
  \label{fig:sub1}
\end{subfigure}%
\begin{subfigure}{.5\textwidth}
  \centering
  \includegraphics[width=.9\linewidth]{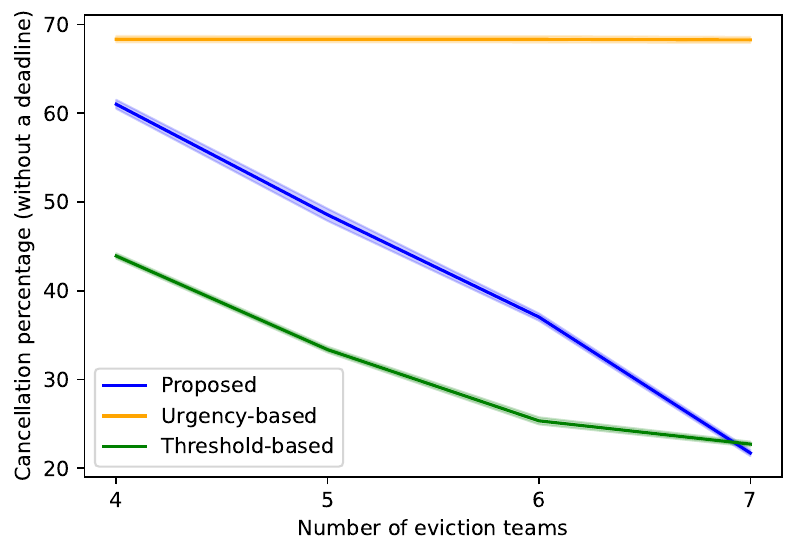}
  \caption{Without a deadline}
  \label{fig:sub2}
\end{subfigure}
\caption{Cancellation percentage as a function of eviction teams.}
\label{fig:test}
\end{figure}

\begin{figure}[h!]
\centering
\begin{subfigure}{.5\textwidth}
  \centering
  \includegraphics[width=.9\linewidth]{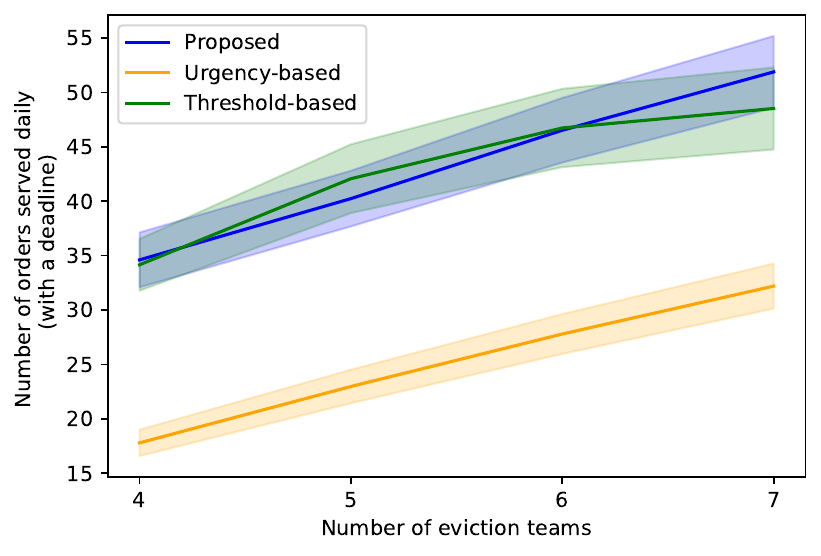}
  \caption{With a deadline}
  \label{fig:sub1}
\end{subfigure}%
\begin{subfigure}{.5\textwidth}
  \centering
  \includegraphics[width=.9\linewidth]{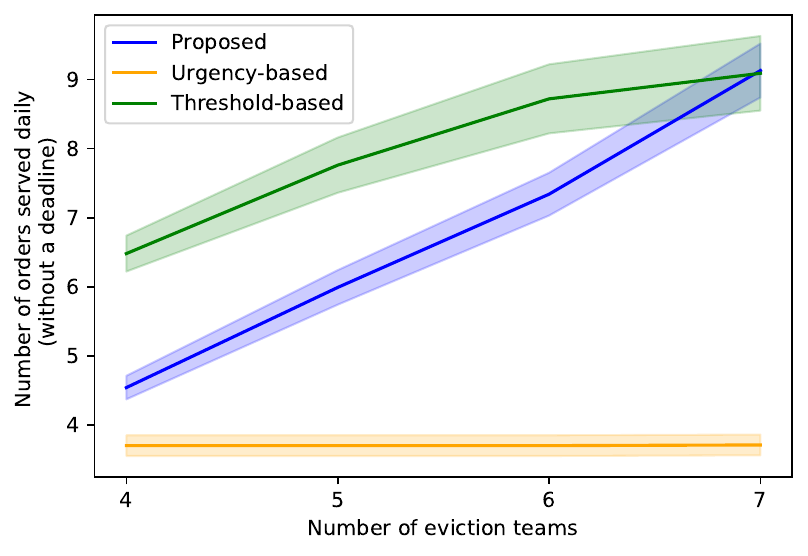}
  \caption{Without a deadline}
  \label{fig:sub2}
\end{subfigure}
\caption{Number of orders served daily as a function of eviction teams.}
\label{fig:test}
\end{figure}

\begin{figure}[h!]
\centering
\begin{subfigure}{.5\textwidth}
  \centering
  \includegraphics[width=.9\linewidth]{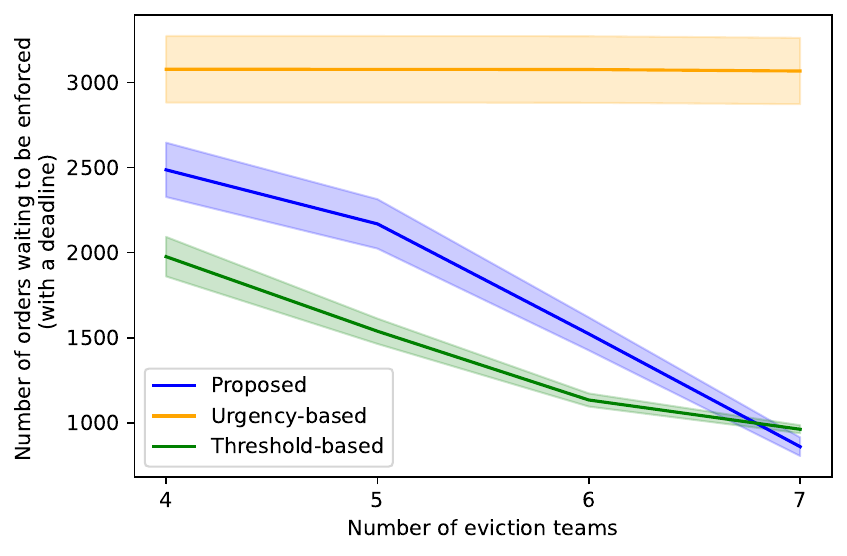}
  \caption{With a deadline}
  \label{fig:sub1}
\end{subfigure}%
\begin{subfigure}{.5\textwidth}
  \centering
  \includegraphics[width=.9\linewidth]{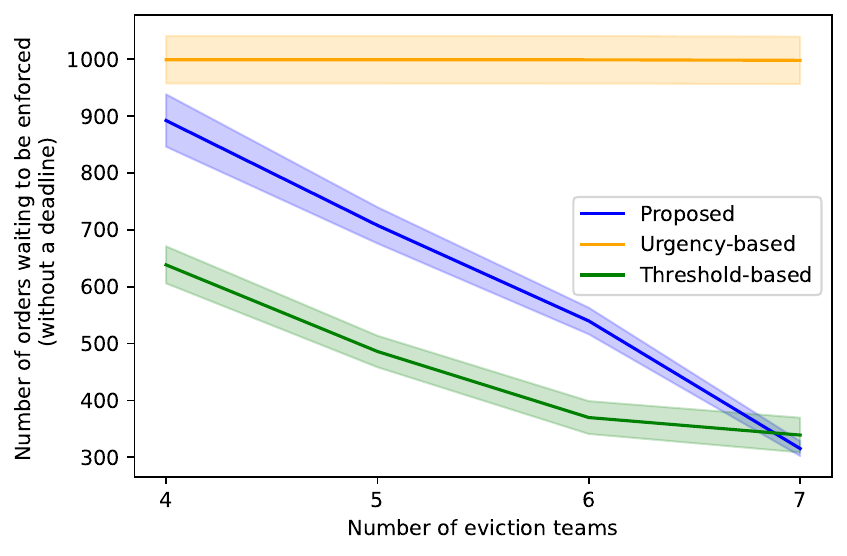}
  \caption{Without a deadline}
  \label{fig:sub2}
\end{subfigure}
\caption{Number of orders waiting to be served in terms as a function of eviction teams.}
\label{fig:Counterfactual/Counterfactual_NumVehicles_NumPending}
\end{figure}

Next, we consider a more granular adjustment to the service capacity. Namely, we study the effect of increasing the number of hours worked in a day for each vehicle, and observe a similar effect as when the number of eviction teams increases; see Figures \ref{fig:Counterfactual_NumWorkingHours_MissingDeadline}--\ref{fig:Counterfactual_NumWorkingHours_NumPending}.

\begin{figure}[h!]
    \centering
    \includegraphics[width=0.45\linewidth]{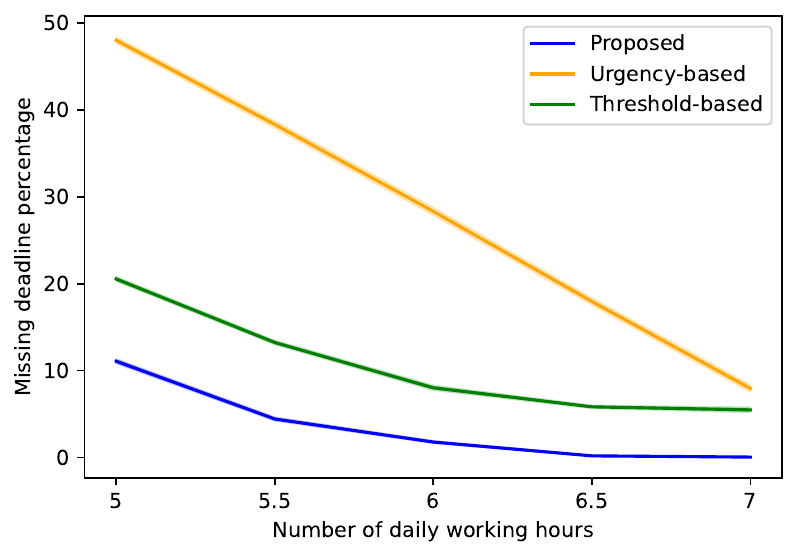}
    \caption{Missing deadline percentage as a function of daily working hours.}
    \label{fig:Counterfactual_NumWorkingHours_MissingDeadline}
\end{figure}

\begin{figure}[h!]
\centering
\begin{subfigure}{.5\textwidth}
  \centering
  \includegraphics[width=.9\linewidth]{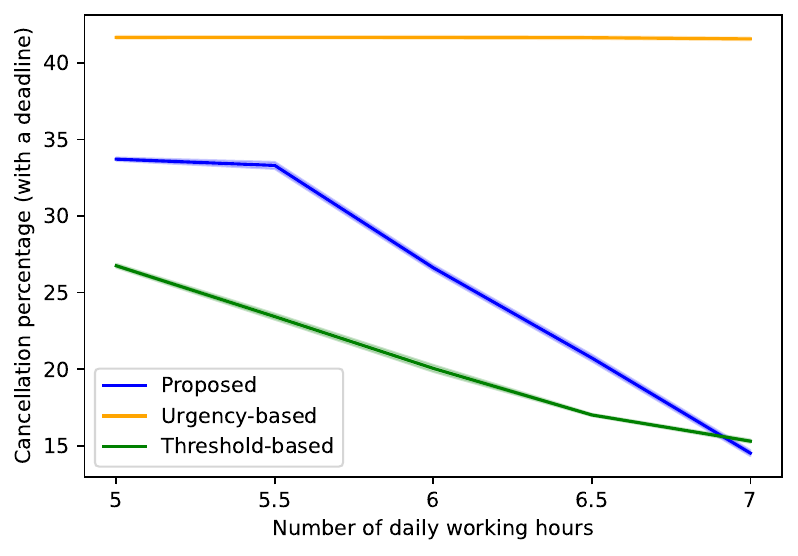}
  \caption{With a deadline}
  \label{fig:sub1}
\end{subfigure}%
\begin{subfigure}{.5\textwidth}
  \centering
  \includegraphics[width=.9\linewidth]{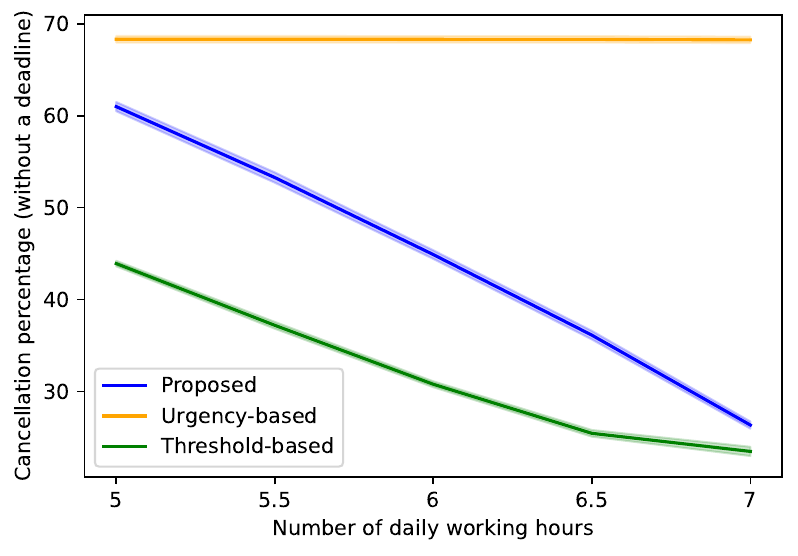}
  \caption{Without a deadline}
  \label{fig:sub2}
\end{subfigure}
\caption{Cancellation percentage as a function of daily working hours.}
\label{fig:test}
\end{figure}

\begin{figure}[h!]
\centering
\begin{subfigure}{.5\textwidth}
  \centering
  \includegraphics[width=.9\linewidth]{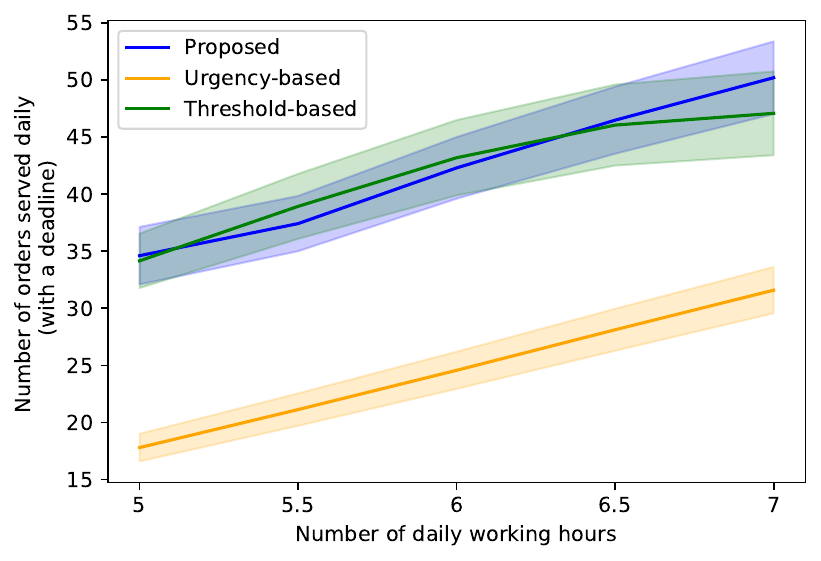}
  \caption{With a deadline}
  \label{fig:sub1}
\end{subfigure}%
\begin{subfigure}{.5\textwidth}
  \centering
  \includegraphics[width=.9\linewidth]{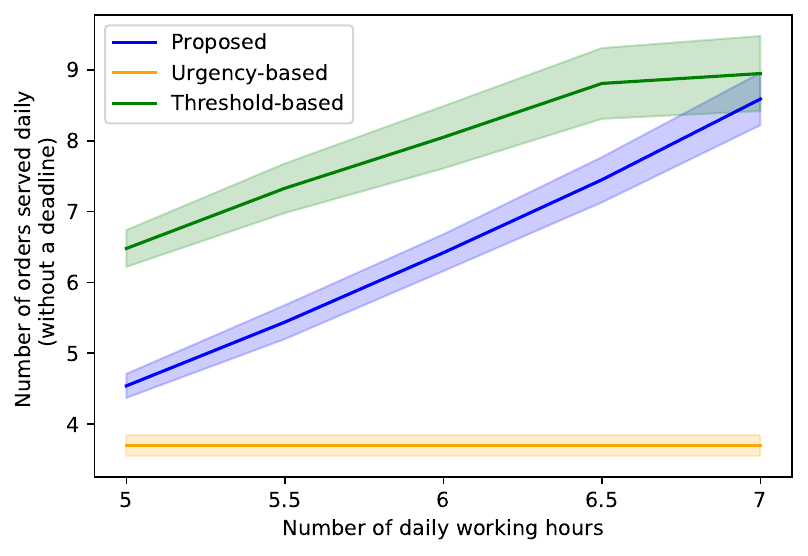}
  \caption{Without a deadline}
  \label{fig:sub2}
\end{subfigure}
\caption{Number of orders served daily as a function of daily working hours.}
\label{fig:Counterfactual_NumWorkingHours_NumServed}
\end{figure}

\begin{figure}[h!]
\centering
\begin{subfigure}{.5\textwidth}
  \centering
  \includegraphics[width=.9\linewidth]{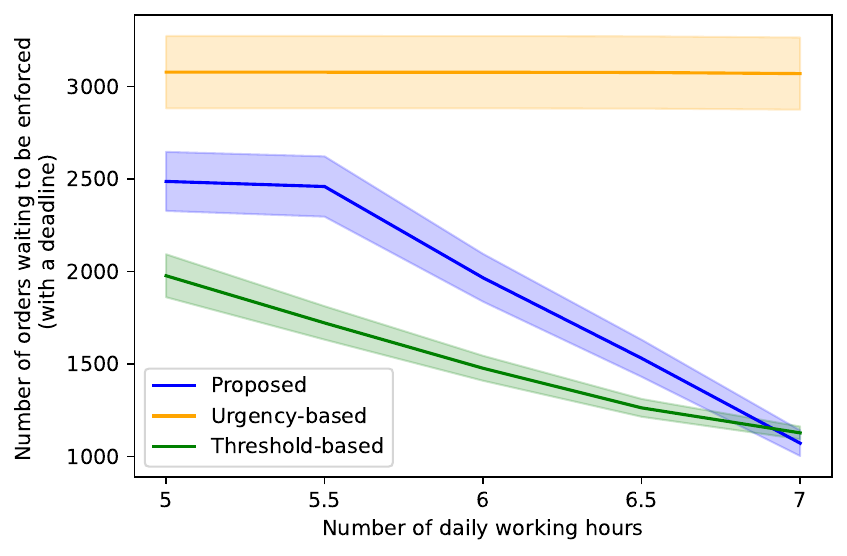}
  \caption{With a deadline}
  \label{fig:sub1}
\end{subfigure}%
\begin{subfigure}{.5\textwidth}
  \centering
  \includegraphics[width=.9\linewidth]{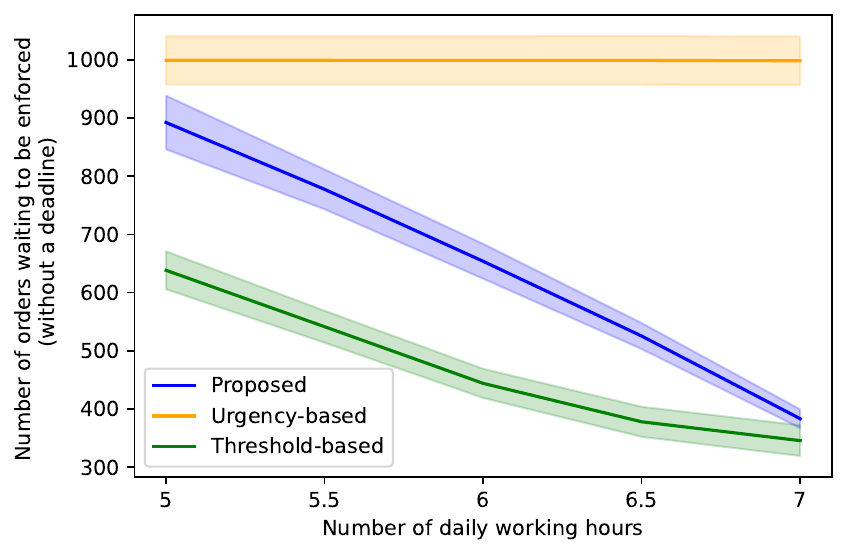}
  \caption{Without a deadline}
  \label{fig:sub2}
\end{subfigure}
\caption{Number of orders waiting to be served as a function of daily working hours.}
\label{fig:Counterfactual_NumWorkingHours_NumPending}
\end{figure}

Finally, we present the effects of extending the deadline by an additional 10, 20, and 30 days; see Figures \ref{fig:Counterfactual_LongerDeadline_MissingDeadline}--\ref{fig:Counterfactual_LongerDeadline_NumPending}. We observe that the percentage of missing deadline decreases as the deadline is extended longer. On the other hand, the cancellation percentage for the orders with a deadline increases as the deadline is extended longer, which is expected as a longer deadline gives those orders a higher chance to be canceled. Another interesting finding is that the proposed policy results in a higher number of daily served orders without a deadline as the deadline of other orders is extended longer. This is due to the fact that the proposed policy assigns a higher prize to the orders with a deadline when the deadline approaches. When the deadline is extended, it assigns them a higher prize later that results in  shifting the service effort slightly to the orders without a deadline.

\begin{figure}[h!]
    \centering
    \includegraphics[width=0.45\linewidth]{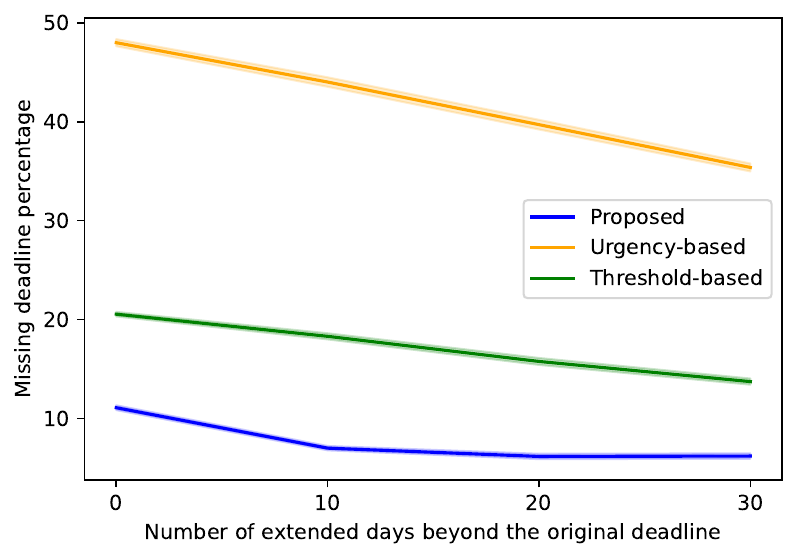}
    \caption{Missing deadline percentage as a function
    of daily working hours.}
    \label{fig:Counterfactual_LongerDeadline_MissingDeadline}
\end{figure}

\begin{figure}[h!]
\centering
\begin{subfigure}{.5\textwidth}
  \centering
  \includegraphics[width=.9\linewidth]{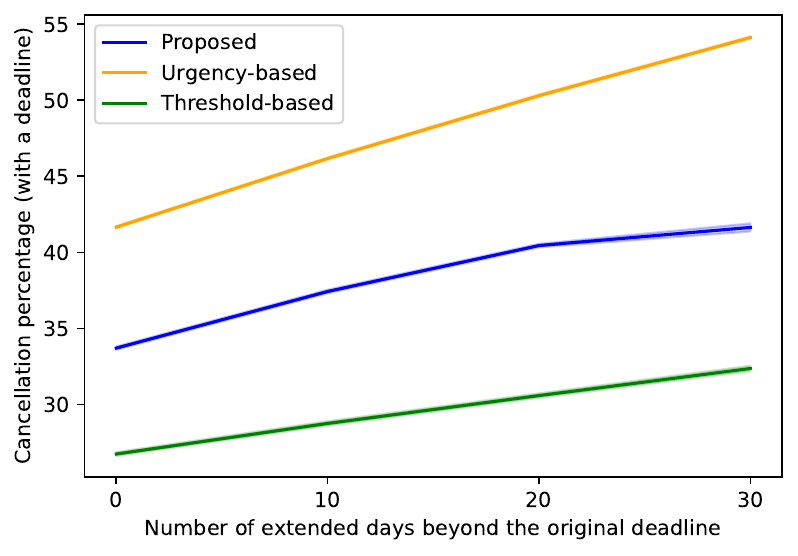}
  \caption{With a deadline}
  \label{fig:sub1}
\end{subfigure}%
\begin{subfigure}{.5\textwidth}
  \centering
  \includegraphics[width=.9\linewidth]{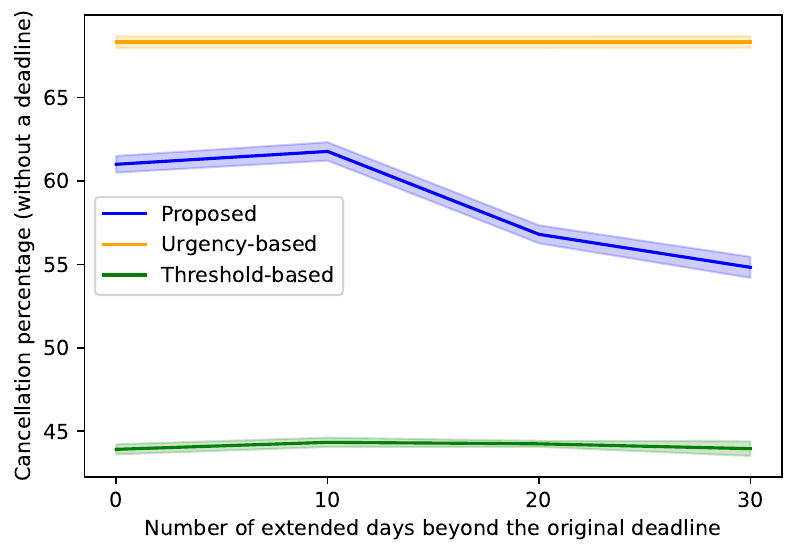}
  \caption{Without a deadline}
  \label{fig:sub2}
\end{subfigure}
\caption{Cancellation percentage as a function of daily working hours.}
\label{fig:test}
\end{figure}

\begin{figure}[h!]
\centering
\begin{subfigure}{.5\textwidth}
  \centering
  \includegraphics[width=.9\linewidth]{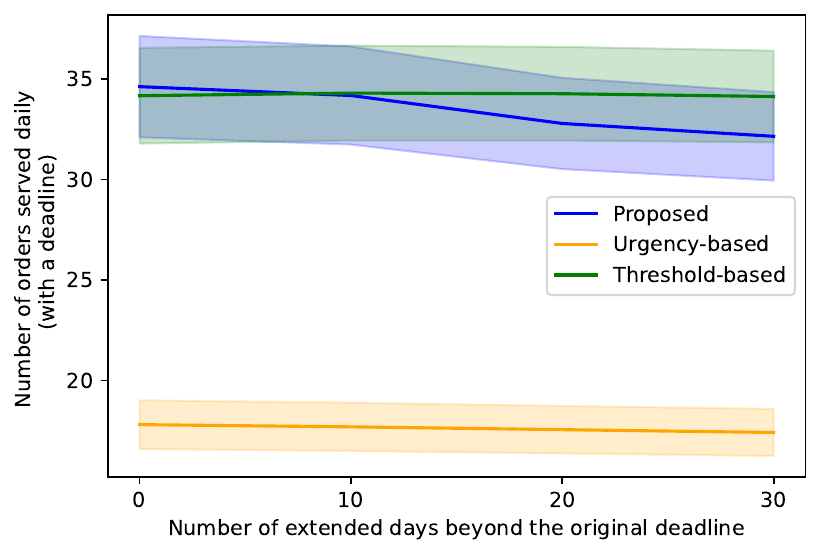}
  \caption{With a deadline}
  \label{fig:sub1}
\end{subfigure}%
\begin{subfigure}{.5\textwidth}
  \centering
  \includegraphics[width=.9\linewidth]{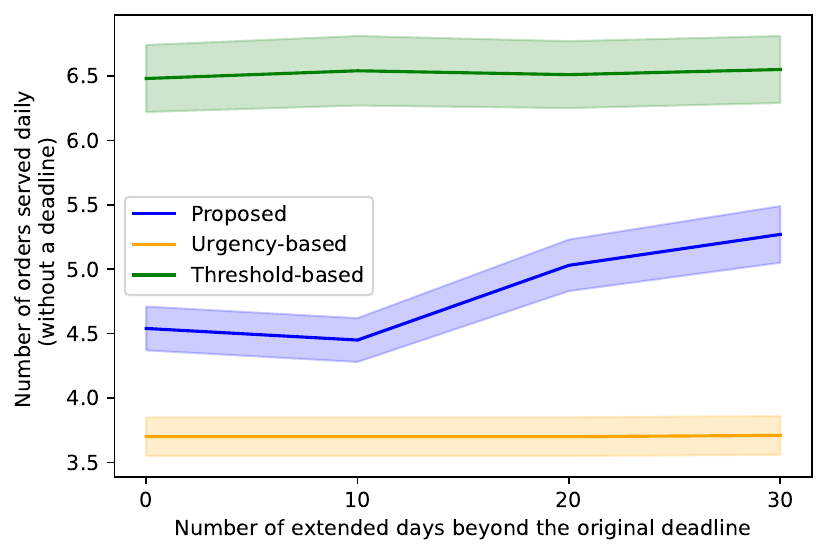}
  \caption{Without a deadline}
  \label{fig:sub2}
\end{subfigure}
\caption{Number of orders served daily as a function of daily working hours.}
\label{fig:Counterfactual_LongerDeadline_NumServed}
\end{figure}

\begin{figure}[h!]
\centering
\begin{subfigure}{.5\textwidth}
  \centering
  \includegraphics[width=.9\linewidth]{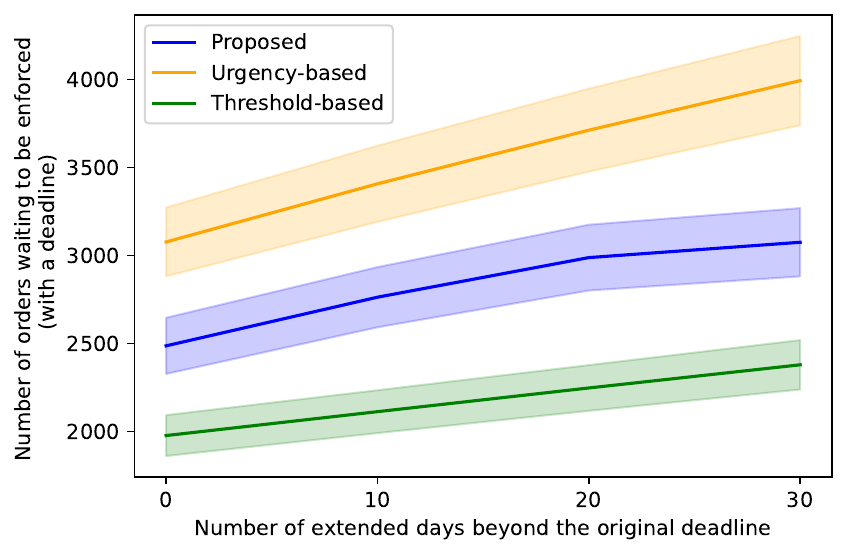}
  \caption{With a deadline}
  \label{fig:sub1}
\end{subfigure}%
\begin{subfigure}{.5\textwidth}
  \centering
  \includegraphics[width=.9\linewidth]{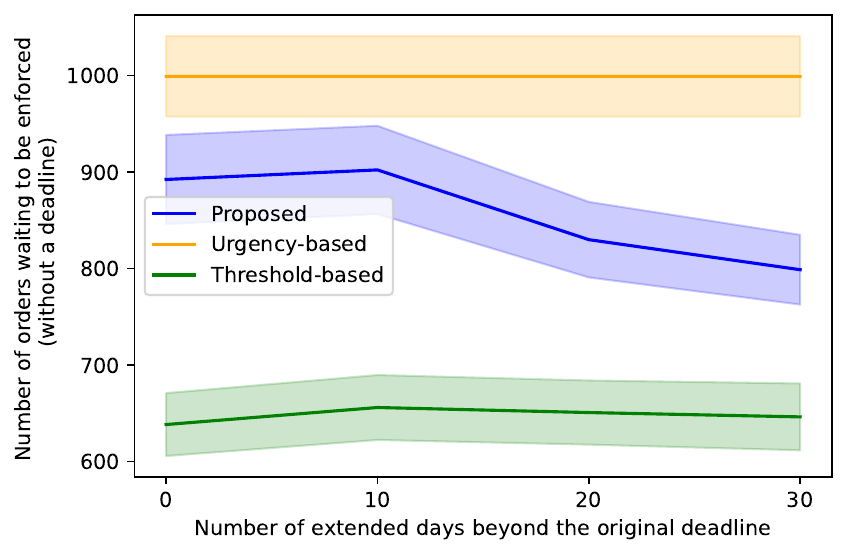}
  \caption{Without a deadline}
  \label{fig:sub2}
\end{subfigure}
\caption{Number of orders waiting to be enforced in terms of different number of daily working hours.}
\label{fig:Counterfactual_LongerDeadline_NumPending}
\end{figure}

\section{Concluding Remarks}

Our study is motivated by the eviction enforcement operations and the data from the CCSO, where the system manager must balance geographical coverage (as pending orders span the entirety of Cook County) and time constraints (as most eviction orders come with an enforcement deadline) when planning daily eviction operations, including the selection of orders to serve and the route to follow. Additionally, daily orders arrive with uncertainty, and cancellations are possible, which should be considered so that the daily decision is not myopic. When facing the trade-off between ``equity'' and ``efficiency'' in eviction enforcement planning, a FCFS policy, which emphasizes equity, may result in serving geographical dispersed orders on each day, thereby wasting resources. Conversely, focusing mainly on nearby orders to improve efficiency can increase the percentage of orders missing their deadlines. 

To address this trade-off, we propose a policy and compare it with two benchmark policies and with the CCSO data. 
Our proposed approach selects orders and routes endogenously by solving a budgeted prize-collecting VRP, where the prize is determined by solving a Brownian control problem. This policy also allows the decision maker to adjust service effort between orders with and without deadlines by setting relative holding costs for each type. 
Our results show that the proposed policy reduces the percentage of missed deadlines by 72.38\% compared to the CCSO data and also outperforms the benchmark policies. Notably, minimizing missing deadlines percentage requires prioritizing orders with deadlines, which leads to a higher cancellation rates by about 50\%. This is because service efforts are more effectively directed toward orders nearing their deadlines, giving other orders more time to self-resolve (i.e., through cancellation), which can be a more desired outcome. Recall that a cancellation is not necessarily an undesirable outcome as it is often due to the landlord and the tenant reaching an agreement. Moreover, by adjusting the holding cost for each order type, the system manager can manage the trade-offs in cancellation rates and the number of orders waiting. These trade-offs also provide insights for eviction enforcement planning; see Figures \ref{fig:MissDieAndCancellation_WithoutDie}--\ref{fig:MissDieAndNumPending}.

We also conduct a counterfactual analysis to explore scenarios with increased service capacity (through more eviction teams and extended working hours) and with extended deadlines. Our results  show that missed deadlines can be further reduced in both scenarios, but for different reasons. With increased service capacity, more orders are served, leading to lower cancellation rates and smaller number of orders waiting to be enforced. Conversely, extending deadlines reduces missed deadlines by allowing orders to stay longer in the system, which increases both the cancellation rate and the number of pending orders.

Lastly, it is worth noting that the data used in Section \ref{sec:Simulation} is collected during 2015 -- 2019, i.e., before the pandemic. During the pandemic, CCSO’s eviction operation were disrupted. As a result, our data set may not represent the current system state. Thus, we emphasize that our numerical study represents the CCSO’s eviction operations during 2015 -- 2019 and its conclusions should be interpreted in that context. 

%\section*{Acknowledgement}
%The authors extend their gratitude to the Cook County Sheriff's Office for providing data and insights essential to this study. We are especially thankful to Joseph Bellettiere, Amanda Gallegos, Jason Hernandez, Marianne Kelly, Rebecca Levin, Colin Luce, Tara Mitchell, Peter Pon, Sydney Purdue, and Anne Fitzgerald for their valuable discussions on eviction enforcement operations. In particular, we wish to express our sincere appreciation to Peter Pon for his detailed guidance and support. We also thank the Rustandy Center at the Chicago Booth School of Business for facilitating our connection with the Cook County Sheriff's Office.

%\clearpage
\bibliographystyle{plainnat}
\bibliography{myref.bib}

\clearpage
\appendix
\section*{Appendix}

\renewcommand{\thealgorithm}{A.\arabic{algorithm}}
\setcounter{algorithm}{0}
\renewcommand{\thefigure}{A.\arabic{figure}}
\setcounter{figure}{0}

\section{Detailed Description of Solving Budgeted Prize-Collecting VRP}
\label{sec:BPCVRP_Appendix}

In this section, we first present related preliminary work related to solving the budgeted prize-collecting TSP in Appendix \ref{subapp:TSPPreliminaries}. We then provide a detailed description of our proposed method of solving the budgeted prize-collecting VRP in Appendix \ref{subsubsec:BPCVRP}. Appendix \ref{subsubapp:InputGraphConstruction} includes a detailed discussion on constructing one of the inputs for the proposed algorithm. Finally, we provide implementation details for our numerical study based on the CCSO dataset in Appendix \ref{subsubapp:InputParametersChoices}. 

\subsection{Preliminaries}
\label{subapp:TSPPreliminaries}

As discussed in Section \ref{sec:Introduction} and \ref{subsec:ApproximateAuxiliaryFunction}, the algorithm presented by \citet{Hegde2015} is computational efficient compared with other existing methods. Given that approximating function $H(z, v)$ requires a large number of data points and computational efficiency is crucial, our method utilizes one of the main algorithms in \citet{Hegde2015}. 

\citet{Hegde2015} consider a so-called Prize-Collecting Steiner Tree (PCST) problem, which, while different from our problem, is closely related. 
Loosely speaking, PCST aims to identify a subtree of a weighted graph that optimally balances the trade-off between the cost of the edges in the selected subtree and the prizes associated with the unconnected vertices. More specifically, we let $G=(V, E)$ be a connected, undirected, and weighted graph with edge costs $c:E\rightarrow \mathbb{R}_0^+$ and vertex prizes $\pi:V\rightarrow \mathbb{R}_0^+$. For any subgraph with vertices $V'\subseteq V$ and edges $E'\subseteq E$, we let the total prize of $V'$ as $\pi(V')=\sum_{v\in V'} \pi(v)$ and the total cost of $E'$ as $c(E')=\sum_{e\in E'} c(e)$. We further let $\overline{V'}=V\setminus V'$. Definition \ref{def:PCST} provides the PCST problem definition studied by \citet{Hegde2015}.

\begin{definition}[PCST problem]
\label{def:PCST}
Find a subtree $T=(V', E')$ of $G$ such that $c(E')+\pi(\overline{V'})$ is minimized. 
\end{definition}

Building on the work of \citet{GoemansWilliamson1995} and \citet{Cole2001}, \citet{Hegde2015} present an algorithm, named \Subroutine{PCST-FAST}, that solves the PSCT in nearly linear time. PCST-FAST takes three parameters: the undirected graph $G=(V, E)$, the edge cost $c_{ij}\geq 0$ for $i,j\in V$ with $i\not=j$, and the vertex prize $\pi_i$ for $i\in V$.
Since the focus of our work is not on the implementation details of PCST-FAST, we omit its algorithm description here. Interested reader may refer to \citet{Hegde2015} for more details. 

Although the goal of the PCST problem appears similar to our problem when a single vehicle is available, the resulting solution is a tree, whereas our problem aims to identify a tour for each vehicle. Following the method discussed in \citet{GoemansWilliamson1995}, we convert the identified tree from the PCST solution into a tour. That is, to solve the prize-collecting TSP (PCTSP), we first apply PCST-FAST and then convert the resulting tree into a tour as the solution to the PCTSP. 
Algorithm \ref{alg:PCTSP} provides the detailed algorithmic statement of solving PCTSP.  
\begin{algorithm}[h!]
	\caption{Solving the PCTSP} \label{alg:PCTSP}
	{\bf Input:} An undirected graph $G=(V, E)$, edge costs $c_{ij}\geq 0$ for $i,j\in V$ and $i\not=j$, vertex prize $\pi_i\geq 0$ for $i\in V$, and a root vertex $r$. 
 
	{\bf Output:} A tour $T$, which includes vertex $r\in V$. 
	\begin{algorithmic}[1]
        \Function{PCTSP}{$G, c, \pi, r$}
		\State Apply $\text{PCST-FAST}(G, c, \pi', r)$ to the problem instance with graph $G$, edge costs $c$, root $r$, and prizes $\pi_i'=\pi_i/2$ for $i\in V$. Obtain the solution as a subgraph $F$. 
		\State Duplicate the edges of $F$ to form an Eulerian graph $F'$. 
		\State Form an Eulerian circuit in $F'$ and then make the identified Eulerian circuit into a Hamiltonian circuit $T$ by skipping repeated vertices (shortcutting). 
        \State \Return $T$
        \EndFunction
	\end{algorithmic}
\end{algorithm}

\subsection{Algorithm to Solve Budgeted Prize-Collecting VRP}
\label{subsubsec:BPCVRP}

To solve the budgeted prize-collecting VRP, we build on Algorithm \ref{alg:PCTSP} with two additional requirement: {\it i}) the total travel and service time of each vehicle cannot exceed the maximum daily working hours, and {\it ii}) multiple vehicles are considered. First, we discuss a proposed procedure that addresses the budget constraint on travel costs when a single vehicle is considered. We then discuss how this procedure can be modified to incorporate multiple vehicles. 

\paragraph{Budget constraint on edge cost.}
To account for both travel time and service time, we assign the cost of each edge $e=(i,j)$, where $i,j\in V$ are the two vertices connected by $e$, as the sum of the travel time on edge $e$ and the service time required for serving the order at vertex $j$; see Section \ref{subsec:Data} for details on how travel time and service time are estimated. Recall from Definition \ref{def:PCST} that the tour $T=(E', V')$ identified by Algorithm \ref{alg:PCTSP} aims to minimize $c(E')+\pi(\overline{V'})$. However, this objective does not fully match our goal of maximizing $\pi(V')$ (which is equivalent to minimizing $\pi(\overline{V'})$) while also enforcing a constraint that $c(E')$ remains below the maximum daily working hours. 

To address this, inspired by \citet{Hegde2015b}, we introduce a scaling parameter $\zeta>0$ that adjusts the prizes associated to all vertices through iterations updates by modifying the prize as $\zeta \pi_i$ for order $i\in V$. One may see that $\zeta$ needs to be neither too large nor too small. If $\zeta$ is significantly large, the algorithm prioritizes obtaining higher total reward over minimizing total travel cost, making the identified tour more likely to violate the budget constraint on edge cost. Conversely, if $\zeta$ is too small, satisfying the budget constraint on edge cost becomes more important than collecting a higher total prize, making it more likely to return a tour with few jobs (or even an empty tour). To find a $\zeta$ that effectively balance both vertex prizes and edge costs, we use a binary search method, ensuring that the identified tour satisfies the budget constraint on edge cost. 
Algorithm \ref{alg:BudgetConstraint} provides the algorithmic statement for solving the budgeted prize-collecting TSP. 

\begin{algorithm}[h!]
	\caption{Budgeted-PCTSP} \label{alg:BudgetConstraint}
	{\bf Input:} An undirected graph $G=(V, E)$, edge costs $c_{ij}\geq 0$ for $i,j\in V$ with $i\not=j$, initial vertex prize $\pi_i\geq 0$ for $i\in V$, and a root vertex $r$. An initial value of the scaling parameter $\zeta^{\rm init}>0$ and the common edge cost budget $D$ for each vehicle. Tolerance level $\delta>0$ and tolerated precision $\epsilon>0$.
	
	{\bf Output:} A tour $T=(V_T, E_T)$, where $V_T\subseteq V$ with $r\in V_T$ and $E_T\subseteq E$.
	\begin{algorithmic}[1]
		%\State We write $\pi_\zeta(i)=\zeta \cdot \pi(i)$. 
        \Function{Budgeted-PCTSP}{$G, c, \pi, r, D, \zeta^{\rm init}, \delta, \epsilon$}
        \State $\zeta_l\leftarrow 0$
        \State $\zeta_r\leftarrow\zeta_{\rm init}$
		\State $T=(V_T, E_T)\leftarrow \Call{PCTSP}{G, c, \zeta_r\cdot \pi,r}$  
		\If{$|c(E_T)-D|\leq\delta$} 
        \State \Return $T=(V_T, E_T)$
        \EndIf
		\While{($|c(E_T) - D| > \delta$) and ($\zeta_r-\zeta_l >\epsilon$)}
		\State $\zeta_m \leftarrow (\zeta_l+\zeta_r)/2$ 
		\State $T=(V_T, E_T)\leftarrow \Call{PCTSP}{G, c, \zeta_m\cdot\pi, r}$
		\If{$c(E_T) < D-\delta$} 
		\State $\zeta_r \leftarrow \zeta_m$
		\Else 
		\State $\zeta_l \leftarrow \zeta_m$
		\EndIf
		\EndWhile
        \If{$|c(T)-D|\leq\delta$} 
        \State \Return $T=(V_T,E_T)$
        \Else 
        \State $T_l=(V_{T_l}, E_{T_l}) \leftarrow \Call{PCTSP}{G, c, \zeta_l\cdot \pi, r}$
        \State $T_r=(V_{T_r}, E_{T_r}) \leftarrow \Call{PCTSP}{G, c, \zeta_r\cdot \pi, r}$
        \State $c^{T_r}\leftarrow \{\}$
        \For{$i\leftarrow 1,\ldots,|V_{T_r}|-1$}
            \State $c^{T_r}\leftarrow c^{T_r}\cup \{c_{v_i, v_{i+1}}\}$
        \EndFor
        \State $T_r'=(V_{T_r'}, E_{T_r'}) \leftarrow \Call{PruneTour}{T_r, c^{T_r}, D, \delta, r}$
        \If{$\pi(V_{T_l}) \geq \pi(V_{T_r'})$}
        \State \Return $T_l$
        \Else 
        \State \Return $T_r'$
        \EndIf
        \EndIf
        \EndFunction
	\end{algorithmic}
\end{algorithm}

Note that through the binary search over $\zeta$, it is possible that one of the resulting tours (i.e., the resulting tour corresponds to $\zeta_r$) significantly exceeds the edge cost budget, even if $\zeta_l$ and $\zeta_r$ are extremely close (i.e., $\zeta_r-\zeta_l\leq \epsilon$). To address this, we use a subroutine \Subroutine{PruneTour} to extract a subtour when such a situation occurs; see Algorithm \ref{alg:PruneTour}. Since the tour that violates the cost budget constraint is expected to include vertices with high prizes, \Subroutine{PruneTour} simply selects vertices along the tour, starting from the root vertex, and continues until the edge cost budget is exhausted.

\begin{algorithm}[h!]
	\caption{PruneTour} \label{alg:PruneTour}
	{\bf Input:} A tour $T=(V_T, E_T)$, edge costs $c_{ij}\geq 0$ for $i, j\in V_T$ with $i\not=j$, the edge cost budget $D$, and a root vertex $r$. Tolerance level $\delta>0$.
 
	{\bf Output:} A pruned tour $T'=(V_{T'}, E_{T'})$, where $V_{T'}\subseteq V_T$ with $r\in V_{T'}$ and $E_{T'}\subseteq E_T$.
	\begin{algorithmic}[1]
        \Function{PruneTour}{$T, c, D, \delta, r$}
        \State Let $L=\{v_1,v_2,\ldots, v_{|V_T|}, v_1\}$ be the tour of $T$, where $v_1$ is the root vertex $r$.
        \If{$c(V_T)\leq D+\delta$}
        \State \Return $T=(V_T, E_T)$
        \EndIf
        \State $T'\leftarrow \{v_1\}$
        \For{$i\leftarrow 2,\ldots,|V_T|$}
        \State $T'\leftarrow T'\cup \{v_i\}$
        \State $b\leftarrow \sum_{i=1}^{|T'|-1} c_{v_i,v_{i+1}}$
        \If{$b>D+\delta$} 
        \State $V_{T'}\leftarrow \left\{v_1,\ldots,v_{i-1}\right\}$
        \State $E_{T'}\leftarrow \cup_{i=1}^{|T'|-1} \{(v_i, v_{i+1})\}$
        \State \Return $T'=(V_{T'}, E_{T'})$
        \EndIf
        \EndFor
        \EndFunction
	\end{algorithmic}
\end{algorithm}

\paragraph{Incorporating multiple vehicles.}

To incorporate multiple vehicles, we repeatedly apply Algorithm \ref{alg:BudgetConstraint} for each vehicle with respect to the remaining orders, as the number of daily pending orders can be substantial.

Algorithm \ref{alg:BPCVRP} presents our proposed heuristic method for solving the budgeted prize-collecting VRP when facing a set of pending orders, denoted as ${\cal P}$. Each pending order in ${\cal P}$ includes information of the order, such as location (in latitude and longitude), pending duration, deadline, etc., which is used to determine the prize assigned based on the chosen policy; see detailed discussion in Appendix \ref{subapp:PolicyDescription}. 
Note that Algorithm \ref{alg:BPCVRP} requires two additional algorithms: \textsc{CreateGraph}, which generates a complete graph as input for \textsc{Budgeted-PCTSP}, and \textsc{AssignPrize}, which assigns prizes to the pending orders according to the chosen policy. Detailed discussion of \textsc{CreateGraph} and \textsc{AssignPrize} are provided in Appendices \ref{subsubapp:InputGraphConstruction} and \ref{subapp:AssignPrize}, respectively.  

\begin{algorithm}[h!]
	\caption{Solution to the Budgeted-PCVRP} \label{alg:BPCVRP}
	{\bf Input:} Set of pending orders ${\cal P}$. A root vertex $r$. An initial value of the scaling parameter $\zeta^{\rm init}>0$. 
    Number of vehicles $W$. Edge cost budget $D=\{D_w\mid w=1,\ldots, W\}$ for all vehicles. Tolerance level $\delta>0$ and tolerated precision $\epsilon>0$. The chosen policy and the required inputs (see Algorithm \ref{alg:AssignPrize}).
 
	{\bf Output:} A set of $W$ tours where all the tours include root vertex $r$.
	\begin{algorithmic}[1]
        \Function{Budgeted-PCVRP}{${\cal P}, r, W, D, \zeta^{\rm init}, \delta, \epsilon, {\rm policy}$}
        \For{$w\leftarrow 1,\ldots,W$} 
        \State $\pi\leftarrow\Call{AssignPrize}{{\rm policy}, {\cal P}, \text{Inputs for the chosen policy}}$
        \State $G,\pi^{\rm init},c \leftarrow \Call{CreateGraph}{{\cal P}, \pi}$
        \State $T_w=(V_w, E_w) \leftarrow \Call{Budgeted-PCTSP}{G, c, \pi^{\rm init}, r, D_w, \zeta^{\rm init}, \delta, \epsilon}$
        \State Remove orders whose location is within $V_w$ from ${\cal P}$
        \EndFor
        \State \Return $\{T_w=(V_w,E_w) \mid w=1,\ldots,W\}$
        \EndFunction
	\end{algorithmic}
\end{algorithm}

\subsection{Construction of the Input Undirected Graph}
\label{subsubapp:InputGraphConstruction}

A crucial step in Algorithm \ref{alg:BPCVRP} is the implementation of \textsc{Budgeted-PCTSP}, which requires an undirected graph along with the prizes for its vertices as part of the input.
To create the input graph $G=(V, E)$, where $V$ is the set of vertices and $E$ is the set edges connecting these vertices, we obtain the information for $V$ and $E$ from the daily pending orders. As discussed in Appendix \ref{alg:BudgetConstraint}, we let ${\cal P}$ be the set of pending orders, which includes essential location information (latitude and longitude) for each order.
%The essential information that we use is the orders location information, which includes its latitude and longitude. 
To determine the cost of an edge $(i,j)$, where $i,j\in V$, we first compute the estimated travel time for the edge using the estimation in Section \ref{subsec:Data}. We then add the average service time, also estimated in Section \ref{subsec:Data}, since the total cost for traveling from vertex $i$ to vertex $j$ and completing enforcement at vertex $j$ constitutes the total cost of edge $(i,j)$. Furthermore, prior to creating the input graph $G$, the system manager generates a prize for each pending order $i\in {\cal P}$ based on the chosen policy. We denote these generated prizes as $\pi_{\rm pending}$ and assign them accordingly to the vertices in graph $G$. 
Algorithm \ref{alg:GraphCreation} shows the procedure for constructing the input graph $G=(V, E)$. 

\begin{algorithm}[h!]
	\caption{Input graph creation} \label{alg:GraphCreation}
	{\bf Input:} Set of pending orders ${\cal P}$. The prize vector $\pi$, where $\pi_i$ corresponds to the prize for order $i\in {\cal P}$. 

	{\bf Output:} Graph $G=(V,E)$, vertex prize $\pi_i$ for $i\in V$, and edge cost $c_{ij}$ for the edge in between of $i,j\in V$ and $i\not=j$.
	\begin{algorithmic}[1]
        \Function{GraphCreation}{${\cal P}, \pi$}
        \State $V\leftarrow \{\}$
        \State $E\leftarrow \{\}$
        \For{$i=({\rm lat}_i, {\rm lng}_i, \ell^{\rm pending}_i, t^{\rm cancel}_i, t^{\rm deadline}_i, {\rm class}_i)\in {\cal P}$}
        \State $V\leftarrow V\cup \{i\}$
        %\State $\pi_i\leftarrow \pi_{\rm pending}^i$ 
        \For{$j=({\rm lat}_j, {\rm lng}_j, \ell^{\rm pending}_j, t^{\rm cancel}_j, t^{\rm deadline}_j, {\rm class}_j)\in {\cal P}$ where $i\not= j$}
        \State $E\leftarrow E\cup \{(i,j)\}$
        \State $d_{ij}\leftarrow$ Haversine distance between $({\rm lat}_i, {\rm lng}_i)$ and $({\rm lat}_j, {\rm lng}_j)$
        \State $c_{ij}\leftarrow \beta d_{ij}+t_{service}$, where $t_{service}$ is the average service time (see Section \ref{subsec:Data}) 
        \EndFor
        \EndFor
        \State \Return $G=(V,E), \pi_i$ for $i\in V$, and $c_{ij}$ for $i,j\in V$ and $i\not=j$
        \EndFunction
	\end{algorithmic}
\end{algorithm}

\subsection{Implementation Details on the CCSO Eviction Data}
\label{subsubapp:InputParametersChoices}

In this section, we discuss our choices and adjustments for some inputs to \textsc{Budgeted-PCTSP}. 

\paragraph{Adjustment for $\pi^{\rm init}$.}
To balance the priority values of the pending orders across all policies, we scale the initial prizes, $\pi^{\rm init}$, through a binary search of $\zeta$ to align them with the edge cost. Since the prizes determined by each policy are initially much smaller than the edge cost, we multiply them by 1,000,000. This scaling factor ensures that the prizes are significantly larger than the edge cost, preserving the relative differences between the prizes. Note that other scaling factors may also be considered, and the results are not expected to change.

\paragraph{Adjustment for edge cost $c$.} Based on the CCSO dataset, we observe that some eviction requests can be very close, such as within nearby buildings or even the same building. In these cases, the edge cost $c_{ij}$ would be extremely small. However, in practice, deputies are still expected to spend time traveling between two closely located orders. To account for this, we set a minimum edge cost of 5 minutes when creating the input graph, as described in Algorithm \ref{alg:GraphCreation}. 

\paragraph{Choice of $\zeta^{\rm init}$.} The value of $\zeta^{\rm init}$ depends on the order of magnitude difference between the assigned initial prize $\pi^{\rm init}$ and the edge cost $c$.
In our simulation study, we set $\zeta^{\rm init}=100$.

\paragraph{Choice of $\delta$.} The tolerate level $\delta$ serves as a buffer for the edge cost budget. Specifically, Algorithm \ref{alg:BPCVRP} aims to find routes with an overall edge cost within $[D-\delta, D+\delta]$ with high probability, where $D$ is the user-defined edge cost budget. The choice of $\delta$ depends on how much deviation from $D$ is acceptable to the system manager. In our simulation study, we assume the CCSO is indifferent with respect to 30 minutes in terms of their daily working hour and set $\delta$ as 30 minutes.

\paragraph{Choice of $\epsilon$.} The tolerated precision $\epsilon$ measures the maximum allowable difference between the scaling parameters $\zeta_l$ and $\zeta_r$ through the binary search (see Algorithm \ref{alg:BudgetConstraint}). A smaller $\epsilon$ requires a finer distinction between $\zeta_l$ and $\zeta_r$, leading to longer computational time. In our simulation study, we set $\epsilon=0.0001$.

\section{Additional Discussion on Approximating Function $H(z, v)$}
\label{app:TrainingFFunction_Appendix}

In this section, we provide additional discussion on how we approximate $H(z, v)$ (discussed in Section \ref{subsec:ApproximateAuxiliaryFunction}). Appendix \ref{subapp:ApproximateFunctionH} provides an overview and algorithm for approximating $H(z, v)$, followed by implementation details in Appendix \ref{subapp:ResultsApproximationH}. 

\subsection{Procedure to Approximate $H(z, v)$}
\label{subapp:ApproximateFunctionH}

We use a neural network model to approximate $H(z, v)$ because incorporating it in Algorithm \ref{alg:MultiDimControl} is easier. Completing this task involves two main steps: generating training data and training the neural network model.

%We start by discussing how to generate a single training data point. 
Recall from Section \ref{subsec:ApproximateAuxiliaryFunction} that $H(z, w)=\max_{\mu\in {\cal A}(z)}\{\mu z\}$, where $z$ and $v$ are $2K$-dimensional vectors representing the number of daily pending eviction orders and the prize of serving one eviction order from each class, respectively.
Each training data point includes three elements: the $2K$-dimensional vectors $z$ and $v$, and the corresponding response $\max_{\mu \in {\cal A}(z)} \{\mu v\}$. This response is the maximum total prize collectable, or the solution to the budgeted prize-collecting VRP, where $z_k$ represents the number of pending orders for class $k$, and $v_k$ represents the prize of serving one order of class $k$.

To facilitate the discussion, we let $n_k^{\rm max}$ be the maximum number of pending orders from class $k$, as observed in the CCSO dataset, and let ${\cal O}_k$ be the set of historical order locations from the dataset. To generate one training data point, we proceed as follows:
\begin{enumerate}
    \item We determine the number of pending orders for each class $z_k$ by generating a uniform random variable over the interval $[0, 1.5 n_k^{\rm max}]$.\footnote{We enlarge $n_k^{\rm max}$ by a factor of 1.5 to ensure that the generated number of pending orders spans well with respect to all possible values that may later occur in the simulation.} 
    \item Next, we randomly select $z_k$ locations from the set ${\cal O}_k$. For each selected order, we assign a prize by generating a uniform random variable over the interval $[0, p]$ if the order has a deadline, or $[0, h_2/\gamma_k]$ if it does not. These ranges are based on the boundary conditions in the HJB equation (Equations (\ref{eqn:HJB_Boundary1})--(\ref{eqn:HJB_Boundary3})).
    \item Finally, solve the budgeted VRP to determine the maximum total prize collected $\max_{\mu\in {\cal A}(z)}\{\mu v\}$. 
\end{enumerate}

Algorithms \ref{alg:DataGenSingleObs} shows the algorithm for generating one training data point.

\begin{algorithm}[h!]
	\caption{Generate one training data point\protect\footnotemark} 
    \label{alg:DataGenSingleObs}
	{\bf Input:} Set of locations of the historical orders ${\cal O}$. 
    Penalty of missing deadline $p$. 
    Holding cost of the orders without a deadline $h_2$. 
    Cancellation rate of class $k$, $\gamma_k$, where $k=1,\ldots,2K$.
    Number of vehicles $W$. 
    Edge cost budget $D=\{D_w\mid w=1,\ldots, W\}$.
    An initial value of the scaling parameter $\zeta^{\rm init}>0$, a root vertex $r$, tolerance level $\delta>0$, and tolerated precision $\epsilon>0$. 
 
	{\bf Output:} Number of pending orders $z$, assigned prizes for all classes $v$, and total prize collected by all vehicles $p_{\rm total}$.
	\begin{algorithmic}[1]
        \Function{DataGen}{${\cal O}, p, h_2, \gamma, r, W, D, \zeta^{\rm init}, \delta, \epsilon$}
        \State $V_{\rm pending}\leftarrow \{\}$
        \State $\pi_{\rm pending}\leftarrow \{\}$
        \For{$k=1,\ldots,2K$}
        \If{$k\leq K$}
        \State $v_k\leftarrow {\rm Unif}(0, p)$
        \Else
        \State $v_k\leftarrow {\rm Unif}(0, h_2/\gamma_k)$
        \EndIf
        \State $z_k\leftarrow {\rm Unif}(0, 1.5n^k_{\max})$
        \For{$n=1,\ldots,z_k$}
        \State $i_n\leftarrow$ Randomly chosen location $({\rm lat}_n, {\rm lng}_n)$ from ${\cal O}_k$
        \State $V_{\rm pending}\leftarrow V_{\rm pending}\cup \{i_n\}$
        \State $\pi_{\rm pending}\leftarrow \pi_{\rm pending}\cup \{v_k\}$
        \EndFor
        \EndFor
        \State $G, \pi^{\rm init}, c \leftarrow \Call{GraphCreation}{V_{\rm pending}, \pi_{\rm pending}}$ 
        \For{$w=1,\ldots,W$}
        \State $T_w=(V_w, E_w)\leftarrow \Call{Budgeted-PCTSP}{G, c, \pi^{\rm init}, r, D, \zeta^{\rm init}, \delta, \epsilon}$
        \EndFor
        \State $p_{\rm total}\leftarrow \sum_{w=1}^W \pi(V_w)$
        \State \Return $z, v, p_{\rm total}$
        \EndFunction
	\end{algorithmic}
\end{algorithm}
\footnotetext{The \textsc{GraphCreation} function shown in Algorithm \ref{alg:GraphCreation} requires the input $V_{\rm pending}$ to be a set of pending orders with order information, including latitude, longitude, pending duration, time until cancellation, deadline and class index (see Section \ref{subsubapp:InputGraphConstruction}). For the purpose of generating training data points, we only need information on the orders' latitude and longitude. Therefore, we let $V_{\rm pending}$ only includes such information in Algorithm \ref{alg:DataGenSingleObs}.}

With a set of generated training data, we train a neural network model to approximate $H(z, v)$.  
Algorithm \ref{alg:HApproximation} shows the algorithm for approximating $H(\cdot)$ using a neural network model.  

\begin{algorithm}[h!]
	\caption{Approximating function $H(z, v)$} \label{alg:HApproximation}
	{\bf Input:} Number of training data $N$. 
    Set of locations of the historical orders ${\cal O}$.
    Penalty of missing deadline $p$.
    Holding cost of the orders without a deadline $h_2$.
    Cancellation rate $\gamma$.
    Number of vehicles $W$. 
    Edge cost budget $D=\{D_w\mid w=1,\ldots, W\}$. 
    An initial value of the scaling parameter $\zeta^{\rm init}>0$, a root vertex $r$, tolerance level $\delta>0$, and tolerated precision $\epsilon>0$. 
 
	{\bf Output:} Approximated function $\hat{H}(z, v)$.
	\begin{algorithmic}[1]
        \State $z_{\rm training}\leftarrow \{\}$
        \State $v_{\rm training}\leftarrow \{\}$
        \State $p_{\rm training}\leftarrow \{\}$
        \For{$n= 1, \ldots, N$}
        \State $z, v, p_{\rm total}\leftarrow\Call{DataGen}{{\cal O}, p, h_2, \gamma, r, W, D, \zeta^{\rm init}, \delta, \epsilon}$
        \State $z_{\rm training}\leftarrow z_{\rm training}\cup \{z\}$
        \State $v_{\rm training}\leftarrow v_{\rm training}\cup \{v\}$
        \State $p_{\rm training}\leftarrow p_{\rm training}\cup \{p_{\rm total}\}$
        \EndFor
        \State Train neural network on training set $\{z_{\rm training}, v_{\rm training}, p_{\rm training}\}$
        \State \Return Function $\hat{H}(z, v)$.
	\end{algorithmic}
\end{algorithm}

\subsection{Implementation Details on the Approximation of $H(z, v)$}
\label{subapp:ResultsApproximationH}

In this section, we discuss some implementation details on how we approximate $H(z, v)$.

\paragraph{Training data preprocessing.} 

In order to improve the training, we first preprocess the training data so that their magnitude are at a similar level. 
%Specifically, we scale the data on the number of pending eviction orders and the total prize collected (i.e., the $z_{\rm training}$ and $p_{\rm training}$ in Algorithm \ref{alg:HApproximation}) so that they have a similar magnitude as . 
%Since the values of each data point in $z$ roughly ranges from 0 to 500 for each dimension, we scale the 
In our simulation study, we scale $z_{\rm training}$ and $p_{\rm training}$ by a factor of 1/500.

\paragraph{Neural network architecture.}

Table \ref{tab:NeuralNetworkConfigurationVRP} shows the neural network architecture that we use to approximate $H(z, v)$.
\begin{table}[h!]
	\centering
	\begin{tabular}{ c|c }
		\toprule
		Hyper-parameters & Value \\   \midrule
		Number of hidden layers & 1  \\  \hline
		Number of neurons per layer & 200  \\  \hline
        Loss function & Mean squared error \\ 
        \hline 
        Activation function & ReLU  \\   \hline
		Batch size & 128   \\   \hline 
		Number of iterations & 1000  \\  \hline
		Learning rate & 0.0001  \\
 		\bottomrule
	\end{tabular}
	\caption{Neural network configuration for approximating $H(z, v)$.}
	\label{tab:NeuralNetworkConfigurationVRP}
\end{table}

\section{Detailed Description of the Simulation for the Proposed Policy and Benchmarks}
\label{sec:BenchmarksAlg}

In this section, we provide a detailed description of how we simulate the proposed policy and benchmark policies in Appendix \ref{subapp:PolicyDescription}.  
Appendix \ref{subapp:AssignPrize} contains detailed algorithms used to assign prizes within the simulation.

\subsection{Description of Policy Simulation}
\label{subapp:PolicyDescription}

%As described in Section \ref{sec:BenchmarkPolicies}, the main difference in simulating all three policies is how the prizes are assigned for the daily pending eviction orders. 
In this section, we discuss the generic framework used to simulate all policies.

Our simulation contains five major. Initially, the simulation model generates a set of initial pending orders. We then run the simulation for $D$ time units (days). Within each time unit, we first generate a set of newly arrived orders and add them to the set of pending orders. Next, we remove any orders that have been canceled from the set of pending orders. We then implement the budgeted prize-collecting VRP  to serve a subset of the pending orders and subsequently remove them from the pending orders set. Finally, we remove any pending orders that miss their deadlines.  
In the remainder of this section, we provide implementation details for each of the events in our simulation.

Recall from Appendix \ref{app:TrainingFFunction_Appendix} that we use ${\cal O}_k$ to denote the set of locations of orders from class $k$ in the CCSO dataset. To facilitate the discussion, we introduce additional notation whose values are directly obtained from the CCSO dataset. We let $L$ be the set of durations between the order's received date and the deadline (as shown in Figure \ref{fig:CommonDeadline})\footnote{In the counterfactual analysis regarding extended deadlines, we add an additional $\tau$ days to all elements within $L$, where $\tau\in \{10, 20, 30\}$.}. We also let $R$ be the set of $2K$-dimensional vectors corresponding to the number of daily arrivals across all classes. Additionally, let $n^{\rm mean}$ be the $2K$-dimensional vector, where $n_k^{\rm mean}$ represents the average number of daily pending orders of class $k$ for $k=1,\ldots,2K$.

Throughout the simulation, we let ${\cal P}$ be the set of pending orders, where each element $i\in {\cal P}$ contains information regarding the order's location (the latitude ${\rm lat}_i$ and the longitude ${\rm lng}_i$), the duration it has been pending ($\ell^{\rm pending}_i$), the duration until it will be canceled if not served nor miss its deadline  ($t^{\rm cancel}_i$), the deadline ($t^{\rm deadline}_i$), and the class index (${\rm class}_i$). Specifically, we let $i=({\rm lat}_i, {\rm lng}_i, \ell^{\rm pending}_i, t^{\rm cancel}_i, t^{\rm deadline}_i, {\rm class}_i)\in {\cal P}$ be an element of set ${\cal P}$. 
Furthermore, we use $i(x)$ as the associated value of $x$ for order $i$, where $x\in \{{\rm lat}, {\rm lng}, \ell^{\rm pending}, t^{\rm cancel}, t^{\rm deadline}, {\rm class}\}$. 
Algorithm \ref{alg:GenericPolicy} presents the generic algorithm for implementing three different policies. 
\begin{algorithm}[h!]
	\caption{Policy simulation} \label{alg:GenericPolicy}
	{\bf Input:} Chosen policy. 
    Simulation run-length $D$ (in days). 
    Set of locations of the historical orders ${\cal O}$. 
    Set of durations between orders' received date and the deadline $L$. 
    Set of $2K$-dimensional vectors of number of daily arrivals from all classes. 
    Average number of daily pending orders $n^{\rm mean}$. 
    Estimated cancellation rate $\gamma$.
    Number of vehicles $W$ and cost budget $D$. An initial value of the scaling parameter $\zeta^{\rm init}>0$. Tolerance level $\delta>0$ and tolerated precision $\epsilon>0$. 
 
	{\bf Output:} Estimated performance measures.
	\begin{algorithmic}[1]
        \State ${\cal P}\leftarrow\Call{InitGen}{{\cal O}, n^{\rm mean}, L, \gamma}$
        \For{$\tau\leftarrow 1, \ldots, D$}
        \State ${\cal P}\leftarrow\Call{ArrivalGen}{{\cal P}, {\cal O}, R, L, \gamma}$ as in Algorithm \ref{alg:ArrivalGen}
        \State ${\cal P}\leftarrow\Call{RemoveCanceled}{{\cal P}}$ as in Algorithm \ref{alg:RemoveCanceled}
        %\State $G, \pi^{\rm init}, c\leftarrow \Call{GraphCreation}{V_{\rm pending}, \pi_{\rm pending}}$
        \State ${\cal P}\leftarrow \Call{Budgeted-PCVRP}{{\cal P}, r, W, D, \zeta^{\rm init}, \delta, \epsilon, {\rm policy}}$ as in Algorithm \ref{alg:BPCVRP}
        \State ${\cal P}\leftarrow\Call{RemoveMissingDeadline}{{\cal P}}$ as in Algorithm \ref{alg:RemoveMissingDeadline}
        \EndFor
        \State \Return ${\cal P}$
	\end{algorithmic}
\end{algorithm}

We now discuss each event within the simulation in detail.
\paragraph{Generate initial set of pending orders.}

This event creates the initial set ${\cal P}$. To do so, we utilize the location information of historical orders from the CCSO dataset. Specifically, 
for each class $k$, where $k=1,\ldots,2K$, we randomly select $n_k^{\rm mean}$ locations (including latitude and longitude) from the CCSO dataset that belong to the same class. For each order $i$, we also randomly generate its pending duration $\ell^{\rm pending}_i$ as a uniform random variable between 0 and 73 days (median duration of the deadlines; see Figure \ref{fig:CommonDeadline}). Building on the pending duration, the duration until the order will be canceled is generated as an exponential random variable with rate $\gamma$, added up to the pending duration. Finally, the associated deadline is randomly chosen from the set $R$. 
Algorithm \ref{alg:InitGen} shows the algorithm statement to generate the initial set of pending orders. 
\begin{algorithm}[h!]
	\caption{Initial state generation} \label{alg:InitGen}
	{\bf Input:} Set of locations ${\cal O}$. 
    Average number of daily pending orders $n^{\rm mean}$. 
    Set of durations between orders’ received date and the deadline $L$.
    Estimated cancellation rate $\gamma$.
 
	{\bf Output:} Set of pending orders ${\cal P}$. 
	\begin{algorithmic}[1]
        \Function{InitGen}{${\cal O}, n^{\rm mean}, L, \gamma$}
        %\State $V_{\rm pending}\leftarrow \{\}$
        %\State $\pi_{\rm pending}\leftarrow \{\}$
        \State ${\cal P}\leftarrow \{\}$
        \For{$k=1,\ldots,2K$}
            \For{$n=1,\ldots,n^{\rm mean}_k$}
                \State $({\rm lat}, {\rm lng})\leftarrow$ Randomly chosen from ${\cal O}_k$
                \State $\ell^{\rm pending}\leftarrow {\rm Unif}(0, 73)$
                \State $t^{\rm cancel}\leftarrow \ell^{\rm pending}+\exp(\gamma_k)$
                \If{$k\leq K$}
                \State $t^{\rm deadline}\leftarrow$ Randomly
                chosen value from $L$
                \Else
                \State $t^{\rm deadline}\leftarrow\inf$ 
                \EndIf
                \State $i\leftarrow ({\rm lat}, {\rm lng}, \ell^{\rm pending}, t^{\rm cancel}, t^{\rm deadline}, k)$
                \State ${\cal P}\leftarrow {\cal P}\cup \{i\}$
            \EndFor
        \EndFor
        \State \Return ${\cal P}$
        \EndFunction
	\end{algorithmic}
\end{algorithm}

\paragraph{Generate daily arrivals.}

Building on the generated initial pending orders ${\cal P}$, we add newly generated orders as daily arrivals. To do this, we first determine the number of daily arrivals for each class from the set $R$. Next, for each class, based on the number of daily arrivals, we select locations for the new orders and set their pending duration to 0. The duration before cancellation and deadlines are generated in the same manner as in \textsc{InitGen}. These newly generated orders are then added to ${\cal P}$.  
Algorithm \ref{alg:ArrivalGen} shows the algorithm for generating daily arrivals. 
\begin{algorithm}[h!]
	\caption{Daily arrival orders generation} \label{alg:ArrivalGen}
	{\bf Input:} Set of pending orders ${\cal P}$. Set of locations ${\cal O}$. 
    Set of number of daily arrivals $R$. 
    Set of durations between orders’ received date and the deadline $L$. 
 
	{\bf Output:} Updated set of pending orders ${\cal P}$.
	\begin{algorithmic}[1]
        \Function{ArrivalGen}{${\cal P}, {\cal O}, R, L$}
        \State $n^{\rm arrival}\leftarrow$ Randomly chosen from $R$
        \For{$k\leftarrow 1$ to $2K$}
            \For{$n\leftarrow 1$ to $n^{\rm arrival}_k$}
                \State $({\rm lat}, {\rm lng})\leftarrow$ Randomly chosen from ${\cal O}_k$
                \State $\ell^{\rm pending}\leftarrow 0$
                \State $t^{\rm cancel}\leftarrow \exp(\gamma_k)$
                \If{$k\leq K$}
                \State $t^{\rm deadline}\leftarrow$ Randomly chosen value from $L$
                \Else
                \State $t^{\rm deadline}\leftarrow\inf$ 
                \EndIf
                \State $i\leftarrow ({\rm lat}, {\rm lng}, \ell^{\rm pending}, t^{\rm cancel}, t^{\rm deadline}, k)$
                \State ${\cal P}\leftarrow {\cal P}\cup \{i\}$
            \EndFor
        \EndFor
        \State \Return ${\cal P}$
        \EndFunction
	\end{algorithmic}
\end{algorithm}

\paragraph{Remove orders that are canceled.}

To remove canceled orders, we compare each order's pending duration with its cancellation time. Specifically, we remove orders for which the pending duration exceeds the cancellation time. Algorithm \ref{alg:RemoveCanceled} shows the algorithm to remove canceled orders. 
\begin{algorithm}[h!]
	\caption{Order removal due to cancellation} \label{alg:RemoveCanceled}
	{\bf Input:} Set of pending orders ${\cal P}$. 
 
	{\bf Output:} Updated set of pending orders ${\cal P}$.
	\begin{algorithmic}[1]
        \Function{RemoveCanceled}{${\cal P}$}
        \For{$i=({\rm lat}_i, {\rm lng}_i, \ell^{\rm pending}_i, t^{\rm cancel}_i, t^{\rm deadline}_i, {\rm class}_i)$ in ${\cal P}$} 
            \If{$\ell_i^{\rm pending} > t_i^{\rm cancel}$}
            \State ${\cal P}\leftarrow {\cal P}\setminus \{i\}$
            \EndIf
        \EndFor
        \State \Return ${\cal P}$
        \EndFunction
	\end{algorithmic}
\end{algorithm}

\paragraph{Remove orders that miss their deadlines.}

To remove the orders that miss their deadlines, we compare each order's pending duration with its associated deadline. Specifically, if an order's pending duration exceeds its deadline, we remove it from the system. Additionally, if an order is still in the system (i.e., it has not missed its deadline), we increase its pending duration by one day. 
Algorithm \ref{alg:RemoveMissingDeadline} shows the algorithm to remove orders that miss their deadlines. 
\begin{algorithm}[h!]
	\caption{Order removal due to missing deadline} \label{alg:RemoveMissingDeadline}
	{\bf Input:} Set of pending orders ${\cal P}$. 
 
	{\bf Output:} Updated set of pending orders ${\cal P}$.
	\begin{algorithmic}[1]
        \Function{RemoveMissingDeadlien}{${\cal P}$}
        \For{$i=({\rm lat}_i, {\rm lng}_i, \ell^{\rm pending}_i, t^{\rm cancel}_i, t^{\rm deadline}_i, {\rm class}_i)$ in ${\cal P}$} 
            \If{$\ell_i^{\rm pending} > t_i^{\rm deadline}$}
            \State ${\cal P}\leftarrow {\cal P}\setminus \{i\}$
            \Else 
            \State $\ell_i^{\rm pending}\leftarrow \ell_i^{\rm pending}+1$
            \EndIf
        \EndFor
        \State \Return ${\cal P}$
        \EndFunction
	\end{algorithmic}
\end{algorithm}

\subsection{Assigning Prizes for Different Policies}
\label{subapp:AssignPrize}

In this section, we provide a more detailed discussion on how prizes are assigned to the daily pending orders based on different policies. For completeness, we also briefly discuss how prizes are set for each policy, as discussed in Section \ref{sec:BenchmarkPolicies}.
Algorithm \ref{alg:AssignPrize} presents the generic framework for assigning prizes to each pending order, depending on the chosen policy. We then discuss the detailed approach and the specific algorithm for each policy separately. 
\begin{algorithm}[h!]
	\caption{Assign prizes for each vehicle} \label{alg:AssignPrize}
	{\bf Input:} Chosen policy. Set of pending orders ${\cal P}$. \\
    Inputs for the Proposed policy: trained $2K$-dimensional gradient function $G^{w_2}$.  \\
    Inputs for the urgency-based policy: artificial common deadline for orders without a deadline $\tilde{d}$.  \\
    Inputs for the threshold-based policy: vehicle index $v$, suburb threshold $\xi$, probability to include orders without a deadline for service $w$.  
 
	{\bf Output:} Assigned prize $\pi_i$ for $i\in {\cal P}$.
	%{\fontsize{10}{12}\selectfont	
    \begin{algorithmic}[1]
        \Function{AssignPrize}{${\rm policy},{\cal P}, \text{Inputs for the chosen policy}$}
        \State $\pi\leftarrow$ Algorithm \ref{alg:AssignPrize}a, \ref{alg:AssignPrize}b, or \ref{alg:AssignPrize}c based on the chosen policy
        \State \Return $\pi$
        \EndFunction
    \end{algorithmic}
    %}
\end{algorithm}

\paragraph{Proposed policy.} To determine the prize for each pending order, we first sort the orders within each class ${\cal I}^{(k)}$, where $k=1,\ldots,2K$. Orders with a deadline are sorted in ascending order by the remaining days until their deadlines, while orders without a deadline are sorted in descending order based on how long they have been waiting. The proposed policy requires a trained gradient function ${G}^{w_2}(\cdot)$, as described in Section \ref{sec:ComputationalMethod}, to assign prizes. To construct the input for ${G}^{w_2}$, which requires a 2K-dimensional vector, we first set $\tilde{z}=z_k$, the number of pending orders in class $k$, for each $k=1,\ldots,2K$. For the $i$th order in the sorted list of class $k$, we update $\tilde{z}_k$ to $z_k-i+1$, where $i\in {\cal I}^{(k)}$. The prize for the $i$th order in class $k$ is then determined by the $k$th dimension of $G^{w_2}(\tilde{z})$, i.e., $G_k^{w_2}(\tilde{z})$.
Algorithm \ref{alg:AssignPrize_Proposed} provides the algorithm statement for assigning prizes under the proposed policy. 
\begin{algorithm}[h!]
    \renewcommand{\thealgorithm}{\ref{alg:AssignPrize}a}
	\caption{Assign prizes for one vehicle based on the proposed policy} \label{alg:AssignPrize_Proposed}
	%{\fontsize{10}{12}\selectfont	
    \begin{algorithmic}[1]
        \For{$k=1,\ldots,2K$}
        \State $I_k\leftarrow \{i\in {\cal P}\mid i(\text{class})=k\}$
        \State $z_k\leftarrow |I_k|$
        \If{$k\leq K$}
        \State $I_k\leftarrow$ Sorted $I_k$ in ascending order of $i(t^{\rm deadline})-i(\ell^{\rm pending})$ for $i\in I_k$
        \Else
        \State $I_k\leftarrow$ Sorted $I_k$ in descending order of $i(\ell^{\rm pending})$ for $i\in I_k$
        \EndIf
        \EndFor
        \For{$k=1,\ldots,2K$}
        \State ${\rm index}\leftarrow 0$
        \For{$i\in I_k$}
        \State $\tilde{z}\leftarrow 2K$-vector of $z_k$ for $k=1,\ldots, 2K$
        \State $\tilde{z}_k\leftarrow z_k-{\rm index}+1$
        \State $\pi_i\leftarrow G_k^{w_2}(\tilde{z})$
        \State ${\rm index}\leftarrow {\rm index}+1$
        \EndFor
        \EndFor
    \end{algorithmic}
    %}
\end{algorithm}

\paragraph{Urgency-based policy.} 
The urgency-based policy categorizes pending orders into three sets: ${\cal I}_1, {\cal I}_2$, and ${\cal I}_3$ \footnote{Although this classification differs from what discussed in Section \ref{sec:BenchmarkPolicies}, it is used here for computational purposes and is equivalent to that in Section \ref{sec:BenchmarkPolicies}}. For orders in ${\cal I}_3$, all prizes are set to 1, which are the highest prizes among all pending orders. For orders in ${\cal I}_1\cup {\cal I}_2$, we sort them in ascending order based on the remaining days until their deadlines if they belong to ${\cal I}_1$, or the remaining days until $\tilde{d}$ if they belong to ${\cal I}_2$. The prize for the $i$th order in ${\cal I}_1\cup {\cal I}_2$ is assigned as $\pi_i=1/2^i$, where $i\in {\cal I}_1\cup {\cal I}_2$.
Algorithm \ref{alg:AssignPrize_BudgetedFCFS} provides the algorithm for assigning prizes under the urgency-based policy. 

\begin{algorithm}[h!]
    \renewcommand{\thealgorithm}{\ref{alg:AssignPrize}b}
	\caption{Assign prizes for one vehicle based on the urgency-based policy} \label{alg:AssignPrize_BudgetedFCFS}
	%{\fontsize{10}{12}\selectfont	
    \begin{algorithmic}[1]
        \State $I_1\leftarrow \{\}$
        \State $I_2\leftarrow \{\}$
        \State $I_2\leftarrow \{\}$
        \For{$i=({\rm lat}_i, {\rm lng}_i, \ell^{\rm pending}_i, t^{\rm cancel}_i, t^{\rm deadline}_i, {\rm class}_i)$ in ${\cal P}$}
        \If{$i(t^{\rm deadline})\not=\inf$}
        \State $I_1\leftarrow I_1\cup \{i\}$
        \Else 
        \If{$i(\ell^{\rm pending})\leq \tilde{d}$}
        \State $I_2\leftarrow I_2\cup \{i\}$
        \Else 
        \State $I_3\leftarrow I_3\cup \{i\}$
        \EndIf
        \EndIf
        \EndFor
        \State $I'\leftarrow$ Sorted $I_1\cup I_2$ in ascending order of $i(t^{\rm deadline})-i(\ell^{\rm pending})$ if $i\in I_1$ and $\tilde{d}-i(\ell^{\rm pending})$ if $i\in I_2$
        \State ${\rm index}\leftarrow 0$
        \For{$i\in I'\cup I_3$}
        \If{$i\in I'$}
        \State $\pi_i\leftarrow 1/(2^{{\rm index}+1})$
        \State ${\rm index}\leftarrow {\rm index}+1$
        \Else
        \State $\pi_i\leftarrow 1$
        \EndIf
        \EndFor
    \end{algorithmic}
    %}
\end{algorithm}

\paragraph{Threshold-based policy.}

To assign prizes for pending orders under the threshold-based policy, we first determine the dispatch zone for each vehicle, then set prizes for the orders within that zone. 
Recall that $z_k$ represents the number of pending orders in class $k$, ${\cal G}_s$ and $G_d$ are the sets of suburb and city zones, respectively, and $\xi$ is the calibrated threshold used to decide if a vehicle should be dispatched to a suburb zone. If there exists a suburb zone $r\in {\cal G}_s$ such that $z_r+z_{r+K}\geq \xi$, the dispatch zone $r^\ast$ for the first vehicle is chosen from ${\cal G}_s$, with each zone $r\in {\cal G}_s$ selected with probability $(z_r+z_{r+K})/\sum_{r'\in G_s} (z_{r'}+z_{r'+K})$. For remaining vehicles, the dispatch zone $r^\ast$ is chosen from ${\cal G}_d$, with each zone $r\in {\cal G}_d$ selected based on the probability $(z_r+z_{r+K})/\sum_{r'\in {\cal G}_d} (z_{r'}+z_{r'+K})$. If no suburb zone meets the threshold $\xi$ (i.e., $z_r<\xi$ for all $r\in {\cal G}_s$), then $r^\ast$ for all vehicles is chosen among city zones ${\cal G}_d$, with selection probabilities calculated similarly.

Once the dispatch zone $r^\ast$ for a vehicle is determined, prizes for the orders in $r^\ast$ are assigned based on a Bernoulli random variable, with the probability dependent on whether the order has a deadline. 
Algorithm \ref{alg:AssignPrize_ThresholdBased} provides the algorithm for assigning prizes under the threshold-based policy. 

\begin{algorithm}[h!]
    \renewcommand{\thealgorithm}{\ref{alg:AssignPrize}c}
	\caption{Assign prizes for one vehicle based on the threshold-based policy} \label{alg:AssignPrize_ThresholdBased}
	%{\fontsize{10}{12}\selectfont	
    \begin{algorithmic}[1]
        \State $G_s\leftarrow \{2, 3, 8, 9, 11, 12\}$
        \For{$k=1,\ldots,2K$}
        \State $z_k\leftarrow |\{i\in {\cal P}\mid i({\rm class})=k\}|$
        \EndFor
        \If{there exists $r\in {\cal G}_s$ such that $(z_r+z_{r+K})\geq \xi$ and $v=1$} 
        \State $r^\ast\leftarrow$ Randomly select $r\in G_s$ with probability $(z_r+z_{r+K})/\sum_{r'\in {\cal G}_s} (z_{r'}+z_{r'+K})$
        \Else 
        \State $r^\ast\leftarrow$ Randomly select $r\in G_d$ with probability $(z_r+z_{r+K})/\sum_{r'\in {\cal G}_d} (z_{r'}+z_{r'+K})$
        \EndIf
        \For{$i=({\rm lat}_i, {\rm lng}_i, \ell^{\rm pending}_i, t^{\rm cancel}_i, t^{\rm deadline}_i, {\rm class}_i)\in {\cal P}$}
        \If{$i({\rm class})=r^\ast$}
        \State $\pi_i\leftarrow {\rm Ber}(i(\ell^{\rm pending})/i(t^{\rm deadline}))$
        \ElsIf{$i({\rm class})=r^\ast+K$}
        \State $\pi_i\leftarrow {\rm Ber}(w)$ 
        \Else
        \State $\pi_i\leftarrow 0$
        \EndIf
        \EndFor
    \end{algorithmic}
    %}
\end{algorithm}

\section{Additional Information for Data Description and Model Parameters Estimation}
\label{app:ParameterEstimation}

In this section, we first provide a comprehensive description of the CCSO dataset and discuss how we estimate the performance metrics in Appendix \ref{subapp:DetailedDataDescription}. Appendices \ref{subapp:DemandEstimation}--\ref{subapp:MeanMaxNumPending} provide detailed discussion on the estimation of various model parameters from the CCSO dataset for our simulation study. 

\subsection{Detailed Data Description and Estimated Performance Measures}
\label{subapp:DetailedDataDescription}

In this section, we first provide a detailed discussion of the CCSO dataset and then present the estimated performance measures derived from the dataset.  

The CCSO dataset contains 195,428 entries (rows of the data), which corresponds to 73,682 unique eviction orders collected between January 1, 2015 and September 30, 2019. 
Table \ref{tab:DataVariable} shows the variables from the CCSO dataset that are utilized in our analysis, along with their respective meanings.
\begin{table}[h!]
	\centering
	\resizebox{\textwidth}{!}{
		\begin{tabular}{ p{0.2\textwidth} | p{0.8\textwidth} }
			\toprule
			Variable & Description \\  \midrule 
            Case number & Case number associated with each eviction order. \\
            Sheriff number & Sheriff number associated with each action taken on the eviction order. \\
            Received date & Date when the eviction order was received.  \\
            Service date & Date when the eviction order was served (if applicable).  \\
            Canceled date & Date when the eviction order is canceled (if applicable). \\
            Die date & Associated deadline of the eviction order (if applicable).   \\
            Case disposition & Outcome of the eviction order. \\
            Deputy information & Identifier of the deputy who served the eviction order. \\
            Def address & Address where the eviction order was enforced.  \\
            Def zip code & Zip code of the eviction order, which also determines the district and zone.  \\
            Evict start time & Time when the eviction enforcement began.  \\
            Evict end time & Time when the eviction enforcement concluded.  \\
			\bottomrule
		\end{tabular}
	}
	\caption{Variables and their descriptions from the CCSO dataset used in the analysis.}
	\label{tab:DataVariable}
\end{table}

Within the dataset, there are 63,995 orders with valid records for disposition (to determine whether an order was canceled) and case termination date (used to identify if an order was served, canceled, or missed the deadline). Table \ref{tab:CaseOutcomePercentage} presents the number and percentage of eviction orders that were served, canceled, or missed their initial 120-day deadline, according to CCSO dataset.
\begin{table}[h!]
	\centering
	\begin{tabular}{ c||c|c|c }
		\toprule
		Type & Served & Canceled & Missed deadline \\ \midrule
		With a deadline & 33,296 & 11,883 & 7,724 \\
        Without a deadline & 6,399 & 4,693 & 0  \\  \midrule
        Total & 39,695 (62.03\%) & 16,576 (25.90\%) & 7,724 (12.07\%\footnotemark)  \\
		\bottomrule
	\end{tabular}
	\caption{Number (percentage) of the eviction orders that were served, canceled, or missed their deadlines, according to the CCSO dataset.}
	\label{tab:CaseOutcomePercentage}
\end{table}
\footnotetext{Note that 12.07\% does not correspond to the percentage of the orders that missed their deadline.}

We then discuss how we estimate the performance metrics (described in Table \ref{tab:PerformanceMeasure}) using the CCSO dataset. 
\begin{itemize}
    \item {\bf Missing deadline percentage.} To estimate the percentage of orders that were not served before their deadline passed, we specifically focus on orders that had a deadline and were not canceled before the deadline. The missing deadline percentage is calculated as $7,724/(7,724+33,296)=18.83\%$.
    \item {\bf Cancellation percentage.} 
    %Among the 64,252 orders, 16,576 were canceled. The percentage of orders that were canceled is estimated as
    %\begin{align*}
    %    \% \text{ Cancellation} &= \frac{16,576}{64,252}=25.8\%.
    %\end{align*}
    Among the 53,160 orders with a deadline, 11,883 were canceled. Among the 11,092 orders without a deadline, 4,693 orders were canceled.
    Therefore, cancellation percentage for orders with and without a deadline is $11,883/53,160=22.35\%$ and $4,693/11,092=42.31\%$, respectively. 
    \item {\bf Average number of orders waiting to be enforced.} To estimate the number of daily pending orders, we check the number of orders that were not served, not canceled, and have not yet missed their deadline for each day. Since our analysis focuses on 63,995 (a subset of the CCSO dataset after excluding problematic orders), we scale up the estimated daily pending orders proportionally to reflect the full dataset. Additionally, we exclude the first and last 50 days of data from our analysis to minimize the impact of initial and final variance in the system. Table \ref{tab:NumDailyPendingStats} presents the statistics for the number of daily pending orders from the CCSO dataset. 
    \begin{table}[h!]
	\centering
	\begin{tabular}{ c||c|c|c|c|c }
		\toprule
		Type & Mean & Median & Standard deviation & Minimum & Maximum \\ \midrule
		With a deadline & 2,326.46 & 2,345.92 & 252.91 & 1,859.10 & 2,878.11 \\
        Without a deadline & 1057.62 & 990.48 & 325.46 & 532.84 & 1,786.50   \\  
        %\midrule
        %Total & 3,370.54 & 3,350.82 & 500.64 & 2,413.36 & 4,596.88 \\
		\bottomrule
	\end{tabular}
	\caption{Statistics for the estimation of the average number of orders waiting to be enforced.}
	\label{tab:NumDailyPendingStats}
\end{table}
    
    \item {\bf Number of orders served daily.} To estimate the number of orders served daily, we calculate the number of orders were served on each day. As with estimating the daily pending orders, we also scale up the results proportionally and exclude the first and last 50 days from our analysis. Table \ref{tab:NumDailyServedStats} shows the statistics of the number of daily served orders from the CCSO dataset. 
    \begin{table}[h!]
	\centering
	\begin{tabular}{ c||c|c|c|c|c }
		\toprule
		Type & Mean & Median & Standard deviation & Minimum & Maximum \\ \midrule
		With a deadline & 34.47 & 35.00 & 17.13 & 0 & 125.76 \\
        Without a deadline & 7.05 & 6.48 & 5.64 & 0 & 36.30  \\  
        %\midrule
        %Total & 41.35 & 42.61 & 18.82 & 0 & 129.13 \\
		\bottomrule
	\end{tabular}
	\caption{Statistics for the number of daily served eviction orders.}
	\label{tab:NumDailyServedStats}
    \end{table}
\end{itemize}

\subsection{Order Arrival Process Estimation}
\label{subapp:DemandEstimation}

In this section, we discuss how we estimate the key parameters of the eviction order arrival process. 

As discussed in Section \ref{sec:BackgroundAndModel}, we assume that $A_k(t)$, the cumulative number of arrivals for class $k$, follows a Gaussian distribution with mean $\lambda_k t$ and variance $\sigma_k^2 t$ for $t\geq 0$ and $k=1,\ldots,2K$. Additionally, $A_1,\ldots,A_{2K}$ are assumed to be mutually independent Gaussian processes. We estimate $\lambda_k$ and $\sigma_k^2$ using a Maximum Likelihood Estimator (MLE).  
To perform the estimation, we consider the eviction orders received over 1080 days, from 2015/1/2 to 2019/9/30\footnote{Although the dataset starts from 2015/1/1, no new eviction orders were received by the CCSO on this date. Consequently, we exclude 2015/1/1 from our estimation. We also remove weekends and national holidays, as the CCSO does not process new eviction orders on these days.}. We also exclude data of the first and last 50 days of the observation period to minimize the variance within the initial and final periods of the observation (similar as in Section \ref{subapp:DetailedDataDescription}). We use the received date variable to record the number of eviction orders arrived to the system. 

Let $D_{1kd}$ and $D_{0kd}$ be the number of eviction orders with and without a deadline, respectively, received from zone $r$ on day $d$,  where $r=1,\ldots, 12$ and $d=1,\ldots, 1080$. 
The mean $\lambda_k$ and variance $\sigma_k^2$ for class $k$ are estimated as $\hat{\lambda}_k$ and $\hat{\sigma}_k^2$ using the MLE as follows:
\begin{align*}
    \hat{\lambda}_k &= \frac{1}{1080}\sum_{d=1}^{1080} D_{1kd} \text{ and } \hat{\sigma}_k^2=\frac{1}{1080}\sum_{d=1}^{1080}\left( D_{1kd}-\hat{\lambda}_k\right)^2 \text{ for $k=1,\ldots,K$};   \\
    \hat{\lambda}_k &= \frac{1}{1080}\sum_{d=1}^{1080} D_{0kd} \text{ and } \hat{\sigma}_k^2=\frac{1}{1080}\sum_{d=1}^{1080}\left( D_{0kd}-\hat{\lambda}_k\right)^2 \text{ for $k=K+1,\ldots,2K$}.
\end{align*}

Table \ref{tab:DailyDemandByNodes_12Nodes} provides the MLE estimates $\hat{\lambda}_k$ and $\hat{\sigma}_k$ for $k=1,\ldots,2K$, based on the history of the 73,682 received eviction orders. 
\begin{table}[h!]
	\centering
	\begin{tabular}{ c|cc|| c|cc}
		\toprule
		$k$ & $\hat{\lambda}_k$ & $\hat{\sigma}_k$ & $k$ & $\hat{\lambda}_k$ & $\hat{\sigma}_k$   \\   \midrule
		1 & 6.7519 & 4.1093 & 13 & 2.0611 & 2.0037  \\
		2 & 1.4315 & 1.3868 & 14 & 0.8380 & 1.1453  \\
		3 & 1.6185 & 1.6213	& 15 & 0.5241 & 0.8127   \\
		4 & 13.5593 & 7.2243 & 16 & 1.2565 & 1.4194  \\
		5 & 6.2880 & 4.0101 & 17 & 1.0981 & 1.3511  \\
		6 & 5.3102 & 3.4792	& 18 & 0.6556 & 1.0540  \\
		7 & 6.6713 & 4.5532 & 19 & 0.5306 & 0.9940  \\
		8 & 2.4769 & 2.2557 & 20 & 0.6306 & 0.9553  \\
		9 & 1.4537 & 1.4611 & 21 & 0.5259 & 0.8742  \\
		10 & 9.6074 & 5.1836 & 22 & 1.5593 & 1.6511  \\
		11 & 1.2028 & 1.2905 & 23 & 0.5019 & 0.8311   \\
		12 & 2.3657 & 1.9547 & 24 & 1.3407 & 1.6551  \\
		\bottomrule
	\end{tabular}
	\caption{The MLE estimates $\hat{\lambda}_k$ and $\hat{\sigma}_k^2$, where $k=1,\ldots,2K$, for the mean and variance of the demand arrival process.}
	\label{tab:DailyDemandByNodes_12Nodes}
\end{table}

\subsection{Cancellation Rate Estimation}
\label{subapp:CancellationRateEstimation}

In this section, we discuss our approach to estimating the cancellation rate of eviction orders. As discussed in Section \ref{sec:BackgroundAndModel}, an eviction order may be canceled (and thus removed from the system) while it is waiting to be served. To estimate the cancellation rate, we apply the jackknife estimate of a Kaplan-Meier integral due to \citet{StuteWang1994}. 

We first present a general discussion of our approach. 
Considering $N$ total data points, we let
\begin{align*}
    c_i &= \text{duration (in days) that order $i$ remains in the system, where $i=1,\ldots, N$};   \\
    \delta_i &= \begin{cases}
        1, & \text{if order $i$ is canceled (i.e., the data is observed)}, \\
        0, & \text{otherwise (i.e., the data is censored)},
    \end{cases} \quad \text{ where $i=1,\ldots, N$}.
\end{align*}
Orders canceled after their deadline are treated as missing the deadline rather than as canceled. Let $t^{\rm cancel}$ be the average time (in days) until an eviction order is canceled, estimated by the Kaplan-Meier integral:
\begin{align*}
	\hat{t}^{\rm cancel} &=  \sum_{i=1}^n W_i c_{(i)},
\end{align*} 
where $W_i$ is defined for $1\leq i\leq N$ as
\begin{align*}
	W_i &= \frac{\delta_{[i]}}{N-i+1} \prod_{k=1}^{i-1} \left(\frac{N-k}{N-k+1}\right)^{\delta_{[k]}},
\end{align*} 
with $c_{(i)}$ as the $i$th order statistics of $c_i$ and $\delta_{[i]}$ is the concomitant of the $i$th order statistics, i.e., $\delta_{[i]}=\delta_k$ if $c_{(i)}=c_k$, where $1\leq i, k\leq N$.
The bias-corrected jackknife estimate of the Kaplan-Meier integral is given by \begin{align*}
	{\rm BIAS} &= -\frac{N-1}{N} c_{(N)} \delta_{[N]} \left(1-\delta_{[N-1]}\right) \prod_{k=1}^{N-2} \left(\frac{N-1-k}{N-k}\right)^{\delta_{[k]}}. 
\end{align*}
Finally, the bias-corrected jackknife estimator of the Kaplan-Meier integral is $\tilde{t}^{\rm cancel} = \hat{t}^{\rm cancel}-{\rm BIAS}$. 

In our problem, the data of interest (i.e., the duration before an order is canceled) are not always observable. Specifically, when an eviction order is terminated and removed from the system due to a reason other than being canceled (such as missing the deadline or being served), the order is prevented from being canceled and will not be observed as a canceled order. 

Our cancellation estimation is based on 64,252 unique eviction order, including 53,160 orders with a deadline and 11,092 without. We further exclude outliers by removing the top 5\% and bottom 5\% of both censored and uncensored data points.
Table \ref{tab:NumCasesUseForCancellationEstimation_Aggregated} shows the number of censored and uncensored orders that we use for the cancellation estimation as well as the estimated cancellation rate (in days). 
\begin{table}[h!]
	\centering
	\begin{tabular}{ c|c|c }
		\toprule
		Number of & Number of & Estimated cancellation \\ 
		censored orders & uncensored orders & mean time (rate) (in days)  \\ \midrule
		45,324 & 12,975 & 122.56 (0.008) \\
		\bottomrule
	\end{tabular}
	\caption{Estimated cancellation rate of eviction orders based on CCSO dataset.}
	\label{tab:NumCasesUseForCancellationEstimation_Aggregated}
\end{table}

\subsection{Estimation of Service Time}
\label{subapp:ServiceTimeEstimation}

%There are two crucial variables that are estimated using the orders that were served from the CCSO dataset.  
To conduct the service time estimation, we focus on eviction orders with valid records of both their start and end times. 
Our estimation is based on 34,843 unique orders (i.e., Figure \ref{fig:ServiceTimeEstimation} is created based on 34,843 unique orders). 

%\begin{table}[h!]
%	\centering
%	\begin{tabular}{ c|c|| c|c}
%		\toprule
%		$k$ & Average daily served orders & $k$ & Average daily served orders   \\   \midrule
%		1 & 3.5892 & 13 & 1.1956  \\
%		2 & 0.7616 & 14 & 0.5022  \\
%		3 & 0.9409 & 15 & 0.3121   \\
%		4 & 8.2425 & 16 & 0.8584  \\
%		5 & 3.4469 & 17 & 0.6588 \\
%		6 & 3.2198 & 18 & 0.4412  \\
%		7 & 4.0029 & 19 & 0.3216  \\
%		8 & 1.2937 & 20 & 0.3742 \\
%		9 & 0.7831 & 21 & 0.2798  \\
%		10 & 6.0988 & 22 & 1.0222 \\
%		11 & 0.6480 & 23 & 0.2403   \\
%		12 & 1.3032 & 24 & 0.8142  \\
%		\bottomrule
%	\end{tabular}
%	\caption{The estimated maximum number of pending orders from each class based on the CCSO dataset.}
%	\label{tab:MaxNumPendingByClass}
%\end{table}

\subsection{Estimation of Number of Teams on Daily Assignments}
\label{subapp:NumTeamEstimation}

In the CCSO dataset, if a team of deputies served an eviction order, this order is recorded by each deputy in the team, resulting in multiple entries for the same order. In practice, city-dispatched teams typically serve eviction orders together, whereas suburb-dispatched teams may split, with some deputies serving orders individually or in smaller groups. To estimate the number of teams dispatched for daily assignments, we focus on 30,538 eviction orders with both recorded service times and serving deputy information. 
We identify distinct teams based on these rules:
\begin{enumerate}
	\item {\bf Subset of orders}: If a deputy served a subset of the orders served by other deputies, they are considered part of the same team, with the team’s served orders being the union of orders served by each member.

	\item {\bf Intersecting orders}: If a deputy’s set of orders intersects with the orders served by others, those deputies are also considered one team, with the total served orders taken as the union of orders served by each member.
	This happens when one team served one order together and some member served other orders separately.  
	\item {\bf Disjoint orders}: Deputies serving non-overlapping sets of orders are treated as separate teams since we have no evidence linking them to a single team.

    %This is because when two teams served completely different orders, we do not have any evidence that those deputies were from one team. It is reasonable to assume each group is a separate team. 
	\end{enumerate}
This aggregation yields 4,527 teams. By analyzing their recorded service dates, we calculated the daily number of dispatched teams, summarized in  Table \ref{tab:NumDailyTeamStats}. The corresponding bar plot is also included in Figure \ref{fig:NumDailyTeams}. Based on these statistics, we estimate an average of four teams assigned daily.
\begin{table}[h!]
	\centering
	\begin{tabular}{ c|c|c|c|c }
		\toprule
		Mean & Median & Standard deviation & Minimum & Maximum \\ \midrule
		3.99 & 4 & 1.65 & 0 & 10 \\
		\bottomrule
	\end{tabular}
	\caption{Statistics of the number of teams on daily assignments.}
	\label{tab:NumDailyTeamStats}
\end{table}

\subsection{Estimation of Daily Working Hours}
\label{subapp:DailyWorkingHoursEstimation}

In this section, we discuss how we estimate the daily working hours for each eviction team, based on the 4,527 teams identified in Section \ref{subapp:NumTeamEstimation}. 

We first obtain each team’s daily working duration as the time difference between the earliest eviction start time and the latest eviction end time for each team. Table \ref{tab:TeamHourStats} presents the statistics of these working hours, including the total team-hour, which represents the sum of the working hours across all identified teams.
\begin{table}[h!]
	\centering
	\begin{tabular}{ c|c|c|c|c|c }
		\toprule
		Mean & Median & Standard deviation & Minimum & Maximum & Total team-hour \\ \midrule
		3.14 & 3.33 & 1.31 & 0.02 & 7.8 & 14224.08 \\
		\bottomrule
	\end{tabular}
	\caption{Statistics of daily operation hours by teams.}
	\label{tab:TeamHourStats}
\end{table}

To estimate the daily working hours per team, we adjust the total team-hour to reflect the full set of eviction orders in the CCSO dataset. Since our analysis is based on 30,538 unique eviction orders, and 61.78\% of all orders are served (Table \ref{tab:CaseOutcomePercentage}), we calculate the expected number of served orders as $0.6709\times 73,682=45,521$. Thus, our data represents roughly $30,538/45,521=0.6709$ of the served orders. We scale up the total team-hour by dividing by this percentage: 
\begin{align*}
    \text{Adjusted total team-hour} &= \frac{14,224.08}{0.6709} = 21,201.49.
\end{align*}
Assuming four teams and 1,180 days (the time duration covered by the dataset), we estimate the daily working hours as $21,201.49/(4\times 1,180)=4.49$ hours. 

To account for travel time between the depot and the first eviction, as well as between the last eviction and the depot, we add an additional 0.5 hours. This brings the estimated daily working hours per team to approximately 5 hours.

\subsection{Travel Times Estimation}
\label{subapp:TravelTimeEstimation}

To prepare the data for regression analysis on the travel time between eviction orders, we use 30,747 unique orders. The data is processed as follows:
\begin{enumerate}
	\item {\bf Identifying serving deputies and routes}: For each unique order, we determine the deputies involved. Deputies who serve the same orders are grouped into the same route, while deputies with no overlap in orders are treated as having separate routes.
	
	Note that this is different from how we group serving deputies into teams (as in Appendix \ref{subapp:DailyWorkingHoursEstimation}), as as deputies may only handle a subset of orders within a team. As the deputies may serve orders individually and simultaneously, we should consider their routes as separate data points for the travel time estimation. 
	\item {\bf Defining consecutive case pairs}: For each route, we consider consecutive pairs of orders. The travel distance between each case pair can be calculated based on the latitude and longitude of the locations, and the travel time is estimated as the difference between the end time of the first case and the start time of the next case.
\end{enumerate} 

Using this approach, we fit a regression model to estimate travel time and obtain a coefficient $\beta=1.9792$.  

\subsection{Estimation of Deadline}
\label{subapp:CommonDeadlineEstimation}

To estimate the deadline for eviction orders with a deadline, we utilize the time duration between the order received date and the associated deadline from the CCSO dataset. Figure \ref{fig:CommonDeadline} in Section \ref{subsec:Data} is based on 53,160 unique eviction orders that have a valid entry for the deadline variable.

\subsection{Order Age Before Removal Due to Being Served}
\label{subapp:OrderAge}

Based on 39,695 received eviction orders that were served, we check the age of these orders at the time they are removed from the system.
Figure \ref{fig:AgeBeforeServed} presents the distribution of order age upon removal.
\begin{figure}[h!]
\begin{subfigure}[h]{0.5\linewidth}
\includegraphics[width=0.9\linewidth]{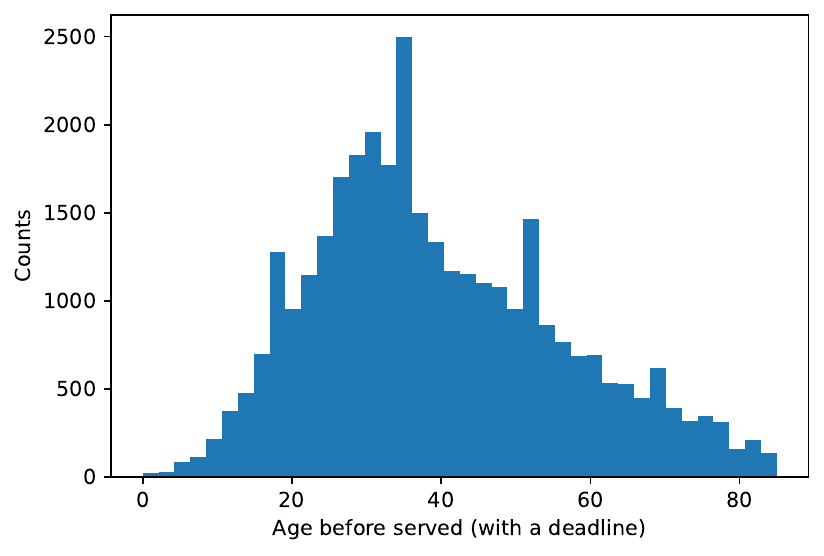}
\caption{With a deadline}
\end{subfigure}
\hfill
\begin{subfigure}[h]{0.5\linewidth}
\includegraphics[width=0.9\linewidth]{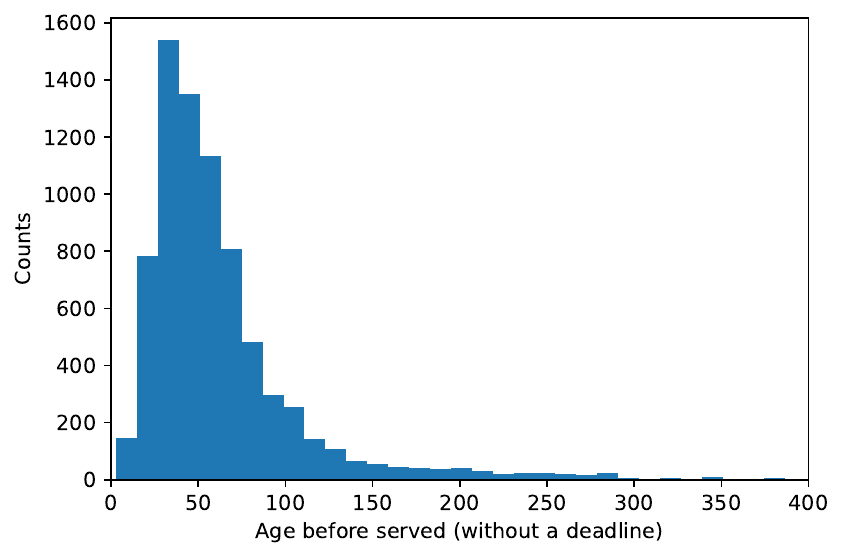}
\caption{Without a deadline}
\end{subfigure}%
\caption{Order age upon removal due to being served.}
\label{fig:AgeBeforeServed}
\end{figure}

For orders with a deadline, we further analyze the time elapsed from the order's received date to its service date as a proportion of the total deadline duration. For each served order, this percentage is calculated as the ratio of the time from the received date to the service date over the time from the received date to the deadline. Figure \ref{fig:UntilDieRatio} presents the histogram of these orders based on the percentage of the deadline utilized for service. We restrict our analysis to orders served before their deadline, which accounts for 30,952 unique orders. We exclude orders served exactly on their deadline because they might have data issues, such as being recorded as served on the deadline while also being marked as missing a deadline. 

\begin{figure}[h!]
    \centering
    \includegraphics[width=0.5\linewidth]{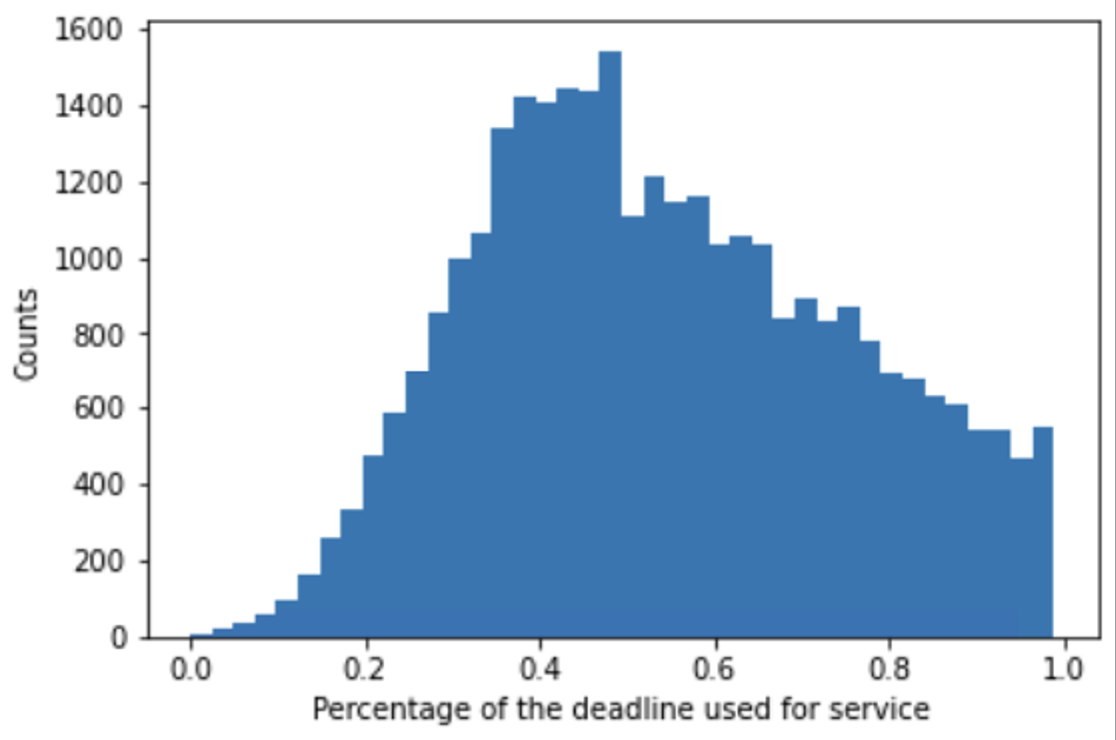}
    \caption{The percentage of the deadline used before an order is served.}
    \label{fig:UntilDieRatio}
\end{figure}

We see that fewer orders are served at lower percentages of the deadline, with an increase in orders served closer to the deadline. This further suggests that the deadline information implicitly affects the eviction enforcement planning in practice.

\subsection{Mean and Maximum Numbers of Daily Pending Orders From Different Classes}
\label{subapp:MeanMaxNumPending}

During our simulation study, Algorithm \ref{alg:HApproximation} requires $n_k^{\max}$, which represents the maximum number of pending orders for class $k$, while Algorithm \ref{alg:GenericPolicy} requires $n_k^{\rm mean}$, the mean number of pending orders for class $k$. 
Table \ref{tab:MaxNumPendingByClass} presents the estimated mean and maximum numbers of pending orders for each class. Note that, similar as in Section \ref{subapp:DetailedDataDescription}, we also scale up the values directly from the data by a factor of 1.29 and exclude the data from the first and last 50 days of our observation period for the analysis. 
\begin{table}[h!]
	\centering
	\begin{tabular}{ c|c|c|| c|c|c}
		\toprule
		$k$ & $n_k^{\rm mean}$ & $n_k^{\max}$ & $k$ & $n_k^{\rm mean}$ & $n_k^{\max}$   \\   \midrule
		1 & 268.47 & 396.42 & 13 & 193.05 & 299.57  \\
		2 & 54.32 & 92.97 & 14 & 77.55 & 126.54  \\
		3 & 63.75 & 111.05 & 15 & 53.24 & 102.01   \\
		4 & 537.85 & 726.98 & 16 & 103.09 & 161.41  \\
		5 & 253.17 & 358.97 & 17 & 105.72 & 174.32 \\
		6 & 197.54 & 299.57 & 18 & 45.71 & 95.55  \\
		7 & 267.50 & 384.80 & 19 & 37.83 & 74.89  \\
		8 & 99.18 & 149.79 & 20 & 71.62 & 118.80  \\
		9 & 56.81 & 92.97 & 21 & 52.32 & 103.30  \\
		10 & 377.48 & 546.20 & 22 & 143.67 & 304.74 \\
		11 & 46.85 & 78.77 & 23 & 51.18 & 89.10   \\
		12 & 94.24 & 158.82 & 24 & 118.40 & 260.83  \\
		\bottomrule
	\end{tabular}
	\caption{The mean and maximum number of pending orders from each class based on the CCSO dataset.}
	\label{tab:MaxNumPendingByClass}
\end{table}

\section{Calibration for the Benchmark Policies}
\label{app:ModelCalibration}

In this section, we discuss how we calibrate the model parameters for the urgency-based policy and the threshold-based policy in Appendices \ref{subsec:FCFS_Calibration} and \ref{subsec:ThresholdBased_Calibration}, respectively

\subsection{Calibration for the Urgency-Based Policy}
\label{subsec:FCFS_Calibration}

As discussed in Section \ref{sec:BenchmarkPolicies}, the urgency-based policy requires a parameter $\tilde{d}$, which represents the artificial common deadline for orders without a deadline. We determine this parameter by calibrating our simulation results to ensure that the service effort allocation between orders with and without a deadline is similar to that observed in the CCSO dataset. According to the dataset, an average of 41.52 orders are served daily, with 34.47 orders (approximately 83\%) having deadlines and 7.05 orders (approximately 17\%) without deadlines; see Appendix \ref{subapp:DetailedDataDescription} for further details.

We aim to select $\tilde{d}$ so that approximately 83\% of the service effort is allocated to orders with a deadline and 17\% to those without. 
The CCSO dataset shows that the average age before orders are removed from the system due to being served is 39.56 days for orders with deadlines and 73.88 days for those without; see Figure \ref{fig:AgeBeforeServed} in Appendix \ref{subapp:OrderAge} for the age distribution. To determine an appropriate $\tilde{d}$, we consider the 90th to 95th percentiles of the age at which orders without a deadline are served. Table \ref{tab:AgeBeforeServedQuantile} shows the value for each of these percentiles.    
\begin{table}[h!]
		\centering
		\begin{tabular}{c || c |c |c|c|c|c}
			\toprule
			Percentile & 95th & 94th & 93th & 92th & 91th & 90th  \\   \midrule
			Age & 181 & 161 & 143 & 131 & 123 & 115  \\  
			\bottomrule
		\end{tabular}
        \caption{Percentile of the age at which orders without a deadline are served.}
        \label{tab:AgeBeforeServedQuantile}
	\end{table}

Table \ref{tab:NumServed_BudgetedFCFS_DiffDie} shows the number of orders served in the simulation and the fraction of service effort allocated to orders with deadlines under the urgency-based policy for various values of $\tilde{d}$. We set $\tilde{d}=143$ as this value ensures a balance in service effort between orders with and without deadlines that aligns closely with the observed data.
	\begin{table}[h!]
		\centering
		\begin{tabular}{c|c|c|c}
			\toprule
			Percentile chosen for & Number of orders  & Number of orders & Fraction of service effort  \\ 
			the urgency-based & served daily & served daily & allocated to orders  \\ 
            policy & (with a deadline) & (without a deadline) &  with a deadline \\ \midrule 
			95th & 18.86 $\pm$ 1.28 & 2.73 $\pm$ 0.11 & 0.8736 \\  
			94th & 18.37 $\pm$ 1.25 & 3.20 $\pm$ 0.13 & 0.8516 \\
			93th & 17.80 $\pm$ 1.21 & 3.70 $\pm$ 0.15 & 0.8279 \\
			92th & 17.37 $\pm$ 1.18 & 4.07 $\pm$ 0.17 & 0.8105  \\
			91th & 17.09 $\pm$ 1.17 & 4.33 $\pm$ 0.18 & 0.7979 \\
			90th & 16.79 $\pm$ 1.15 & 4.62 $\pm$ 0.19 & 0.7842 \\ \midrule
			Data & 34.47 & 7.05 & 0.8302 \\
			\bottomrule
		\end{tabular}
		\caption{Estimated number of daily served orders for the urgency-based policy and various values of $\tilde{d}$.}
		\label{tab:NumServed_BudgetedFCFS_DiffDie}
	\end{table}

\subsection{Calibration for the Threshold-Based Policy}
\label{subsec:ThresholdBased_Calibration}

As discussed in Section \ref{sec:BenchmarkPolicies}, the threshold-based zone clearing policy depends on two key parameters: $\xi$, which determines whether to dispatch a vehicle to one of the suburb zones, and $w$, which is related to the probability that an order without a deadline will be served within the chosen zone for each vehicle. Note that $\xi$ affects the service rate across the suburb and city zones, while $w$ affects the service rate between orders with and without a deadline. 

To calibrate these parameters, we compare the simulated service rate with estimates derived from the CCSO dataset.  
Specifically, we consider $\xi\in\{150, 160, \ldots, 200\}$ and $w\in \{0.3, 0.5, 0.7, 0.9\}$. 
Tables \ref{tab:ThresholdBasedCalibration_CitySuburb} and \ref{tab:ThresholdBasedCalibration_Deadline} show the simulated service rates for city and suburb zones and for orders with and without a deadline, respectively.  
\begin{table}[h!]
    \centering
    \begin{tabular}{ c||c|c|c|c||c|c|c|c }
		\toprule
        & \multicolumn{4}{c||}{City} & \multicolumn{4}{c}{Suburb}  \\  \midrule
        \diagbox{$\xi$}{$w$} & 0.3 & 0.5 & 0.7 & 0.9 & 0.3 & 0.5 & 0.7 & 0.9  \\  \midrule
        150 & 32.02 & 32.21 & 32.24 & 32.49 & 8.23 & 8.18 & 8.21 & 8.04 \\
        160 & 32.09 & 32.54 & 32.77 & 32.77 & 8.19 & 7.98 & 7.83 & 7.80 \\
        170 & 32.56 & 33.00 & 33.41 & 33.33 & 7.86 & 7.59 & 7.28 & 7.32 \\
        180 & 33.12 & 33.59 & 33.70 & 33.80 & 7.40 & 7.13 & 7.07 & 6.99 \\
        190 & 33.50 & 33.90 & 34.38 & 34.65 & 7.14 & 6.89 & 6.49 & 6.32 \\
        200 & 34.15 & 34.49 & 35.21 & 35.21 & 6.62 & 6.41 & 5.85 & 5.82 \\  
        \midrule
        Data & \multicolumn{4}{c||}{33.23} & \multicolumn{4}{c}{8.29} \\
		\bottomrule
	\end{tabular}
    \caption{Simulated service rate in the city and suburb zones.}
    \label{tab:ThresholdBasedCalibration_CitySuburb}   
\end{table}

\begin{table}[h!]
    \centering
    \begin{tabular}{ c||c|c|c|c||c|c|c|c }
		\toprule
        & \multicolumn{4}{c||}{With a deadline} & \multicolumn{4}{c}{Without a deadline}  \\  \midrule
        \diagbox{$\xi$}{$w$} & 0.3 & 0.5 & 0.7 & 0.9 & 0.3 & 0.5 & 0.7 & 0.9  \\  \midrule
        150 & 33.61 & 32.95 & 32.50 & 32.22 & 6.65 & 7.44 & 7.96 & 8.31 \\
        160 & 33.65 & 33.12 & 32.61 & 32.31 & 6.64 & 7.40 & 7.99 & 8.27 \\
        170 & 33.77 & 33.18 & 32.88 & 32.52 & 6.64 & 7.41 & 7.81 & 8.13 \\
        180 & 33.98 & 33.35 & 32.97 & 32.64 & 6.54 & 7.37 & 7.80 & 8.14 \\
        190 & 34.16 & 33.48 & 33.13 & 32.94 & 6.48 & 7.31 & 7.74 & 8.02 \\
        200 & 34.33 & 33.65 & 33.51 & 33.11 & 6.43 & 7.25 & 7.54 & 7.92 \\  
        \midrule
        Data & \multicolumn{4}{c||}{34.47} & \multicolumn{4}{c}{7.05} \\
		\bottomrule
	\end{tabular}
    \caption{Simulated service rate for orders with or without a deadline.}
    \label{tab:ThresholdBasedCalibration_Deadline}    
\end{table}

Given that the majority of the orders are from the city zones and have deadlines, we primarily focus on matching the service rate for city zones and for orders with a deadline with the estimates from the CCSO dataset. Specifically, we consider the combination of $\xi$ and $w$ that result in simulated values within a 5\% range of the data for the orders from the city and orders with a deadline. 
Tables \ref{tab:Performance_ThresholdBasedEqualPrize}--\ref{tab:Performance3_ThresholdBasedEqualPrize} show the simulated performance metrics.
\begin{table}[h!]
	\centering
	\begin{tabular}{c|c|c|c}
		\toprule
		& Missing deadline & Cancellation percentage & Cancellation percentage \\ 
		$(\xi, w)$ & percentage & (with deadline) & (without deadline) \\  \midrule
        (170, 0.3) & $20.97\% \pm 0.37\%$ & $27.10\% \pm 0.13\%$ & $43.38\% \pm 0.33\%$  \\ 
        (180, 0.3) & $20.95\% \pm 0.26\%$ & $26.95\% \pm 0.14\%$ & $43.67\% \pm 0.32\%$  \\
        (190, 0.3) & $20.54\% \pm 0.23\%$ & $26.75\% \pm 0.13\%$ & $43.91\% \pm 0.30\%$ \\
        (190, 0.5) & $21.81\% \pm 0.28\%$ & $27.16\% \pm 0.19\%$ & $37.54\% \pm 0.23\%$  \\ \midrule
        Data & 18.83\% & 22.35\% & 42.31\%  \\
		\bottomrule
	\end{tabular}
	\caption{Estimated missing deadline percentage and cancellation percentage under the threshold-based policy with various combinations of $\xi$ and $w$.}
	\label{tab:Performance_ThresholdBasedEqualPrize}
\end{table}

%\begin{table}[h!]
%	\centering
%	\begin{tabular}{c|c|c|c}
%		\toprule
%		& \# Pending orders & \# Pending orders & \# Pending orders  \\ 
%		$(\xi, w)$ & (with deadline) & (without deadline) & (total) \\  \midrule
%        (170, 0.3) & 1993.59 $\pm$ 121.09 & 637.50 $\pm$ 34.23 & 2631.08 $\pm$ 105.40 \\ 
%        (180, 0.3) & 1993.90 $\pm$ 119.32 & 635.53 $\pm$ 32.90 & 2629.43 $\pm$ 104.23  \\
%        (190, 0.3) & 1976.35 $\pm$ 116.13 & 638.38 $\pm$ 32.59 & 2614.73 $\pm$ 101.90 \\
%        (190, 0.5) & 1993.85 $\pm$ 114.43 & 546.21 $\pm$ 29.61 & 2540.06 $\pm$ 103.07 \\ \midrule
%        Data & 2317.15 & 1053.39 & 3370.54 \\
%		\bottomrule
%	\end{tabular}
%	\caption{Estimated number of orders waiting to be enforced under the threshold-based policy with various combinations of $\xi$ and $w$.}
%	\label{tab:Performance2_ThresholdBasedEqualPrize}
%\end{table}

\begin{table}[h!]
	\centering
	\begin{tabular}{c|c|c}
		\toprule
		& Number of orders waiting & Number of orders waiting \\ 
        & to be enforced & to be enforced \\
		$(\xi, w)$ & (with a deadline) & (without a deadline)  \\  \midrule
        (170, 0.3) & 1,993.59 $\pm$ 121.09 & 637.50 $\pm$ 34.23  \\ 
        (180, 0.3) & 1,993.90 $\pm$ 119.32 & 635.53 $\pm$ 32.90   \\
        (190, 0.3) & 1,976.35 $\pm$ 116.13 & 638.38 $\pm$ 32.59  \\
        (190, 0.5) & 1,993.85 $\pm$ 114.43 & 546.21 $\pm$ 29.61  \\ \midrule
        Data & 2,326.46 & 1,057.62  \\
		\bottomrule
	\end{tabular}
	\caption{Estimated number of orders waiting to be enforced under the threshold-based policy with various combinations of $\xi$ and $w$.}
	\label{tab:Performance2_ThresholdBasedEqualPrize}
\end{table}

%\begin{table}[h!]
%	\centering
%	\begin{tabular}{c|c|c|c}
%		\toprule
%		& \# Served orders & \# Served orders & \# Served orders \\ 
%		$(\xi, w)$ & (with deadline) & (without deadline) & (total) \\  \midrule
%        (170, 0.3) & 33.83 $\pm$ 2.27 & 6.6 $\pm$ 0.25 & 40.43 $\pm$ 1.97 \\ 
%        (180, 0.3) & 33.99 $\pm$ 2.32 & 6.54 $\pm$ 0.26 & 40.53 $\pm$ 2.01  \\
%        (190, 0.3) & 34.16 $\pm$ 2.38 & 6.48 $\pm$ 0.26 & 40.64 $\pm$ 2.05 \\
%        (190, 0.5) & 33.5 $\pm$ 2.42 & 7.28 $\pm$ 0.31 & 40.78 $\pm$ 2.04 \\ \midrule
%        Data & 34.33 & 7.02 & 41.35 \\
%		\bottomrule
%	\end{tabular}
%	\caption{Estimated number of orders served daily under the threshold-based policy with various combinations of $\xi$ and $w$.}
%    \label{tab:Performance3_ThresholdBasedEqualPrize}
%\end{table} 

\begin{table}[h!]
	\centering
	\begin{tabular}{c|c|c}
		\toprule
		& Number of orders & Number of orders  \\ 
        & served daily & served daily  \\
		$(\xi, w)$ & (with a deadline) & (without a deadline)  \\  \midrule
        (170, 0.3) & 33.83 $\pm$ 2.27 & 6.60 $\pm$ 0.25 \\ 
        (180, 0.3) & 33.99 $\pm$ 2.32 & 6.54 $\pm$ 0.26   \\
        (190, 0.3) & 34.16 $\pm$ 2.38 & 6.48 $\pm$ 0.26  \\
        (190, 0.5) & 33.50 $\pm$ 2.42 & 7.28 $\pm$ 0.31  \\ \midrule
        Data & 34.47 & 7.05  \\
		\bottomrule
	\end{tabular}
	\caption{Estimated number of orders served daily under the threshold-based policy with various combinations of $\xi$ and $w$.}
    \label{tab:Performance3_ThresholdBasedEqualPrize}
\end{table} 

After further comparing the simulated performance metrics, we select $\xi=190$ and $w=0.3$ as these values yield simulation results that better match the estimates from the CCSO dataset.

\section{Detailed Discussion on the Calculation of the Performance Metrics}
\label{app:PerformanceMeasures}

As shown in Table \ref{tab:PerformanceMeasure}, we focus on four primary performance metrics to compare the proposed policy, benchmarks, and data. In this section, we provide a detailed description of how each metric is calculated from the simulation results.

Our simulation runs for $D$ simulated days. Let $I_j$ represent the set of initial pending orders, either with a deadline (if $j=1$) or without (if $j=0$).
For each simulated day $d$, where $d=1,\ldots, D$, let $A_{jd}$ denote the set of orders that arrive on day $d$, with a deadline if $j=1$ and without if $j=0$. Additionally, let $C_{jd}$ and $M_d$ represent the sets of canceled or missed-deadline orders, respectively, among all orders arrived up to day $d$ (i.e., orders within $\cup_{d'=1}^d A_{jd}$). Finally, let $S_{jd}$ be the set of served orders among the initial orders and all orders arrived up to day $d$ (i.e., orders within $I_j\cup (\cup_{d'=1}^d A_{jd})$).

With the above notation, we let $N_{jd}$ be the set of pending orders with or without a deadline on the $d$th day, with $N_{j0}$ indicating the initial pending orders.
We then have the following relationship for $N_{jd}$:  
\begin{align*}
    N_{jd} &= \begin{cases}
        I_j, & \text{if $d=0$},  \\
        \left(N_{j d-1}\cup A_{jd}\right)\setminus \left(C_{jd} \cup S_{jd} \cup M_d\right), & \text{if $1\leq d\leq D$ and $j=1$}, \\
        \left(N_{j d-1}\cup A_{jd}\right)\setminus \left(C_{jd} \cup S_{jd}\right), & \text{if $1\leq d\leq D$ and $j=0$}.
    \end{cases}
\end{align*}
We now discuss how the performance metrics are computed. 
\paragraph{Missing deadline percentage.} The missing deadline percentage is estimated as $\sum_{d=1}^D |M_d |/\sum_{d=1}^D (|A_{1d}|-|{\cal C}_{1d}|)$.

\paragraph{Cancellation percentage.} 
The cancellation percentage for orders with and without a deadline is estimated as $\sum_{d=1}^D | C_{1d} |/ \sum_{d=1}^D |A_{1d}|$ and $| C_{0d} |/\sum_{d=1}^D |A_{0d}|$, respectively.

\paragraph{Number of orders waiting to be enforced.} 
The number of pending orders with and without a deadline is estimated as $\sum_{d=1}^D |{\cal N}_{1d}|/D$ and $\sum_{d=1}^D |{\cal N}_{0d}|/D$, respectively.

\paragraph{Number of orders served daily.} 
The number of daily served orders with and without a deadline is estimated as $\sum_{d=1}^D |{\cal S}_{1d}|/D$ and $\sum_{d=1}^D |{\cal S}_{0d}|/D$, respectively.

\section{Implementation Details of the Computational Method}

In this section, we provide additional implementation details of our computational method: Appendix \ref{subapp:DataPreprocessing} discusses how we scale the input data to train the neural network model. Appendix \ref{subapp:ScaledProposedPolicy} describes the adjustment of prizes in the proposed policy based on the trained gradient function with scaled input data.
Appendix
\ref{subapp:ReferencePolicy} discusses the chosen reference and behavior policies used for training. Appendix
\ref{subapp:ShapeConstraint} discusses the shape constraints applied during training to ensure certain properties of the model. Appendix 
\ref{subapp:NNArchitecture} specifies the architecture of the neural network used in the training process.

\subsection{Data Preprocesssing}
\label{subapp:DataPreprocessing}

Recall that Algorithm \ref{alg:MultiDimControl} simulates discretized RBM paths, Brownian increments, and the values of $\tilde{U}(T)$ in each  training iteration. To ensure the optimization algorithm converges efficiently, we scale the training data to balance the magnitudes of input and output.
To do so, we consider a scaling parameter $\kappa>0$ and define the scaled queue length process $\hat{Z}(t)$ that measures the backlog in batches as $\hat{Z}(t)=Z(t)/\kappa$. The state space is 
\begin{align}
    \hat{S} &= \prod_{k=1}^K \left[0, \hat{b}_k\right]\times\mathbb{R}_+^K,
\end{align}
where $\hat{b}_k=b_k/\kappa$. We rewrite the state-space constraint as 
\begin{align}
    \hat{Z}(t) \in \hat{S} \quad \text{ for } t\geq 0. \label{eqn:StateSpaceConstraint_Hat}
\end{align}
The evolution of the scaled system state follow similar forms to Equations (\ref{eqn:StateDynamics_WithDeadline}) and (\ref{eqn:StateDynamics_WithoutDeadline}): For classes with a deadline ($k=1,\ldots, K$), we have 
\begin{align}
\label{eqn:StateDynamics_WithDeadline2}
    \hat{Z}_k(t) &= \hat{Z}_k(0)+\hat{\lambda}_k t+\hat{B}_k(t)-\int_{0}^t \gamma_k \hat{Z}_k(s)ds-\int_0^t \hat{\mu}_k(\hat{Z}_k(s))ds +\hat{L}_k(t)-\hat{U}_k(t), \quad t\geq 0, 
\end{align}
while for classes without a deadline ($k=K+1,\ldots,2K$), we have
\begin{align}
\label{eqn:StateDynamics_WithoutDeadline2}
    \hat{Z}_k(t) &= \hat{Z}_k(0)+\hat{\lambda}_k t+\hat{B}_k(t)-\int_{0}^t \gamma_k \hat{Z}_k(s)ds-\int_0^t \hat{\mu}_k(\hat{Z}_k(s))ds +\hat{L}_k(t).
\end{align}
We have $\hat{\lambda}_k=\lambda_k/\kappa$, $\hat{L}_k(t)=L(t)/\kappa$, $\hat{U}_k(t)=U(t)/\kappa$, $\hat{\mu}(\hat{Z}(t))=\mu(Z(t))/\kappa$, and $\hat{B}_k$ is a driftless Brownian motion with variance parameter $\hat{\sigma}_k^2 = (\sigma_k/\kappa)^2$ for $k=1,\ldots, 2K$. Similar to $L$ and $U$ satisfying Equations (\ref{eqn:L})-(\ref{eqn:UpperPushingConstraint}), $\hat{L}$ and $\hat{U}$ satisfy the following conditions:
\begin{align}
    & \hat{L} \text{ is non-decreasing, non-anticipating with } \hat{L}(0)=0, \label{eqn:LHat}   \\
    & \hat{U} \text{ is non-decreasing, non-anticipating with } \hat{U}(0)=0, \label{eqn:UHat}  \\
    & \int_{0}^t \hat{Z}_k (s) d \hat{L}_k(s)= 0, \quad k=1,\ldots, 2K, \quad t\geq 0, \label{eqn:LowerPushingConstraint_Hat}  \\
    & \int_0^t \left( \hat{b}_k - \hat{Z}_k(s) \right) d\hat{U}_k(s) = 0, \quad k=1, \ldots, K,  \quad \hspace{0.2cm} t\geq 0. \label{eqn:UpperPushingConstraint_Hat}
\end{align} 
Similar to Equations (\ref{eqn:Objective}) and (\ref{eqn:Constraints}), with the scaled process, the system manager seeks to choose a policy $\hat{\mu}(\hat{z})\in \hat{{\cal A}}(\hat{z})$ for $\hat{z}\in \hat{S}$, where $\hat{\cal A}(\hat{z})$ is the set of feasible service rate in terms of batches, such that
\begin{align}
    &\text{minimize} \quad \overline{\lim}_{T\rightarrow\infty} \frac{1}{T} \mathbb{E}\left[ \int_0^T \kappa\sum_{k=1}^{2K} c_k(\hat{Z}_k(t)) dt+\kappa p\sum_{k=1}^{K} \hat{U}_k(T) \right]  \label{eqn:Objective_Hat}  \\
    &\text{subject to} \quad (\ref{eqn:StateSpaceConstraint_Hat})-(\ref{eqn:UpperPushingConstraint_Hat}), \label{eqn:Constraints_Hat}
\end{align} 
Note that we can drop $\kappa$ in (\ref{eqn:Objective_Hat}) since $\kappa$ does not affect the optimal solution. 

\paragraph{HJB and an equivalent SDE with scaled process.} Similar to Section \ref{sec:HJB}, we have the HJB equation with the scaled process as follows. 
\begin{align}
\label{eqn:HJB_Hat}
	\hat{\beta} &= {\cal L} \hat{V}(\hat{z})- \max_{\mu\in \hat{\mathcal{A}}(\hat{z})} \left\{\mu \cdot \nabla \hat{V}(\hat{z}) \right\}+\sum_{k=1}^{2K} c_k(\hat{z}_k), \quad \hat{z}\in \hat{S}, 
\end{align} subject to the boundary conditions 
\begin{align}
    \frac{\partial \hat{V}(\hat{z})}{\partial \hat{z}_k} & = 0 \quad \text{ if } \hat{z}_k=0 \quad (k=1,\ldots, 2K), \label{eqn:HJB_Hat_Boundary1} \\
    \frac{\partial \hat{V}(\hat{z})}{\partial \hat{z}_k} &= p \quad \text{ if } \hat{z}_k=\hat{b}_k \quad (k=1,\ldots, K), \label{eqn:HJB_Hat_Boundary2} \\
    \lim_{\hat{z}_k\rightarrow\infty} \frac{\partial \hat{V}(\hat{z})}{\partial \hat{z}_k} &= \frac{c_2}{\gamma_k} \quad (k=K+1,\ldots,2K),   \label{eqn:HJB_Hat_Boundary3}
\end{align}
where the generator ${\cal L}$ is given as follows: 
\begin{align}
{\cal L} \hat{V}(\hat{z}) &= \frac{1}{2} \sum_{k=1}^{2K} \hat{\sigma}_{k}^2 \frac{\partial^2 \hat{V}(\hat{z})}{\partial \hat{z}_k^2} + \sum_{k=1}^K \left(\hat{\lambda}_k-\gamma_k \hat{z}_k\right) \frac{\partial \hat{V}(\hat{z})}{\partial \hat{z}_k}. 
\end{align}

With the reference policy $\tilde{\hat{\mu}}$ for the scaled process, the corresponding reference process, denoted by $(\tilde{\hat{Z}}, \tilde{\hat{L}}, \tilde{\hat{U}})$, satisfy the following: For $t\geq 0$,
\begin{alignat*}{2}
    &\tilde{\hat{Z}}_k(t) = \tilde{\hat{Z}}_k(0)+\left(\hat{\lambda}_k-\tilde{\hat{\mu}}_k\right) t+\hat{B}_k(t)-\int_{0}^t \gamma_k \tilde{\hat{Z}}_k(s)ds +\tilde{\hat{L}}_k(t)-\tilde{\hat{U}}_k(t),&& \quad k=1,\ldots, K \\
    &\tilde{\hat{Z}}_k(t) = \tilde{\hat{Z}}_k(0)+\left(\hat{\lambda}_k-\tilde{\hat{\mu}}_k\right) t+\hat{B}_k(t)-\int_{0}^t \gamma_k \tilde{\hat{Z}}_k(s)ds+\tilde{\hat{L}}_k(t), && \quad k=K+1,\ldots, 2K,   \\
    &\int_0^t \tilde{\hat{Z}}_k(s) d\tilde{\hat{L}}_k(s) = 0, \quad k=1,\ldots, 2K,  \\
    &\int_0^t (\hat{b}_k - \tilde{\hat{Z}}_k(s) ) d\tilde{\hat{U}}_k (s) = 0, \quad k=1,\ldots, K, \\
    & \tilde{\hat{Z}}(t) \in \hat{S},   \\
    & \tilde{\hat{L}} \text{ is non-decreasing, non-anticipating with } \tilde{\hat{L}}(0)=0, \\
    & \tilde{\hat{U}} \text{ is non-decreasing, non-anticipating with } \tilde{\hat{U}}(0)=0.
\end{alignat*}
We further let 
\begin{align}
    \label{eqn:AuxiliaryFunction_Hat}
    \hat{F}(z, v) &= \max_{\mu\in \hat{\cal{A}}(z)} \left\{ \mu \cdot v \right\}-\sum_{k=1}^{2K} c_k(z_k), \quad z\in \hat{S}, \; v\in \mathbb{R}_+^{2K}. 
\end{align}

Similar to Propositions (\ref{prop:EquivalentSDE}) and (\ref{prop:EquivalentSDE2}), we have the following two propositions that show the equivalent SDE for the scaled process. 
\begin{proposition}
\label{prop:EquivalentSDE_Scaled}
    If $\hat{V}(\cdot)$ and $\hat{\beta}$ solve the HJB equations (\ref{eqn:HJB_Hat})-(\ref{eqn:HJB_Hat_Boundary3}), then the following holds almost surly for $T>0$,
    \begin{align}
    \begin{split}
        \hat{V}(\tilde{\hat{Z}}(T))-\hat{V}(\tilde{\hat{Z}}(0)) &= \int_0^T \nabla \hat{V} (\tilde{\hat{Z}}(t))  d\hat{B}(t) + \hat{\beta} T - p \sum_{k=1}^K \tilde{U}_k(T) \\
        & \quad -\int_0^T \tilde{\hat{\mu}}\cdot \nabla \hat{V}(\tilde{\hat{Z}}(t)) dt +\int_0^T \hat{F} (\tilde{\hat{Z}}(t), \nabla \hat{V}(\tilde{\hat{Z}}(t)) ) dt.  \label{eqn:Identity_Hat}
    \end{split}
    \end{align}
\end{proposition}

\begin{proposition}
    Suppose that $\hat{V}:\hat{S}\rightarrow \mathbb{R}$ is a ${\cal C}^2$ function, $\hat{G}:\hat{S}\rightarrow \mathbb{R}$ is continuous and $\hat{V}, \nabla \hat{V}, \hat{G}$ all have polynomial growth. Also assume that the following identity holds almost surely for some fixed $T>0$, a scalar $\hat{\beta}$ and every $\hat{Z}(0)=\hat{z}\in \hat{S}$: 
    \begin{align}
        \begin{split}
        \hat{V}(\tilde{\hat{Z}}(T))-\hat{V}(\tilde{\hat{Z}}(0)) &= \int_0^T \hat{G}( \tilde{\hat{Z}}(t)) d\hat{B}(t)+\hat{\beta} T- p \sum_{k=1}^{K} \hat{U}_k(T) \\
        &\quad -\int_0^T  \tilde{\hat{\mu}}\cdot \hat{G}(\tilde{\hat{Z}}(t)) dt +\int_0^T \hat{F}(\tilde{\hat{Z}}(t), \hat{G}(\tilde{\hat{Z}}(t))) dt. \label{eqn:Proposition2_Hat}
        \end{split}
    \end{align}
    Then, $\hat{G}(\cdot)=\nabla \hat{V}(\cdot)$ and $(\hat{V}, \hat{\beta})$ solve the HJB equations (\ref{eqn:HJB_Hat})-(\ref{eqn:HJB_Hat_Boundary3}). 
\end{proposition}

Note that we have the trained neural network model $\hat{H}(\hat{z}, v)$ that aims to predict the prize collected for the budgeted prize-collecting VRP. Its input contains two pieces, namely the scaled number of pending orders $\hat{z}$ and the prize of serving one order $v$. Its output comes from the budgeted prize-collecting VRP that considers $\kappa \hat{z}$ orders and prize vector $v$, and reports the total prize collected in this budgeted prize-collecting VRP. 
We let $K(\hat{z}, \hat{v})=\max_{\mu\in\hat{\cal A}(\hat{z})}\{\mu\cdot \hat{v}\}$ be the scaled version of $\hat{H}(\hat{z}, v)$.
Therefore, we have \[ K(\hat{z}, \hat{v}) = \hat{H}(\hat{z}, \hat{v}/\kappa). \] That means, we have \[\hat{F}(z,v)=K(z,v)-\sum_{k=1}^{2K}c_k(z_k), \quad z_k\in \hat{S}, v\in \mathbb{R}_+^{2K}. \]

The loss function with the scaled process is 
\begin{align*}
	\hat{\ell} (w_1, w_2) & = \Var \Bigg(\hat{V}^{w_1} (\tilde{\hat{Z}} (T))- \hat{V}^{w_1} (\tilde{\hat{Z}}(0)) - \int_{0}^T \hat{G}^{w_2} (\tilde{\hat{Z}} (t)) d\tilde{B}(t) +p\sum_{k=1}^K \tilde{\hat{U}}_k(T) \\
    &\hspace{1cm}\quad +\int_0^T \tilde{\hat{\mu}}\cdot \hat{G}^{w_2}(\tilde{\hat{Z}}(t)) dt -\int_{0}^T  \hat{F}(\tilde{\hat{Z}}(t), \hat{G}^{w_2} (\tilde{\hat{Z}}(t)) )dt \Bigg). 
\end{align*}

\subsection{Proposed Policy with the Scaled Process}
\label{subapp:ScaledProposedPolicy}

Recall that the proposed policy utilizes the trained gradient function $G$ to assign prizes for daily pending orders (see Section \ref{sec:BenchmarkPolicies}). In our simulation study, we obtain the trained gradient function $\hat{G}$ based on the scaled process as discussed in Appendix \ref{subapp:DataPreprocessing}. Using $\hat{G}$, we then calculate the prizes following these adjustments: 
\begin{enumerate}
    \item For each day, we compute the number of pending orders $z_k$ for each class $k=1,\ldots,2K$. We then scale the number of pending orders into batches as $z_k/\kappa$. 
    \item 
    %We input $z_k/\kappa$ into $\hat{G}$ and obtain the prize of serving one batch of orders in class $k$. 
    Recall that the unscaled version assigns a prize as $\pi_{{\cal I}^{(k)}}(i)=G_k^{w_2}(z_k-i+1)$, where ${\cal I}^{(k)}$ represents the indices of the pending orders in class $k$. With the scaled output prize from the trained function $\hat{G}^{w_2}$, we scale back the prize and assign the prize for serving order $i$ in class $k$ as follows: 
    \begin{align}
    \label{eqn:ScaledPrize}
        \pi_{{\cal I}^{(k)}}(i)= \frac{(z_k-i+1)/\kappa}{k}.
    \end{align}
\end{enumerate}
Note that the denominator in Equation (\ref{eqn:ScaledPrize}) can be omitted as we only aim to maintain the relative magnitude of prizes across pending orders.

\subsection{Reference Policy and Behavior Policy}
\label{subapp:ReferencePolicy}

As discussed in Sections \ref{sec:HJB} and \ref{sec:ComputationalMethod}, any reference policy can be used for computation. However, selecting a more appropriate reference policy $\tilde{\mu}$ can facilitate more effective training. Additionally, we may consider adjust the behavior of the generated sample path by changing the variance of the diffusion coefficient. 
In general, we aim to select a reference policy and a behavior policy that guide the generated sample paths to closely resemble those under the optimal policy.  

\paragraph{Reference policy.} 

We apply different reference policies for classes depending on whether they have a deadline.
\begin{itemize}
    \item Classes with a deadline:
    We consider the reference policy so that the generated sample paths remain close to values $b_k$, $b_k/2$, and 0. One may easily verify these values with the evolution of the system state discussed in Section \ref{sec:HJB}. 
\begin{align*}
    \tilde{\mu}_k &= \begin{cases}
        \lambda_k-\gamma_k b_k, & \text{with probability 1/3},  \\
        \lambda_k-\frac{1}{2}\gamma_k b_k, & \text{with probability 1/3},  \\
        \lambda_k, & \text{with probability 1/3}.
    \end{cases}
\end{align*}
    \item Classes without a deadline: 
    When a class $k$ does not have a deadline (i.e., $K+1\leq k\leq 2K$), there is no corresponding $b_k$. To also choose $\tilde{\mu}$ so that the generated sample paths are at a high, medium, and low level, we set $\tilde{\mu}_k=0$ (no service) at high levels, $\tilde{\mu}_k=\lambda_k$ (service rate equals arrival rate) at low levels. To decide the medium level, we choose $\tilde{\mu}_k=\lambda_k/2$ as the average of 0 and $\lambda_k$. 
\begin{align*}
    \tilde{\mu}_k &= \begin{cases}
        0, & \text{with probability 1/3}  \\
        \frac{1}{2}\lambda_k, & \text{with probability 1/3}  \\
        \lambda_k, & \text{with probability 1/3}.
    \end{cases}
\end{align*}
\end{itemize}
For the scaled process used in our simulation study, the reference policy $\tilde{\hat{\mu}}$ follows the same structure as $\tilde{\mu}_k$, with $\lambda_k$ and $b_k$ replaced by $\hat{\lambda}_k$ and $\hat{b}_k$, respectively. 

\paragraph{Behavior policy.} For the behavior policy, we consider the variance of the driftless Brownian motion as $4\sigma^2_k$ for $k=1,\ldots,12$. For the scaled process, we change the variance accordingly. 

\subsection{Shape Constraints}
\label{subapp:ShapeConstraint}

To improve the training process, we introduce penalty terms to enforce specific properties (i.e., shape constraints) on the neural network approximation of the gradient function $\hat{G}^{w_2}(z)$, specifically the boundary condition behavior. We discuss how we incorporate these penalty terms in this section. 

As shown in the HJB equation (\ref{eqn:HJB_Hat})-(\ref{eqn:HJB_Hat_Boundary3}), the neural network approximation $\hat{G}^{w_2}(\hat{z})$ is expected to satisfy the boundary conditions: (1) for $k=1,\ldots, K$, $\hat{G}^{w_2}(\hat{z})$ should be 0 when $\hat{z}_k=0$ and equal to $p$ when $\hat{z}_k=\hat{b}_k$; (2) for $k=K+1,\ldots,2K$, $\hat{G}^{w_2}(\hat{z})$ should approach $h_2/\gamma_k$ as $\hat{z}_k$ approaches infinity. To enforce these behaviors, we introduce penalty terms for both the left and right endpoint constraints.

\paragraph{Left endpoint constraint.} The left endpoint constraint ensures conditions under Equation (\ref{eqn:HJB_Hat_Boundary1}). We choose a small value $\epsilon>0$ and define the following penalty term:
\begin{align*}
    \E \left[\left\{ \sum_{k=1, |\hat{z}_k|\leq \epsilon}^{2K} \left| \hat{G}_k^{w_2}(\hat{z}) \right|  \right\}^2 \right].
\end{align*}
This term calculates the expected magnitude of violation from the boundary condition specified in Equation (\ref{eqn:HJB_Hat_Boundary1}). 

\paragraph{Right endpoint constraint.} To enforce the boundary properties described in Equations (\ref{eqn:HJB_Hat_Boundary2}) and  (\ref{eqn:HJB_Hat_Boundary3}), we incorporate a right endpoint constraint with penalty terms, defined differently depending on whether $1\leq k\leq K$ or $K+1\leq k\leq 2K$. 
\begin{itemize}
    \item For $1\leq k\leq K$: Using a similar approach as in the left endpoint constraint, we define the penalty term as
    \begin{align*}
        \E \left[\left\{\sum_{k=1, |\hat{z}_k-\hat{b}_k|\leq \epsilon}^{K} \left|\hat{G}_k^{w_2} (\hat{z})-p \right| \right\}^2 \right].
    \end{align*}
    This term calculates the expected deviation from the boundary condition in Equation (\ref{eqn:HJB_Hat_Boundary2}).
    \item For $K+1\leq k\leq 2K$: Let $z_k^\infty$ be a significantly large value (we discuss specific values used in simulations later). For each $k=1,\ldots,2K$, we construct a modified input $\dot{z}^{(k)}$ by setting $\dot{z}_{k}^{(k)}=\hat{z}_{k}^\infty/\kappa$ and keeping all other components the same as $\hat{z}$ (i.e., \(\dot{z}_{k'} = \hat{z}_{k'}\) for \(k' \neq k\)). This modification replaces the $k$th entry of $\hat{z}$ to approach infinity, while preserving the values of other entries.
    The penalty term is defined as 
    \begin{align*}
        \E \left[ \left\{ \sum_{k=K+1}^{2K} \left|\hat{G}_k^{w_2} (\dot{z}^{(k)})-\frac{c_2}{\gamma_k} \right| \right\}^2 \right]. 
    \end{align*}
    This penalty term calculates the expected violation to the boundary condition of Equation (\ref{eqn:HJB_Hat_Boundary3}).
\end{itemize}

Combining these penalty terms with the loss function, we modify the loss function as 
\begin{align*}
	\hat{\ell} (w_1, w_2) & = \widehat{\Var} \Bigg(\hat{V}^{w_1} (\tilde{\hat{Z}}^{(i)} (T))- \hat{V}^{w_1} (\tilde{\hat{Z}}^{(i)}(0)) - \sum_{j=0}^{N-1} \hat{G}^{w_2} (\tilde{\hat{Z}}^{(i)} (j\Delta t)) \cdot\delta_j^{(i)} +p\sum_{k=1}^K \tilde{\hat{U}}^{(i)}_k(T) \\
    &\hspace{1cm}\quad +\sum_{j=0}^{N-1} \tilde{\hat{\mu}}\cdot \hat{G}^{w_2}(\tilde{\hat{Z}}^{(i)}(j \Delta t)) -\sum_{j=0}^{N-1}  \hat{F}\left(\tilde{\hat{Z}}^{(i)}(j \Delta t), \hat{G}^{w_2} (\tilde{\hat{Z}}^{(i)}(j\Delta t)) \right) \Delta t \Bigg)+{\rm penalty}. 
\end{align*}
where 
\begin{align*}
    \text{penalty} &= \Lambda_1\cdot \E \left[\left\{ \sum_{k=1, |z_k|\leq \epsilon}^{2K} \left| \hat{G}_k^{w_2}(z) \right|  \right\}^2 \right]+\Lambda_2\cdot \E \left[\left\{\sum_{k=1, |\hat{z}_k-\hat{b}_k|\leq \epsilon}^{K} \left|\hat{G}_k^{w_2} (\hat{z})-p \right| \right\}^2 \right] \\
    &\quad + \Lambda_3\cdot \E \left[ \left\{ \sum_{k=K+1}^{2K} \left|\hat{G}_k^{w_2} (\hat{z}^{(k)})-\frac{c_2}{\gamma_k} \right| \right\}^2 \right],
\end{align*}
with $\Lambda_1, \Lambda_2, \Lambda_3$ as tunable constants associated with each penalty term, selected through tuning for optimal performance.

\paragraph{Implementation details of the shape constraints.} 

In our implementation, we set $\epsilon=0.001$.

To determine appropriate values for $z_k^\infty$ for $k=K+1,\ldots,2K$, we select constants large enough to effectively simulate ``infinity''. Specifically, we multiply the maximum values $n_k^{\rm max}$ (shown in Table \ref{tab:MaxNumPendingByClass}) by a factor of 2. The resulting values of $z_k^\infty$ used in our simulations are shown in Table \ref{tab:ZInfinity}.
\begin{table}[h!]
	\centering
	\begin{tabular}{ c|c||c|c }
		\toprule
        $k$ & $z_k^\infty$ & $k$ & $z_k^\infty$ \\  \midrule
        13 & 599 & 19 & 150 \\
        14 & 253 & 20 & 238 \\
        15 & 204 & 21 & 207 \\
        16 & 323 & 22 & 609 \\
        17 & 349 & 23 & 178 \\
        18 & 191 & 24 & 522 \\
		\bottomrule
	\end{tabular}
	\caption{Chosen values of $z_k^\infty$ for $k=K+1,\ldots,2K$ through the training.}
	\label{tab:ZInfinity}
\end{table}

We set $\Lambda_1=1, \Lambda_2=10$, and $\Lambda_3=h$ (i.e., the ratio of $c_2$ and $c_1$). In general, we observe that the penalty term associated with $\Lambda_3$ tends to increase as $c_2$ increases. To account for this, we enlarge the value of $\Lambda_3$ for larger $c_2$.

\subsection{Neural Network Architecture}
\label{subapp:NNArchitecture}

Table \ref{tab:CurrentNeuralNetworkConfiguration} shows the neural network architecture that we use to approximate $\hat{V}^{w_1}(\cdot)$ and $\hat{G}^{w_2}(\cdot)$
\begin{table}[h!]
	\centering
	\begin{tabular}{ c|c }
		\toprule
		Hyper-parameters & Value \\   \midrule
		Number of hidden layers & 3  \\  \hline
		Number of neurons per layer & 100  \\  \hline
        $T$ & 220  \\   \hline
		$\Delta t$ & 1  \\   \hline
		Batch size & 256   \\   \hline 
		Number of iterations & 10,000  \\  \hline
		Learning rate & 0.005, up to the first 3000th iteration  \\
		& 0.001, from 3001th iteration to the 5000th iteration \\
        & 0.0001, after the 5000th iteration  \\
 		\bottomrule
	\end{tabular}
	\caption{Neural network configuration for approximating $\hat{V}^{w_1}(\cdot)$ and $\hat{G}^{w_2}(\cdot)$.}
	\label{tab:CurrentNeuralNetworkConfiguration}
\end{table}

\iffalse
\section{Additional Experimental Results}

\begin{figure}[h!]
\begin{subfigure}[h]{0.45\linewidth}
\includegraphics[width=\linewidth]{Draft/Figures/SimulationResults/MissDieAndCancellation_WithDie.pdf}
\caption{\% Cancellation (with deadline)}
\end{subfigure}
\hfill
\begin{subfigure}[h]{0.45\linewidth}
\includegraphics[width=\linewidth]{Draft/Figures/SimulationResults/MissDieAndCancelDelay_WithDie.pdf}
\caption{Age before canceled (with deadline)}
\end{subfigure}%
\caption{Trade-off curves for the \% Missing deadline against the \% Cancellation (with deadline) and the age of the orders (with deadline) before canceled for the proposed policy with different holding cost, the budgeted FCFS policy, and the threshold-based policy.}
\label{fig:MissDieAndCancellation_WithDie}
\end{figure}

\begin{figure}[h!]
    \centering
    \includegraphics[width=0.5\linewidth]{Draft/Figures/SimulationResults/WithAndWithoutDie_Cancellation.pdf}
    \caption{Relationship between the orders with and without deadlines in terms of the \% Cancellation.}
    \label{fig:WithAndWithoutDie_CancellationApp}
\end{figure}

\fi

\end{document}